\documentclass[3p,times]{amsart}

\usepackage{amssymb}

\usepackage{amsthm}

\usepackage{enumerate} % for enumerate [(i)]/[(a)]/[(1)]

\usepackage{amsmath}
\usepackage{amssymb}

\newtheorem{theorem}{Theorem}[section] 
\newtheorem{lemma}[theorem]{Lemma}   
\newtheorem{corollary}[theorem]{Corollary}
\newtheorem{proposition}[theorem]{Proposition}
\newtheorem{definition}[theorem]{Definition}

\newtheorem{main-theorem}[theorem]{Theorem}
\newtheorem{assumption}[theorem]{Assumption}
\newtheorem*{problem*}{Problem}

\theoremstyle{definition}

\newtheorem*{question*}{Question}

\newtheorem{example}[theorem]{Example}
\newtheorem{remark}[theorem]{Remark}

\renewcommand{\mod}{\operatorname{mod}}

\newcommand{\alg}{\operatorname{alg}}
\newcommand{\rad}{\operatorname{rad}}
\newcommand{\soc}{\operatorname{soc}}

\newcommand{\umod}{\operatorname{\underline{mod}}}

\newcommand{\Ext}{\operatorname{Ext}}
\newcommand{\Hom}{\operatorname{Hom}}
\newcommand{\Ker}{\operatorname{Ker}}
\newcommand{\T}{\operatorname{T}}
\newcommand{\Tr}{\operatorname{Tr}}
\newcommand{\GL}{\operatorname{GL}}
\newcommand{\op}{\operatorname{op}}

\renewcommand{\Im}{\operatorname{Im}}
\newcommand{\lcm}{\operatorname{lcm}}

\newcommand{\bA}{\mathbb{A}}

\newcommand{\bN}{\mathbb{N}}
\newcommand{\bP}{\mathbb{P}}

\newcommand{\bR}{\mathbb{R}}

\newcommand{\bZ}{\mathbb{Z}}

\newcommand{\cO}{\mathcal{O}}

\newcommand{\cB}{\mathcal{B}}

\newcommand{\La}{\Lambda}
\newcommand{\ba}{\bar{\alpha}}
\newcommand{\vf}{\varphi}
\newcommand{\ve}{\varepsilon}

\makeatletter
\newcommand{\vect}[1]{%
  \vbox{\m@th \ialign {##\crcr
  \vectfill\crcr\noalign{\kern-\p@ \nointerlineskip}
  $\hfil\displaystyle{#1}\hfil$\crcr}}}
\def\vectfill{%
  $\m@th\smash-\mkern-7mu%
  \cleaders\hbox{$\mkern-2mu\smash-\mkern-2mu$}\hfill
  \mkern-7mu\raisebox{-3.81pt}[\p@][\p@]{$\mathord\mathchar"017E$}$}

\newcommand{\amsvect}{%
  \mathpalette {\overarrow@\vectfill@}}
\def\vectfill@{\arrowfill@\relbar\relbar{\raisebox{-3.81pt}[\p@][\p@]{$\mathord\mathchar"017E$}}}

\newcommand{\amsvectb}{%
  \mathpalette {\overarrow@\vectfillb@}}
\newcommand{\vecbar}{%
  \scalebox{0.8}{$\relbar$}}
\def\vectfillb@{\arrowfill@\vecbar\vecbar{\raisebox{-4pt}[\p@][\p@]{$\mathord\mathchar"017E$}}}
\makeatother
\usepackage{etex} % need because of using both tikz and xymatrix below

\usepackage{etex} % needed because of using both tikz and xymatrix below
\usepackage{dsfont}         % for \mathds{1}
\usepackage[h]{esvect}

\usepackage[matrix,arrow,curve,frame,arc]{xy}

\usepackage{pgf}
\usepackage{tikz}

\usepackage{color}

\usetikzlibrary{arrows,shapes,automata,backgrounds,petri,calc}

\newcommand{\tikzAngleOfLine}{\tikz@AngleOfLine}
\def\tikz@AngleOfLine(#1)(#2)#3{%
\pgfmathanglebetweenpoints{%
\pgfpointanchor{#1}{center}}{%
\pgfpointanchor{#2}{center}}
\pgfmathsetmacro{#3}{\pgfmathresult}%
}

\begin{document}

\title{Weighted surface algebras: general version}

{\def\thefootnote{}
\footnote{The research was supported by the program
``Research in Pairs'' by the 
Mathematisches Forschungsinstitut Oberwolfach in 2018, and also by the
	Faculty of Mathematics and Computer Science of the Nicolaus Copernicus University in Toru\'{n}.}
}

\author[K. Edrmann]{Karin Erdmann}
\address[Karin Erdmann]{Mathematical Institute,
   Oxford University,
   ROQ, Oxford OX2 6GG,
   United Kingdom}
\email{erdmann@maths.ox.ac.uk}

\author[A. Skowro\'nski]{Andrzej Skowro\'nski}
\address[Andrzej Skowro\'nski]{Faculty of Mathematics and Computer Science,
   Nicolaus Copernicus University,
   Chopina~12/18,
   87-100 Toru\'n,
   Poland}
\email{skowron@mat.uni.torun.pl}

\begin{abstract}
We introduce  general weighted surface algebras 
of triangulated surfaces
with arbitrarily oriented triangles and
describe their basic properties.
In particular, we prove that all these algebras,
except the singular disc, triangle, tetrahedral 
and spherical algebras,
are symmetric tame periodic algebras of period $4$.

\bigskip

 

\end{abstract}

%\linenumbers

\maketitle

\begin{center}
\vspace*{-5mm}
\textit{Dedicated to Helmut Lenzing on the occasion of his 80th birthday}
\vspace*{5mm}
\end{center}

\section{Introduction and main results}\label{sec:intro}

We are interested in the representation theory of tame self-injective algebras.
In this paper,  all algebras are finite-dimensional associative, 
indecomposable as algebras, and basic, 
over an algebraically closed field $K$ of arbitrary characteristic.

Tame self-injective algebras of polynomial growth 
are currently well understood (see \cite{Sk1,Sk2}). 
For non-polynomial growth, much less is known. 
It would be interesting to describe the basic algebras 
of arbitrary tame self-injective algebras of non-polynomial growth. 
Our present project is a step in this direction.

In  the  modular representation theory of finite groups
representation-infinite tame blocks occur only
over fields of characteristic 2,  and their defect groups are dihedral,
semidihedral, or  (generalized) quaternion 2-groups.
In order to study such blocks,  algebras of dihedral, 
semidihedral and quaternion type were introduced
and investigated,  over algebraically closed fields of
arbitrary characteristic (see \cite{E3}). 
In particular, it was shown in \cite{ESk2,Ho}
that every algebra of quaternion type is a tame
periodic algebra of period $4$.

Recently cluster theory has led to  new directions. 
Inspired by this, we study in \cite{ESk-WSA} 
a  class of symmetric algebras defined in terms of surface
triangulations, which we call \emph{weighted surface algebras}. 
They  are tame and we show that they are periodic as algebras, 
of period 4 (with one exception, which we call
the singular tetrahedral algebra).
We observe that many algebras of quaternion type as described in \cite{E3} 
occur in this setting but the construction in \cite{ESk-WSA} only produces
algebras whose Gabriel quiver is 2-regular (that is, at each vertex, two
arrows start and two arrows end).

In this paper we extend and improve the results of 
\cite{ESk-WSA}. We generalize the previous definition slightly, and 
obtain a larger class of algebras. This new version also 
includes algebras whose Gabriel quiver is not 2-regular, in particular
we obtain now  
almost all algebras of quaternion type.
As well, we obtain the endomorphism algebras of cluster tilting
objects in the stable categories of maximal Cohen-Macaulay modules over
minimally elliptic
curve singularities, as  discussed in \cite{BIKR}.

An important further motivation for the generalisation is the study
of idempotent algebras. In \cite{ESk7} we show that
any Brauer graph algebra occurs as an idempotent algebra
of some weighted surface algebra. Analysing an arbitrary idempotent
algebra of a weighted surface algebra, we  discovered that 
it is natural to extend the original definition.
In a subsequent paper we will give a complete 
description of all idempotent algebras
of weighted surface algebras. 
%
%\bigskip

\medskip

The main result in this paper shows that, with four exceptions, 
a general weighted surface algebra is periodic as an algebra, of period 4.
The  exceptions are the singular tetrahedral algebra
which already occured in \cite{ESk-WSA}, and three others,  which we call
singular disc algebra, singular triangle algebra, and singular spherical algebra. 
%\bigskip

\medskip

Let $A$ be an algebra.
Given a module $M$ in  $\mod A$, its \emph{syzygy}
is defined to be the kernel $\Omega_A(M)$ of a minimal
projective cover of $M$ in $\mod A$.
The syzygy operator $\Omega_A$ is a very important tool
to construct modules in $\mod A$ and relate them.
For $A$ self-injective, it induces an equivalence
of the stable module category $\umod A$,
and its inverse is the shift of a triangulated structure
on $\umod A$ \cite{Ha1}.
A module $M$ in $\mod A$ is said to be \emph{periodic}
if $\Omega_A^n(M) \cong M$ for some $n \geq 1$, and if so
the minimal such $n$ is called the \emph{period} of $M$.
The action of $\Omega_A$ on $\mod A$ can effect
the algebra structure of $A$.
For example, if all simple modules in $\mod A$ are periodic,
then $A$ is a self-injective algebra.
Sometimes one can even recover the algebra $A$ and its
module category from the action of $\Omega_A$.
For example, the self-injective Nakayama algebras
are precisely the algebras $A$ for which $\Omega_A^2$ permutes
the isomorphism classes of simple modules in $\mod A$.
An algebra $A$ is defined to be \emph{periodic} if it is periodic
viewed as a module over the enveloping algebra
$A^e = A^{\op} \otimes_K A$, or equivalently,
as an $A$-$A$-bimodule.
It is known that if $A$ is a periodic algebra of period $n$ then
for any indecomposable non-projective module $M$
in $\mod A$  the syzygy $\Omega_A^n(M)$ is isomorphic to $M$.

Finding or possibly classifying periodic algebras is an important
problem, because of interesting connections with group theory,
topology, singularity theory, cluster algebras, cluster tilting theory
(we refer to \cite{ESk3,L4} and
the introduction of \cite{ESk-WSA} for some details).

\medskip

The following three theorems describe basic properties of the general
weighted surface algebras.

\begin{main-theorem}
\label{th:main1}
Let $\Lambda = \Lambda(S,\vec{T},m_{\bullet},c_{\bullet})$
be a weighted surface algebra over an algebraically
closed field $K$,
which is not isomorphic to a singular triangle or spherical algebra.
Then  $\Lambda$ is a symmetric algebra.
\end{main-theorem}

\begin{main-theorem}
\label{th:main2}
Let $\Lambda = \Lambda(S,\vec{T},m_{\bullet},c_{\bullet})$
be a weighted surface algebra over an algebraically
closed field $K$,
which is not isomorphic to a disc algebra,
triangle algebra,
tetrahedral algebra,
spherical algebra.
Then the following statements hold:
\begin{enumerate}[(i)]
 \item
  $\Lambda$ degenerates to the biserial
  weighted surface algebra
  $B(S,\vec{T},m_{\bullet},c_{\bullet})$.
 \item
  $\Lambda$ is a tame algebra of non-polynomial growth.
\end{enumerate}
\end{main-theorem}

\begin{main-theorem}
\label{th:main3}
Let $\Lambda = \Lambda(S,\vec{T},m_{\bullet},c_{\bullet})$
be a weighted surface algebra over an algebraically
closed field $K$.
Then the following statements are equivalent:
\begin{enumerate}[(i)]
 \item
  All simple modules in $\mod \Lambda$ are periodic of period $4$.
 \item
  $\Lambda$ is a periodic algebra of period $4$.
 \item
  $\Lambda$ is a weighted surface algebra other than  
  a singular disc, triangle,  tetrahedral or spherical algebra.
\end{enumerate}
\end{main-theorem}

This paper is organized as follows.
In Section~\ref{sec:context}
we introduce the algebras. 
This is  slightly more general
as needed for weighted surface algebras, 
in order to show how they fit into
a more general context of tame symmetric algebras; 
as well it will be needed for the study of idempotent algebras.
We review much as needed from \cite{ESk-WSA}, 
this is done in Section~\ref{sec:context},
and discuss the modifications of the definition needed.
Section~\ref{sec:exceptions}
introduces the algebras which play a special role:
the disc algebras, 
the tetrahedral algebras,
the triangle algebras,
and the spherical algebras.
In Section~\ref{sec:properties}
we prove some general results, in particular we show that
weighted surface algebras 
are symmetric (except for a few small cases which we identify).
Section~\ref{sec:per}
proves the periodicity result.
The final section proves tameness, and also classifies polynomial growth.
For general background on the relevant representation theory
we refer to the books
\cite{ASS,E3,SS,SY}.

\section{Weighted surface algebras, and the general context}
\label{sec:context}

Recall that a quiver is a quadruple $Q = (Q_0, Q_1, s, t)$ 
where $Q_0$ is a finite set of vertices, 
$Q_1$ is a finite set of arrows, and 
where $s, t$ are maps $Q_1\to Q_0$ associating 
to each arrow $\alpha\in Q_1$ its source $s(\alpha)$  
and its target $t(\alpha)$. 
We say that
$\alpha$ starts at $s(\alpha)$ and ends at $t(\alpha)$. 
We assume throughout that any quiver is connected. 

Denote by $KQ$ the path algebra of $Q$ over $K$. 
The underlying space has basis the set of all paths in $Q$. 
Let $R_Q$ be the ideal of $KQ$ generated by all paths of length $\geq 1$. 
For each vertex $i$, let $e_i$ be the path of length zero at $i$, then 
the $e_i$ are pairwise orthogonal, and their sum is the identity of $KQ$. 
We will consider algebras of the form $A=KQ/I$ where $I$ is an ideal of $KQ$ 
which contains $R_Q^m$ for some $m\geq 2$, so that
the algebra is finite-dimensional and basic.
The Gabriel quiver $Q_A$ of $A$ is then the full subquiver of $Q$ 
obtained from $Q$ by removing all arrows $\alpha$ which belong to the ideal $I$.

\medskip
The setting for weighted surface algebras and the algebras occuring 
in \cite{ESk6} and \cite{ESk7} (and also in future work) has unified description, 
which we
will now present.

A quiver $Q$ is \emph{$2$-regular} if for each vertex $i\in Q_0$ 
there are precisely two arrows starting
at $i$ and two arrows ending at $i$. 
All quivers we consider will be 2-regular. 
Such a quiver has an involution
on the arrows, $\alpha \mapsto \ba$,  
such that for each arrow $\alpha$, 
the arrow $\ba$ is the arrow $\neq \alpha$ such that $s(\alpha) = s(\ba)$.

 A \emph{biserial quiver} \  is a pair $(Q, f)$ where $Q$ 
is a  (finite) connected 2-regular quiver,  with at least two vertices, and where 
$f$ is a fixed 
permutation of the arrows such that  $t(\alpha) = s(f(\alpha))$ 
for each arrow $\alpha$. 
The permutation $f$ uniquely determines a permutation $g$ of the arrows, 
defined by $g(\alpha) := \overline{f(\alpha)}$ for any arrow $\alpha$. 
Let $(Q, f)$ be a biserial quiver. We say that $(Q, f)$ 
is a \emph{triangulation quiver}  if $f^3$ is the identity. 
That is, all cycles of $f$ have length $3$ or $1$.

In this paper we will focus on triangulation quivers.
As we have proved in \cite{ESk-WSA},
these are 
precisely the quivers $(Q(S, \vec{T}), f)$ constructed 
from a triangulation $T$ of a compact connected (real)
surface, with or without boundary, and where 
the orientation $\vec{T}$ in each triangle can be chosen arbitrarily. 
For details we refer to \cite{ESk-WSA}, we will not repeat this
since we will not use the geometric version in any essential way.

\bigskip

We fix an algebraically closed field $K$, and we introduce some notation. 
This will be used throughout. For each arrow $\alpha$, we fix 
\begin{align*} 
m_{\alpha}\in \mathbb{N}^* &&& \mbox{ a multiplicity, constant on $g$-cycles, and }\cr
c_{\alpha} \in K^*  &&&  \mbox{ a weight,  constant on $g$-cycles, and define }\cr
n_{\alpha}:=   &&&  \mbox{ the length of the $g$-cycle of $\alpha$, } \cr
B_{\alpha}:=  \alpha g(\alpha)\ldots g^{n_{\alpha}-1}(\alpha) &&&   \mbox{ the path  along the $g$-cycle of $\alpha$ 
                     of length $m_{\alpha}n_{\alpha}$} \cr
 A_{\alpha}:=  \alpha g(\alpha)\ldots g^{n_{\alpha}-2}(\alpha) &&&   \mbox{ the path  along the $g$-cycle of $\alpha$ 
                     of length $m_{\alpha}n_{\alpha}-1$.}
\end{align*}

\noindent 
For the algebras of the form $A= KQ/I$, 
we will fix relations such that: 
\begin{enumerate}[(1)]
 \item
   Each paths $\alpha f(\alpha)$ of length 2 
   occurs in some distinguished element in $I$. 
 \item
   We will  ensure that in  the algebra, 
   $c_{\alpha}B_{\alpha} = c_{\ba}B_{\ba}$, 
   and that the  elements $B_{\alpha}$  
   span the socle of $A$.
 \item
   We will ensure that $A$ has a basis consisting 
   of initial subwords of elements $B_{\alpha}$ and $B_{\ba}$. 
\end{enumerate}
%
%\bigskip
%\medskip

That is, the cycles of $f$ describe minimal relations, 
and the cycles of
$g$ describe a basis for the algebra. 
There are two types of distinguished  relations,
\begin{enumerate}
 \item[(Q)]
   $\alpha f(\alpha)  = c_{\ba}A_{\ba}$  \  in $\La$, 
   only when  $f^3(\alpha)=\alpha$ (`quaternion' relations);  
 \item[(B)]
   $\alpha f(\alpha)=0$ \ in $\La$ \ (`biserial' relations).
\end{enumerate}
In addition, one needs   zero relations so that (2) is satisfied.

\bigskip

This includes Brauer graph algebras: Take an algebra 
	$R=KQ/I$ where $(Q, f)$ is a  biserial quiver
and where $I$ is generated by biserial relations, 
for all arrows $\alpha$, together with  relations 
$B_{\alpha} = B_{\ba}$, for all arrows $\alpha$, 
taking as the  weight function has $c_{\alpha}=1$ for all $\alpha$. 
For details we refer to \cite{ESk7}. 
This includes Brauer tree algebra, and motivated by this 
we think of the cycles of $f$ as `Green walks'. 
	
	 In \cite{ESk-WSA} and \cite{ESk6} we have studied biserial weighted surface algebras,
these are the Brauer graph algebras where in addition $f^3=1$, that is, 
	$(Q, f)$ is a triangulation quiver. These occur for blocks
	with dihedral defect groups. In this case, the Green walks
	are in bijection with tubes of rank 3 in the stable Auslander-Reiten quiver. In fact, this suggests that the condition $f^3=1$ should play a
	special role.
	
 On the other extreme, if all distinguished relations are quaternion 
relations, we get weighted surface algebras, which we will study in detail
in this paper.

To deal with tameness, we use special biserial algebras, and we only
need those which are symmetric, for the general definition we refer to the literature. It is known that special biserial symmetric
algebras are precisely the Brauer graph algebras as described above, for a 
detailed discussion see \cite{ESk7}.
We have
the following (proved in this generality in \cite{WW}, see also \cite{BR, DS}
for alternative proofs).

\begin{proposition}
\label{prop:2.1}
Every special biserial algebra is tame.
\end{proposition}

For a positive integer $d$, we denote by $\alg_d(K)$ the affine
variety of associative $K$-algebra structures with identity on
the affine space $K^d$.
Then the general linear group $\GL_d(K)$ acts on $\alg_d(K)$
by transport of the structures, and the $\GL_d(K)$-orbits in
$\alg_d(K)$ correspond to the isomorphism classes of $d$-dimensional
algebras (see \cite{Kr} for details). We identify a $d$-dimensional
algebra $A$ with the point of $\alg_d(K)$ corresponding to it.
For two $d$-dimensional algebras $A$ and $B$, we say that $B$
is a \emph{degeneration} of $A$ ($A$ is a \emph{deformation} of $B$)
if $B$ belongs to the closure of the $\GL_d(K)$-orbit
of $A$ in the Zariski topology of $\alg_d(K)$.

Geiss' Theorem \cite{Ge} shows that if $A$ and $B$ are two
$d$-dimensional algebras, $A$ degenerates to $B$ and $B$ is a tame
algebra, then $A$ is also a tame algebra (see also \cite{CB2}).
We will apply this theorem in the following special situation.

\begin{proposition}
\label{prop:2.2}
Let $d$ be a positive integer, and $A(t)$, $t \in K$,
be an algebraic family in $\alg_d(K)$ such that $A(t) \cong A(1)$
for all $t \in K \setminus \{0\}$.
Then $A(1)$ degenerates to $A(0)$.
In particular, if $A(0)$ is tame, then $A(1)$ is tame.
\end{proposition}

A family of algebras $A(t)$, $t \in K$, in $\alg_d(K)$
is said to be \emph{algebraic} if the induced map
$A(-) : K \to \alg_d(K)$ is a regular map of affine varieties.

An important combinatorial and homological invariant
of the module category $\mod A$ of an algebra $A$
is its Auslander-Reiten quiver $\Gamma_A$.
Recall that $\Gamma_A$ is the translation quiver whose
vertices are the isomorphism classes of indecomposable
modules in $\mod A$, the arrows correspond
to irreducible homomorphisms, and the translation
is the Auslander-Reiten translation $\tau_A = D \Tr$.
For $A$ self-injective, we denote by $\Gamma_A^s$
the stable Auslander-Reiten quiver of $A$, obtained
from $\Gamma_A$ by removing the isomorphism classes
of projective modules and the arrows attached to them.
By a stable tube we mean a translation quiver $\Gamma$
of the form $\mathbb{Z} \mathbb{A}_{\infty}/(\tau^r)$,
for some $r \geq 1$, and we call $r$ the rank of $\Gamma$.
We note that, for a symmetric algebra $A$, we have
$\tau_A = \Omega_A^2$ (see \cite[Corollary~IV.8.6]{SY}).
In particular, we have the following equivalence.

\begin{proposition}
\label{prop:2.3}
Let $A$ be an indecomposable, representation-infinite
symmetric algebra.
The following statements are equivalent:
\begin{enumerate}[(i)]
 \item
  $\Gamma_A^s$ consists of stable tubes.
 \item
  All indecomposable non-projective modules in $\mod A$
  are periodic.
\end{enumerate}
\end{proposition}

Therefore, we conclude that, if $A$ is an indecomposable,
representation-infinite, symmetric, periodic algebra (of period $4$)
then $\Gamma_A^s$ consists of stable tubes (of ranks $1$ and $2$).
We also note that, if $A$ is a representation-infinite special
biserial symmetric algebra, then $\Gamma_A^s$ admits
an acyclic component (see \cite{ESk1}),
and consequently $A$ is not a periodic algebra.

Let $A$ be an algebra over $K$ and $\sigma$ a $K$-algebra
automorphism of $A$. Then for any $A$-$A$-bimodule $M$
we denote by ${}_1M_{\sigma}$ the $A$-$A$-bimodule with
the underlying $K$-vector space $M$ and action defined
as $a m b = a m \sigma(b)$ for all $a, b \in A$ and $m \in M$.

The following has been proved in \cite[Theorem~1.4]{GSS}.

\begin{theorem}
\label{th:2.4}
Let $A$ be an algebra over $K$ and $d$ a positive integer.
Then the following statements are equivalent:
\begin{enumerate}[(i)]
 \item
  $\Omega_A^d(S) \cong S$ in $\mod A$ for every simple
  module $S$ in $\mod A$.
 \item
  $\Omega_{A^e}^d(S) \cong {}_1A_{\sigma}$ in $\mod A^e$ for
  some $K$-algebra automorphism $\sigma$ of $A$ such that
  $\sigma(e) A \cong e A$ for any primitive idempotent
  $e$ of $A$.
\end{enumerate}
Moreover, if $A$ satisfies these conditions, then $A$ is self-injective.
\end{theorem}

The \emph{Cartan matrix} $C_A$ of an algebra $A$ is the matrix
$(\dim_K \Hom_A(P_i,P_j))_{1 \leq i,j \leq n}$
for a complete family $P_1,\dots,P_n$
of a pairwise non-isomorphic indecomposable
projective modules in $\mod A$.
The following main result from \cite{E0}
shows why the original class of algebras of quaternion type is very
restricted compared with the algebras which we will study in this paper.

\begin{theorem}
\label{th:2.5}
Let $A$ be an indecomposable, representation-infinite tame
symmetric algebra
with non-singular Cartan matrix
such that every non-projective
indecomposable module in $\mod A$
is periodic of period dividing $4$.
Then $\mod A$ has at most three
pairwise non-isomorphic simple modules.
\end{theorem}

In \cite{ESk-WSA} we define a weighted surface algebra, where the quiver
is constructed in terms of a triangulation $T$ of a surface $S$, 
with arbitrarily oriented $\vec{T}$ triangles, 
and such a quiver is denoted by $Q(S, \vec{T})$. 
Such a quiver is a triangulation quiver $(Q, f)$ 
as we have defined above. Moreover, it was proved that triangulation quivers
are the same as quivers of the form $Q(S, \vec{T})$. 
We have also at that stage
distinguished between weighted surface algebras (which use $Q(S, \vec{T})$)
and weighted triangulation algebra (which use $(Q, f)$).

In the present
paper we will almost entirely use triangulation quivers, but we will refer
to weighted surface algebras for the algebras constructed.
We will now give the general definition, and we use the notation which
we have introduced earlier.

Roughly speaking, the modification of the definition consists
of 
\begin{enumerate}[(i)]
 \item
	allowing quivers with $\geq 2$ vertices 
	(previously we excluded the case of two vertices),
 \item
	allowing $m_{\alpha}n_{\alpha}\geq 2$ 
	(previously we assumed $m_{\alpha}n_{\alpha}\geq 3$).
\end{enumerate}

We require socle conditions as described in part (3) of the notation, 
as well since  $m_{\alpha}n_{\alpha}=2$
we have to  modify  the zero relations, and exclude
a few degenerate cases. 

\bigskip

\begin{definition}\label{def:virtual} We say that an arrow $\alpha$ of
	$Q$ is virtual if $m_{\alpha}n_{\alpha}=2$. Note that
	this condition is preserved under the permutation $g$, and that
	virtual arrows form $g$-orbits of sizes 1 or 2. 
\end{definition}

\medskip

\begin{assumption}\label{ass} 
For the general 
weighted surface algebra
we assume that the following conditions are
satisfied:
\begin{enumerate}[(1)]
 \item
	$m_{\alpha}n_{\alpha}\geq 2$ for all arrows $\alpha$,
 \item
	$m_{\alpha}n_{\alpha}\geq 3$ for all arrows $\alpha$ such
	that $\ba$ is virtual and  $\ba$ is not a loop,
 \item
	$m_{\alpha}n_{\alpha}\geq 4$ for all arrows $\alpha$ 
	such that $\ba$ is virtual and  $\ba$ is a loop.
\end{enumerate}
\normalfont
Condition (1) is a general assumption, and (2) and (3) are needed to eliminate
	two small algebras, see below. In particular we exclude the
	possibility that
	both arrows starting at a vertex are virtual, and also
	that both arrows ending at a vertex are virtual.
\end{assumption}

The definition of a weighted surface algebra
is now as follows.

\begin{definition} \label{def:2.8}
The algebra $\La = \La(Q, f, m_{\bullet}, c_{\bullet}) = KQ/I$ 
is a weighted surface algebra if
$(Q, f)$ is a triangulation quiver, with $|Q_0| \geq 2$, 
and $I= I(Q, f, m_{\bullet}, c_{\bullet})$ is the ideal of 
$KQ$ generated by:
\begin{enumerate}[(1)]
 \item
	$\alpha f(\alpha) - c_{\ba}A_{\ba}$ for all arrows
	$\alpha$ of $Q$,
 \item
	$\alpha f(\alpha) g(f(\alpha)$ \ for all arrows $\alpha$ of $Q$
	such that $f^2(\alpha)$ is not virtual,
 \item
	$\alpha g(\alpha)f(g(\alpha))$ for all arrows $\alpha$ of $Q$
	such that $f(\alpha)$ is not virtual.
\end{enumerate}
\end{definition}

Note that the ideal is not admissible in general. 
Namely if an arrow $\bar{\alpha}$ (say) is virtual then by (1) 
it lies in the square of the radical. In 
fact, we can see from the relations  that the Gabriel quiver $Q_{\La}$ of
$\La$ is obtained from $Q$ by removing
all virtual arrows. 

%\bigskip
\medskip

As long as we do not
have any special conditions on 
other scalars, we can assume that for a virtual arrow $\alpha$, the 
weight $c_{\alpha}$ is equal to $1$, namely we may replace $\alpha$ (and
$g(\alpha)$) by $c_{\alpha}^{-1}\alpha$ (and $c_{\alpha}^{-1}g(\alpha)$), 
the $g$-orbit of $\alpha$ has length two, and this scalar does not occur
anywhere else. In the first part of this paper we will keep $c_{\alpha}$ since
it will clarify proofs. On the other hand, in the last part we will take
$c_{\alpha}=1$ for a virtual arrow $\alpha$ 
since it will simplify the formulae.

%\bigskip
\medskip

We recall some elementary consequences of the definition, which we will use
freely throughout.

\begin{lemma}
%\label{lem:5.3}
\label{lem:2.9}
Let $\alpha$ be an arrow in $Q$.
We have in $\Lambda$ the identities:
\begin{enumerate}[(i)]
 \item
  $f^2(\alpha) = g^{n_{\bar{\alpha}} - 1}(\bar{\alpha})$ 
  so that $g(f^2(\alpha) = \ba$.
 \item
  $A_{\bar{\alpha}}  f^2(\alpha) = B_{\bar{\alpha}}$.
 \item
  $\alpha A_{g(\alpha)}  = B_{{\alpha}}$.
 \item
  $c_{{\alpha}} B_{{\alpha}}
   = \alpha  f(\alpha) f^2(\alpha)
   = \bar{\alpha}  f(\bar{\alpha}) f^2(\bar{\alpha})
   = c_{\bar{\alpha}} B_{\bar{\alpha}}$.
 \item
 $A'_{\alpha} f^2(\bar{\alpha}) = A_{g(\alpha)}$.
\end{enumerate}
\end{lemma}

This is the same as Lemma 5.3 in \cite{ESk-WSA}.

%\bigskip

\begin{definition} \label{def:2.10}
The algebra $B = B(Q, f, m_{\bullet}, c_{\bullet}) = KQ/J$ 
is a biserial weighted triangulation algebra if
$(Q, f)$ is a triangulation quiver, with $|Q_0| \geq 2$, 
and $J= J(Q, f, m_{\bullet}, c_{\bullet})$ is the ideal of 
$KQ$ generated by:
\begin{enumerate}[(1)]
 \item
	$\alpha f(\alpha)$ for all arrows
	$\alpha$ of $Q$,
 \item
	$c_{\alpha}B_{\alpha}  - c_{\ba}B_{\ba}$ for all arrows
	$\alpha$ of $Q$.
\end{enumerate}
\end{definition}

We note that 
$B(Q, f, m_{\bullet}, c_{\bullet}) \cong B(Q, f, m_{\bullet}, \mathds{1})$,
where $\mathds{1}$ is the constant parameter function
taking only value $1$.
We have the following consequence of
\cite[Proposition~5.2]{ESk6}
(see also \cite[Proposition~2.3]{ESk7}).

\begin{proposition}
\label{prop:2.11}
Let $(Q,f)$ be a triangulation quiver,
$m_{\bullet}$ and $c_{\bullet}$
weight and parameter functions of $(Q,f)$,
and
$B = B(Q, f, m_{\bullet}, c_{\bullet})$.
The following statements hold:
\begin{enumerate}[(i)]
 \item
  $B$ is finite-dimensional with 
  $\dim_K B = \sum_{\cO \in \cO(g)}  m_{\cO} n_{\cO}^2$.
 \item
  $B$ is a symmetric sperical biserial algebra.
\end{enumerate}
In particular, $B$ is a tame algebra.
\end{proposition}

Let $T$ be a triangulation of a surface $S$,
$\vec{T}$ an orientation of triangles in $T$,
$(Q(S,\vec{T}),f)$ the associated triangulation quiver,
and
$m_{\bullet}$ and $c_{\bullet}$
weight and parameter functions of $(Q(S,\vec{T}),f)$.
Then 
$\Lambda(S, \vec{T}, m_{\bullet}, c_{\bullet}) 
  = \Lambda(Q(S,\vec{T}), f, m_{\bullet}, c_{\bullet})$ 
is said to be a \emph{weighted surface algebra},
and
$B(S, \vec{T}, m_{\bullet}, c_{\bullet}) 
  = B(Q(S,\vec{T}), f, m_{\bullet}, c_{\bullet})$ 
a \emph{biserial weighted surface algebra}.

\section{Exceptional weighted surface algebras}\label{sec:exceptions}

%\noindent
In this section we
present several families of
weighted surface algebras,
which have exceptional properties, and explain also the assumptions
\ref{ass}, and show why some algebras must be excluded.
For the examples we will use surface triangulations, but only to motivate
the names for the algebras. The background is explained in \cite{ESk-WSA} and
we will not repeat this.

\begin{example}
\label{ex:3.1}
We introduce \emph{disc algebras}.
Let $T$ be the self-folded triangulation
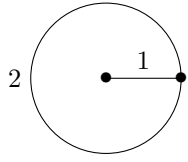
\[
%\begin{tikzpicture}[scale=.7,auto]
\begin{tikzpicture}[scale=1,auto]
\coordinate (c) at (0,0);
\coordinate (a) at (1,0);
%\coordinate (b) at (0,-1);
\coordinate (b) at (-1,0);
\draw (c) to node {$1$} (a);
%\draw (b) arc (-90:270:1) node [below] {$b$};
%\draw (.707,.707) arc (45:405:1) node [above right] {$2$};
\draw (b) arc (-180:180:1) node [left] {$2$};
\node (a) at (1,0) {$\bullet$};
\node (c) at (0,0) {$\bullet$};
\end{tikzpicture}
\]
of the unit disc $D = D^2$ in $\bR^2$,
and $\vec{T}$ the canonical orientation
of the edges of $T$.
Then the associated triangulation quiver
$(Q,f) = (Q(D,\vec{T}),f)$ is the quiver
\[
  \xymatrix{
%  \xymatrix@C=1pc{
    1
    \ar@(ld,ul)^{\alpha}[]
    \ar@<.5ex>[r]^{\beta}
    & 2
    \ar@<.5ex>[l]^{\gamma}
    \ar@(ru,dr)^{\sigma}[]
  }
%,
\]
with $f$-orbits
$(\alpha\ \beta\ \gamma)$
and
$(\sigma)$.
Then the $g$-orbits are
$\cO(\alpha) = (\alpha)$
and
$\cO(\beta) = (\beta\ \sigma\ \gamma)$.
Let $c_{\bullet} : \cO(g) \to K^*$
be a parameter function,
and let $a = c_{\cO(\alpha)}$
and  $b = c_{\cO(\beta)}$.
We consider special cases of weight functions
$m_{\bullet} : \cO(g) \to \bN^*$, the first special
case gives the disc algebras, and the second
 needs to be excluded.

\smallskip

	(1) \  
Assume  that $m_{\cO(\alpha)} = 3$
and $m_{\cO(\beta)} = 1$.
Then the associated weighted surface algebra
$D(a,b) = \Lambda(D,\vec{T},m_{\bullet},c_{\bullet})$
is given by the quiver $Q$ and the relations:
\begin{align*}
 \alpha\beta &= b \beta \sigma ,
 &
 \beta \gamma &= a \alpha^2 ,
 &
 \gamma \alpha &= b \sigma \gamma ,
 &
 \sigma^2 &= b \gamma \beta ,
%\\
 &
 \alpha\beta\sigma &= 0,
 &
 \beta\gamma\beta &= 0 ,
% &
\\
 \gamma \alpha^2 &= 0,
 &
 \sigma^2 \gamma &= 0 ,
%\\
 &
 \alpha^2 \beta &= 0,
 &
 \beta \sigma^2 &= 0 ,
 &
 \sigma \gamma \alpha &= 0,
 &
 \gamma \beta \gamma &= 0 .
\end{align*}
An algebra $D(a,b)$, with $a,b \in K^*$,
is said to be a \emph{disc algebra}.
We note that the algebra $D(a,b)$
is isomorphic to the algebra $D(ab,1)$.
Indeed, there is an isomorphism of algebras
$\varphi : D(a b, 1) \to D(a,b)$
given by
$\varphi(\alpha) = \alpha$,
$\varphi(\beta) = \beta$,
$\varphi(\gamma) = b \gamma$,
$\varphi(\sigma) = b \sigma$.
For $\lambda \in K^*$,
we set
$D(\lambda) = D(\lambda, 1)$.
A disc algebra $D(\lambda)$
with $\lambda \in K \setminus \{0,1\}$
is said to be a \emph{non-singular disc algebra},
and $D(1)$ the \emph{singular disc algebra}.

\smallskip

(2) \  
Assume that $m_{\cO(\alpha)} = 2$
and $m_{\cO(\beta)} = 1$.
Then the associated weighted surface algebra
$\Lambda = \Lambda(D,\vec{T},m_{\bullet},c_{\bullet})$
is given by the quiver $Q$ and the relations:
\begin{align*}
 \gamma \alpha &= b \sigma \gamma ,
 &
 \alpha\beta &= b \beta \sigma ,
 &
 \beta \sigma^2 &= 0 ,
 &
 \gamma \alpha^2 &= 0,
 &
 \alpha\beta\sigma &= 0,
\\
\beta \gamma &= a \alpha ,
 &
 \sigma^2 &= b \gamma \beta ,
 &
 \sigma^2 \gamma &= 0 ,
%\\
 &
 \alpha^2 \beta &= 0,
 &
 \sigma \gamma \alpha &= 0.
\end{align*}
Then $\Lambda$ is isomorphic to the algebra $A$
given by the quiver
\[
  \xymatrix{
%  \xymatrix@C=1pc{
    1
    \ar@<.5ex>[r]^{\beta}
    & 2
    \ar@<.5ex>[l]^{\gamma}
    \ar@(ru,dr)^{\sigma}[]
  }
%,
\]
and the relations:
$\beta \sigma = 0$,
$\sigma \gamma = 0$,
$\sigma^2 = \gamma \beta$.
Hence $\Lambda$ is a $7$-dimensional symmetric
representation-finite algebra of Dynkin type $\bA_4$.
Observe that we have
$m_{\beta} n_{\beta} = 3$
and
$m_{\bar{\beta}} n_{\bar{\beta}} = m_{\alpha} n_{\alpha} = 2$.
Therefore, we do not consider such an algebra $\Lambda$,
by the general assumption \ref{ass}. 
\end{example}

\begin{example}
\label{ex:3.2} 
We recall the \emph{tetrahedral algebras} introduced in
\cite[Example~6.1]{ESk-WSA}.
Let $T$ be the tetrahedron
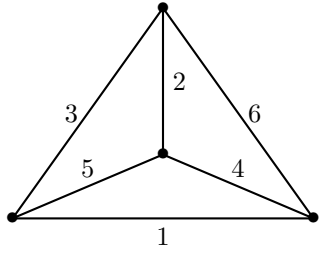
\[
\begin{tikzpicture}
[scale=1]
\node (A) at (-2,0) {$\bullet$};
\node (B) at (2,0) {$\bullet$};
\node (C) at (0,.85) {$\bullet$};
\node (D) at (0,2.8) {$\bullet$};
\coordinate (A) at (-2,0) ;
\coordinate (B) at (2,0) ;
\coordinate (C) at (0,.85) ;
\coordinate (D) at (0,2.8) ;
\draw[thick]
(A) edge node [left] {3} (D)
(D) edge node [right] {6} (B)
(D) edge node [right] {2} (C)
(A) edge node [above] {5} (C)
(C) edge node [above] {4} (B)
(A) edge node [below] {1} (B) ;
\end{tikzpicture}
\]
with the coherent orientation $\vec{T}$ of triangles:
$(1\ 5\ 4)$,
$(2\ 5\ 3)$,
$(2\ 6\ 4)$,
$(1\ 6\ 3)$.
Then the associated triangulation quiver
$(Q,f) = (Q(T,\vec{T}),f)$ is of the form
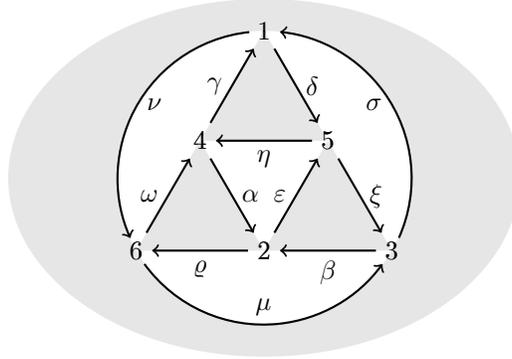
\[
\begin{tikzpicture}
[scale=.85]
\node (1) at (0,1.72) {$1$};
\node (2) at (0,-1.72) {$2$};
\node (3) at (2,-1.72) {$3$};
\node (4) at (-1,0) {$4$};
\node (5) at (1,0) {$5$};
\node (6) at (-2,-1.72) {$6$};
\coordinate (1) at (0,1.72);
\coordinate (2) at (0,-1.72);
\coordinate (3) at (2,-1.72);
\coordinate (4) at (-1,0);
\coordinate (5) at (1,0);
\coordinate (6) at (-2,-1.72);
\fill[fill=gray!20]
      (0,2.22cm) arc [start angle=90, delta angle=-360, x radius=4cm, y radius=2.8cm]
 --  (0,1.72cm) arc [start angle=90, delta angle=360, radius=2.3cm]
     -- cycle;
\fill[fill=gray!20]
    (1) -- (4) -- (5) -- cycle;
\fill[fill=gray!20]
    (2) -- (4) -- (6) -- cycle;
\fill[fill=gray!20]
    (2) -- (3) -- (5) -- cycle;

\node (1) at (0,1.72) {$1$};
\node (2) at (0,-1.72) {$2$};
\node (3) at (2,-1.72) {$3$};
\node (4) at (-1,0) {$4$};
\node (5) at (1,0) {$5$};
\node (6) at (-2,-1.72) {$6$};
\draw[->,thick] (-.23,1.7) arc [start angle=96, delta angle=108, radius=2.3cm] node[midway,right] {$\nu$};
\draw[->,thick] (-1.87,-1.93) arc [start angle=-144, delta angle=108, radius=2.3cm] node[midway,above] {$\mu$};
\draw[->,thick] (2.11,-1.52) arc [start angle=-24, delta angle=108, radius=2.3cm] node[midway,left] {$\sigma$};
\draw[->,thick]
% (5) edge node [right] {$\delta$} (1)
 (1) edge node [right] {$\delta$} (5)
(2) edge node [left] {$\varepsilon$} (5)
(2) edge node [below] {$\varrho$} (6)
(3) edge node [below] {$\beta$} (2)
% (1) edge node [left] {$\gamma$} (4)
 (4) edge node [left] {$\gamma$} (1)
(4) edge node [right] {$\alpha$} (2)
(5) edge node [right] {$\xi$} (3)
% (4) edge node [below] {$\eta$} (5)
 (5) edge node [below] {$\eta$} (4)
(6) edge node [left] {$\omega$} (4)
;
\end{tikzpicture}
\]
where $f$ is the permutation of arrows of order $3$ described
by the shaded triangles.
Then $g$ is the permutation of arrows of $Q$ of order $3$
described by the four white triangles.
Let
$m_{\bullet} : \cO(g) \to \bN^*$
be the trivial weight function
and
$c_{\bullet} : \cO(g) \to K^*$
an arbitrary parameter function.
It was shown in \cite[Section~6]{ESk-WSA}
that the  weighted surface algebra
$\Lambda(T,\vec{T},m_{\bullet},c_{\bullet})$
is isomorphic to the weighted triangulation algebra
$\Lambda(\lambda) = \Lambda(Q,f,m_{\bullet},c_{\bullet}^{\lambda})$,
with $\lambda \in K^*$,
and
the parameter function
$c_{\bullet}^{\lambda} : \cO(g) \to K^*$
given by
$c_{\cO(\alpha)}^{\lambda} = \lambda$,
$c_{\cO(\beta)}^{\lambda}  = 1$,
$c_{\cO(\gamma)}^{\lambda}  = 1$,
$c_{\cO(\sigma)}^{\lambda}  = 1$.
Observe that
$\Lambda(\lambda)$
is given by the quiver $Q$ and
the relations:
\begin{align*}
  \gamma \delta &= \lambda \alpha \varepsilon ,
  &
  \delta \eta &= \nu \omega ,
  &
  \eta \gamma &= \xi \sigma ,
  &
  \alpha \varrho &= \gamma \nu ,
  &
  \varrho \omega &= \lambda \varepsilon \eta ,
  &
  \omega \alpha &= \mu \beta ,
  \\
  \beta \varepsilon &= \sigma \delta ,
  &
  \varepsilon \xi &= \varrho \mu ,
  &
  \xi \beta &= \lambda \eta \alpha ,
  &
  \sigma \nu &= \beta \varrho ,
  &
  \nu \mu &= \delta \xi ,
  &
\mu \sigma &= \omega \gamma ,
  \\
  && \!\!\!\!\!\!\!\!\!\!\!\!\!\!\!\!\!\!\!\!\!
\theta f(\theta) g\big(f(\theta)\big) &= 0 &
\!\!\!\!\mbox{ and } \ && \!\!\!\!\!\!\!\!\!\!\!\!\!\!\!\!\!\!\!\!
\theta g(\theta) f\big(g(\theta)\big) &= 0 &
 \!\!\!\!\!\!\!\!\!\!\!\! \mbox{ for all arrows } \!\!\!\!\!\!\!\!\!\!\!\!\!\!\!\!\! &&
 \!\!\!\!\!\theta \in Q_1 .&
\end{align*}
Moreover, by  \cite[Lemma~6.2]{ESk-WSA},
the algebra $\Lambda(\lambda)$
is isomorphic to the trivial extension algebra
$\T(B(\lambda))$
of the algebra $B(\lambda)$
given by the quiver
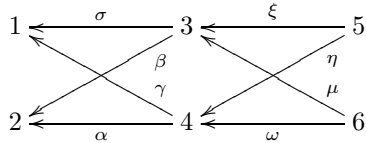
\[
%  \xymatrix@C=1.5pc{
%  \xymatrix@C=4.5pc@R=3pc{
  \xymatrix@C=4.5pc@R=2pc{
    1 &
    3 \ar[l]_{\sigma} \ar[ld]^(.2){\beta} &
    5 \ar[l]_{\xi} \ar[ld]^(.2){\eta}
    \\
    2 &
    4 \ar[l]^{\alpha} \ar[lu]_(.2){\gamma} &
    6 \ar[l]^{\omega} \ar[lu]_(.2){\mu}
  }
\]
and the relations:
\begin{align*}
 \eta \gamma &= \xi \sigma,
 &
 \xi \beta &= \lambda \eta \alpha,
 &
 \mu \sigma &= \omega \gamma,
 &
 \omega \alpha &= \mu \beta.
\end{align*}
We note that, for $\lambda \in K \setminus \{ 0,1\}$,
$B(\lambda)$ is a tubular algebra of type $(2,2,2,2)$
in the sense of \cite{R},
and hence it is an algebra of polynomial growth.
On the other hand,
$B(1)$ is the tame minimal non-polynomial growth
algebra (30)
from \cite{NoS}.
Following \cite{ESk-WSA},
an algebra $\Lambda(\lambda)$
with $\lambda \in K^*$
is said to be a \emph{tetrahedral algebra}.
Further,
an algebra $\Lambda(\lambda)$ with
with $\lambda \in K \setminus \{ 0,1\}$
is called to be a \emph{non-singular tetrahedral algebra},
while the algebra $\Lambda(1)$
the \emph{singular tetrahedral algebra}.

There is a natural connection between
the disc algebra $D(\lambda)$ and the
tetrahedral algebra $\Lambda(\lambda)$,
for any $\lambda \in K^*$.
Namely, the  cyclic group $H$
of order $3$ acts on  $\Lambda(\lambda)$  by  cyclic
rotation of vertices and arrows of the quiver
$Q = Q(T,\vec{T})$:
\begin{align*}
 (1\ 6\ 3),
  &&
 (4\ 2\ 5),
  &&
 (\alpha\ \varepsilon\ \eta),
  &&
  (\beta\ \delta\ \omega),
  &&
  (\gamma\ \varrho\ \xi),
  &&
  (\sigma\ \nu\ \mu) .
\end{align*}
Then $D(\lambda)$ is the orbit
algebra $\Lambda(\lambda)/H$.
\end{example}
%\newpage

\begin{example}
	\label{ex:3.3} 
	We introduce {\it triangle algebras}, and also
	describe an algebra which we must exclude.
	Let $T$ be the triangulation
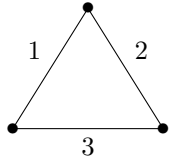
\[
\begin{tikzpicture}[auto]
\coordinate (a) at (0,1.6);
\coordinate (b) at (-1,0);
\coordinate (c) at (1,0);
\draw (a) to node {$2$} (c)
(c) to node {$3$} (b);
\draw (b) to node {$1$} (a);
\node (a) at (0,1.6) {$\bullet$};
\node (b) at (-1,0) {$\bullet$};
\node (c) at (1,0) {$\bullet$};
\end{tikzpicture}
\]
of the sphere $S^2$ in $\bR^3$ given by two
unfolded triangles and
$\vec{T}$ the coherent orientation
$(1\ 2\ 3)$
and
$(2\ 1\ 3)$
of the triangles in $T$.
Then the associated triangulation quiver
$(Q,f) = (Q(S^2,\vec{T}),f)$ is of the form
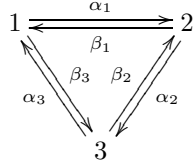
\[
%  \xymatrix@R=2pc@C=1.5pc{
%  \xymatrix@R=3.5pc@C=1.8pc{
  \xymatrix@R=3.pc@C=1.8pc{
%  \xymatrix@C=.8pc{
    1
    \ar@<.35ex>[rr]^{\alpha_1}
    \ar@<.35ex>[rd]^{\beta_3}
    && 2
   \ar@<.35ex>[ll]^{\beta_1}
    \ar@<.35ex>[ld]^{\alpha_2}
    \\
    & 3
    \ar@<.35ex>[lu]^{\alpha_3}
    \ar@<.35ex>[ru]^{\beta_2}
  }
\]
with  $f$-orbits
$(\alpha_1\ \alpha_2\ \alpha_3)$ and $(\beta_1\ \beta_3\ \beta_2)$.
Then $\cO(g)$ consists of the three $g$-orbits
$\cO(\alpha_1) = (\alpha_1\ \beta_1)$,
$\cO(\alpha_2) = (\alpha_2\ \beta_2)$,
$\cO(\alpha_3) = (\alpha_3\ \beta_3)$.
Let $m_{\bullet} : \cO(g) \to \bN^*$
be the weight function
with
$m_{\cO(\alpha_1)} = 2$,
$m_{\cO(\alpha_2)} = 2$,
and
$m_{\cO(\alpha_3)} = 1$.
Moreover,
let $c_{\bullet} : \cO(g) \to K^*$
be an arbitrary parameter function
and
$c_1 = c_{\cO(\alpha_1)}$,
$c_2 = c_{\cO(\alpha_2)}$,
$c_3 = c_{\cO(\alpha_3)}$.
Then the associated weighted surface algebra
$T(c_1,c_2,c_3) = \Lambda(S^2,\vec{T},m_{\bullet},c_{\bullet})$
is given by the quiver $Q$ and the relations:
\begin{align*}
 \alpha_1 \alpha_2 &= c_3 \beta_3 ,
 &
 \alpha_2 \alpha_3 &= c_1\beta_1 \alpha_1 \beta_1 ,
 &
 \alpha_3 \alpha_1 &= c_2 \beta_2 \alpha_2 \beta_2 ,
\\
 \beta_2 \beta_1 &= c_3 \alpha_3 ,
 &
 \beta_1 \beta_3 &= c_2 \alpha_2 \beta_2 \alpha_2 ,
 &
 \beta_3 \beta_2 &= c_1 \alpha_1 \beta_1 \alpha_1 ,
\\
 \alpha_2 \alpha_3 \beta_3 &= 0,
 &
 \alpha_3 \alpha_1 \beta_1 &= 0,
 &
 \beta_1 \beta_3 \alpha_3 &= 0,
 &
 \beta_3 \beta_2 \alpha_2 &= 0,
\\
\alpha_1 \beta_1 \beta_3 &= 0,
 &
 \alpha_3 \beta_3 \beta_2 &= 0,
 &
 \beta_2 \alpha_2 \alpha_3 &= 0,
 &
 \beta_3 \alpha_3 \alpha_1 &= 0.
\end{align*}
An algebra
$T(c_1,c_2,c_3)$, with
$c_1,c_2,c_3 \in K^*$,
is said to be a \emph{triangle algebra}.
We note that the algebra
$T(c_1,c_2,c_3)$
is isomorphic to the algebra
$T(c_1 c_2 c_3^2,1,1)$.
Indeed, there is an isomorphism of algebras
$\varphi : T(c_1 c_2 c_3^2,1,1) \to T(c_1,c_2,c_3)$
given by
\begin{align*}
 \varphi(\alpha_1) &= (c_2 c_3)^{-\tfrac{1}{4}} \alpha_1 , &
 \varphi(\alpha_2) &= (c_2 c_3)^{\tfrac{1}{4}} \alpha_ 2, &
 \varphi(\alpha_3) &= c_3 \alpha_ 3, \\
 \varphi(\beta_1) &= (c_2 c_3)^{-\tfrac{1}{4}} \beta_1 , &
 \varphi(\beta_2) &= (c_2 c_3)^{\tfrac{1}{4}} \beta_ 2, &
 \varphi(\beta_3) &= c_3 \beta_ 3.
\end{align*}
For $\lambda \in K^*$,
we set
$T(\lambda) = T(\lambda, 1, 1)$.
A triangle algebra $T(\lambda)$
with $\lambda \in K \setminus \{0,1\}$
is said to be a \emph{non-singular triangle algebra},
and $T(1)$ the \emph{singular triangle algebra}.

The triangle algebra $T(\lambda)$
is isomorphic to the algebra $T(\lambda)^{0}$
given by the Gabriel quiver $Q_{T(\lambda)}$
\[
  \xymatrix{
%  \xymatrix@C=1pc{
    1
    \ar@<.5ex>[r]^{\alpha_1}
    & 2
    \ar@<.5ex>[l]^{\beta_1}
 \ar@<.5ex>[r]^{\alpha_2}
    & 3
    \ar@<.5ex>[l]^{\beta_2}
  } 
\]
of $T(\lambda)$
and the induced relations:
\begin{align*}
 \alpha_2 \beta_2 \beta_1  &= \lambda \beta_1 \alpha_1 \beta_1 ,
 &
 \alpha_2 \beta_2 \beta_1 \alpha_1 \alpha_2  &= 0 ,
 &
 \alpha_1 \beta_1 \alpha_1 \alpha_2  &= 0 ,
 \\
 \beta_2 \beta_1 \alpha_1  &= \beta_2 \alpha_2 \beta_2 ,
 &
 \beta_2 \beta_1 \alpha_1 \beta_1  &= 0 ,
 &
 \beta_2 \beta_1 \alpha_1 \alpha_2 \beta_1 &= 0 ,
 \\
 \beta_1 \alpha_1 \alpha_2  &=  \alpha_2 \beta_2 \alpha_2  ,
 &
 \beta_1 \alpha_1 \alpha_2 \beta_2 \beta_1 &= 0 ,
&
 \beta_2 \alpha_2 \beta_2 \beta_1 &= 0 ,
 \\
 \alpha_1 \alpha_2 \beta_2  &= \lambda \alpha_1 \beta_1 \alpha_1 ,
 &
 \alpha_1 \alpha_2 \beta_2 \alpha_2  &= 0 ,
&
 \alpha_1 \alpha_2 \beta_2 \beta_1 \alpha_1 &= 0 .
\end{align*}

We also note that $T(1)^0$
is not a symmetric (even self-injective) algebra.
Indeed, if $\lambda = 1$,
then
$\beta_1 \alpha_1 - \beta_2 \alpha_2$
and
$\beta_1 \alpha_1 \beta_1 \alpha_1 = \beta_2 \alpha_2 \beta_2 \alpha_2$
are independent elements of the indecomposable
projective module $P_2$ at the vertex $2$,
which are annihilated by the radical of $T(\lambda)^0$,
and hence are in the socle of $P_2$.
Therefore,  $T(1) \cong T(1)^0$ is excluded here.

\smallskip

(2) \ 
It follows from our general assumption that if
$m_{\alpha_3} n_{\alpha_3} = 2$
then
$m_{\alpha_1} n_{\alpha_1} \geq 3$
and
$m_{\alpha_2} n_{\alpha_2} \geq 3$,
and hence
$m_{\alpha_1} n_{\alpha_1} \geq 4$
and
$m_{\alpha_2} n_{\alpha_2} \geq 4$.
The reason for such restriction is as follows:
if we would allow that two  or three of the numbers
$m_{\alpha_1} n_{\alpha_1}$,
$m_{\alpha_2} n_{\alpha_2}$,
$m_{\alpha_3} n_{\alpha_3}$,
are equal to $2$ then the associated
triangulation algebra
$\Lambda(Q,f,m_{\bullet},c_{\bullet})$
would be infinite dimensional.
\end{example}

\begin{example}\label{ex:3.4} The following example will give 
another construction 
of some triangle algebras, as well it is related to algebras 
of quaternion type in \cite{E3}.
Consider the triangulation $T$
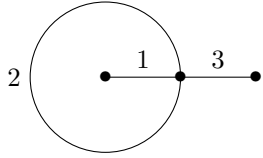
\[
%\begin{tikzpicture}[scale=.7,auto]
\begin{tikzpicture}[scale=1,auto]
\coordinate (c) at (0,0);
\coordinate (a) at (1,0);
\coordinate (d) at (2,0);
%\coordinate (b) at (0,-1);
\coordinate (b) at (-1,0);
\draw (c) to node {$1$} (a);
\draw (a) to node {$3$} (d);
%\draw (b) arc (-90:270:1) node [below] {$b$};
%\draw (.707,.707) arc (45:405:1) node [above right] {$2$};
\draw (b) arc (-180:180:1) node [left] {$2$};
\node (a) at (1,0) {$\bullet$};
\node (c) at (0,0) {$\bullet$};
\node (d) at (2,0) {$\bullet$};
\end{tikzpicture}
\]
of the sphere $S^2$ in $\bR^3$
given by two self-folded triangles,
and the canonical orientation $\vec{T}$
of triangles of $T$.
Then the associated triangulation quiver
$(Q,f) = (Q(S^2,\vec{T}),f)$ is the quiver
\[
  \xymatrix{
%  \xymatrix@C=1pc{
    1
%    \ar `ld_u[] `_rd[]^{\alpha} []
    \ar@(ld,ul)^{\alpha}[]
    \ar@<.5ex>[r]^{\beta}
    & 2
    \ar@<.5ex>[l]^{\gamma}
%    \ar `ru_d[] `_lu[]^{\eta} [] &
    \ar@<.5ex>[r]^{\sigma}
    & 3
    \ar@<.5ex>[l]^{\delta}
%    \ar `ru_d[] `_lu[]^{\xi} [] &
    \ar@(ru,dr)^{\eta}[]
  }
%,
\]
with $f$-orbits
$(\alpha\ \beta\ \gamma)$
and
$(\eta\ \delta\ \sigma)$.
Then $\cO(g)$ consists of the $g$-orbits:
$\cO(\alpha) = (\alpha)$,
$\cO(\beta) = (\beta\ \sigma\ \delta\ \gamma)$,
$\cO(\eta) = ( \eta)$.
Let
$m_{\bullet} : \cO(g) \to \bN^*$
be the weight function with
$m_{\cO(\alpha)} = 2$,
$m_{\cO(\beta)} = 1$
and
$m_{\cO(\eta)} = 2$.
Moreover, let
$c_{\bullet} : \cO(g) \to K^*$
be an arbitrary parameter function.
Write 
$a = c_{\cO(\alpha)}$,
$b = c_{\cO(\beta)}$,
$c = c_{\cO(\eta)}$.
Then the associated weighted surface algebra
$\Sigma(a,b,c) = \Lambda(S^2,\vec{T},m_{\bullet},c_{\bullet})$
is given by the quiver $Q$ and the relations:
\begin{align*}
  \alpha \beta &= b \beta \sigma \delta ,&
  \beta \gamma &= a \alpha ,&
  \gamma \alpha &= b \sigma \delta \gamma ,
 &
 \eta \delta &= b \delta \gamma \beta ,&
  \delta \sigma &= c \eta ,&
  \sigma \eta &= b \gamma \beta \sigma,
 \\
  &&
  \alpha \beta \sigma &= 0 ,&
  \gamma \alpha^2 &= 0 ,&
  \eta \delta \gamma &= 0,&
%  \delta \sigma \delta &= 0, &
  \sigma \eta^2 &= 0 ,
% &
\\  &&
  \alpha^2 \beta &= 0 ,
 &
  \beta \sigma \eta &= 0 ,
 &
% \\ && &&
  \eta^2 \delta &= 0 ,&
  \delta \gamma \alpha &= 0 .
%  \sigma \delta \sigma &= 0 .
\end{align*}
Observe now that there is an isomorphism of algebras
$\psi : \Sigma(1, b(a c)^{\tfrac{1}{2}},1) \to \Sigma(a,b,c)$
given by
\begin{align*}
 \psi(\alpha) &= (a c)^{\tfrac{1}{2}} \alpha , &
 \psi(\beta) &= a^{-\tfrac{1}{2}} \beta , &
 \psi(\gamma) &= c^{\tfrac{1}{2}} \gamma , &
 \psi(\eta) &= c \eta , &
 \psi(\delta) &= \delta , &
 \psi(\sigma) &= \sigma .
\end{align*}
For $\lambda \in K^*$,
we set
$\Sigma(\lambda) = \Sigma(1, \lambda, 1)$.
We observe that $\Sigma(\lambda)$
is isomorphic to the algebra
$\Sigma(\lambda)^0$
given by the Gabriel quiver $Q_{\Sigma(\lambda)}$
\[
  \xymatrix{
%  \xymatrix@C=1pc{
    1
    \ar@<.5ex>[r]^{\beta}
    & 2
    \ar@<.5ex>[l]^{\gamma}
    \ar@<.5ex>[r]^{\sigma}
    & 3
    \ar@<.5ex>[l]^{\delta}
  } 
%,
\]
of $\Sigma(\lambda)$
and the induced relations:
\begin{align*}
  \beta \gamma \beta &= \lambda \beta \sigma \delta ,&
%  \beta \gamma &= a \alpha ,&
  \gamma \beta \gamma  &= \lambda \sigma \delta \gamma ,
 &
  \delta \sigma \delta &= \lambda \delta \gamma \beta ,&
%  \delta \sigma &= c \eta ,&
  \sigma \delta \sigma &= \lambda \gamma \beta \sigma,
 \\
%  &&
  \beta \gamma  \beta \sigma &= 0 ,&
  \gamma \beta \gamma \beta \gamma &= 0 ,&
  \delta \sigma \delta \gamma &= 0,&
%  \delta \sigma \delta &= 0, &
  \sigma \delta \sigma \delta \sigma &= 0 ,
% &
\\
%  &&
  \beta \gamma \beta \gamma \beta &= 0 ,
 &
% \\ && &&
  \delta \sigma \delta \sigma \delta &= 0 ,&
  \delta \gamma  \beta \gamma &= 0 .
\end{align*}
We note that for any $\lambda \in K^*$,
there is an isomorphism of algebras
$\theta : T(\lambda)^0 \to \Sigma(-\lambda)^0$
given by
$\theta(\beta) = -\beta$,
$\theta(\gamma) = \gamma$,
$\theta(\sigma) = \sigma$,
$\theta(\delta) = \delta$.
\end{example}

We have the following lemma.

\begin{lemma}
\label{lem:3.5}
For each $\lambda \in K^*$,
the algebras $T(\lambda^{-2})$ and $\Sigma(\lambda)$
are isomorphic.
\end{lemma}

\begin{proof}
Fix $\lambda \in K^*$.
Then there is an isomorphism of algebras
$\varphi : T(\lambda^{-2})^0 \to \Sigma(\lambda)^0$
given by
$\varphi(\alpha_1) = \lambda \beta$,
$\varphi(\beta_1) = \gamma$,
$\varphi(\alpha_2) = \sigma$,
$\varphi(\beta_2) = \delta$.
\end{proof}

In particular, we conclude that
$\Sigma(1)^0 \cong T(1)^0$
and
$\Sigma(-1)^0 \cong T(1)^0$
are not symmetric (even self-injective) algebras.
Hence $\Sigma(1)$ and $\Sigma(-1)$ are excluded here. 

\begin{example}
	\label{ex:3.6} We will now define 
	{\it spherical algebras}.
Consider the following triangulation $T(2)$
of the sphere $S^2$ in $\bR^3$
(compare \cite[Example~7.5]{ESk6})
\[
\begin{tikzpicture}[xscale=2,yscale=1.5,auto]
\coordinate (o) at (0,0);
\coordinate (a) at (0,1);
\coordinate (b) at (-.5,0);
\coordinate (c) at (.5,0);
\coordinate (d) at (0,-1);
\coordinate (l) at (-1,0);
\coordinate (r) at (1,0);
\draw (a) to node {4} (c)
(c) to node {6} (d)
(d) to node {5} (b)
(b) to node {2} (a);
\draw (d) to node {3} (a);
\node (a) at (a) {$\bullet$};
\node (b) at (b) {$\bullet$};
\node (c) at (c) {$\bullet$};
\node (d) at (d) {$\bullet$};
\draw (l) arc (-180:180:1) node [left] {$1$};
\draw (r) node [right] {$1$};
\end{tikzpicture}
\]
and $\amsvectb{T(2)}$ is the coherent orientation
of triangles in $T(2)$:
\[
\mbox{
  (1 2 5), (2 3 5), (3 4 6),
  (4 1 6).
}
\]
Then the associated triangulation quiver
$(Q,f) = (Q(S^2,\amsvectb{T(2)}),f)$ is of the form
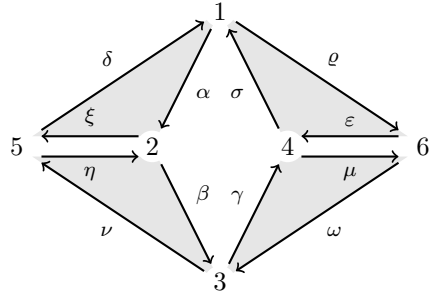
\[
\begin{tikzpicture}
[->,scale=.9]
\coordinate (1) at (0,2);
\coordinate (2) at (-1,0);
\coordinate (2u) at (-.925,.15);
\coordinate (2d) at (-.925,-.15);
\coordinate (3) at (0,-2);
\coordinate (4) at (1,0);
\coordinate (4u) at (.925,.15);
\coordinate (4d) at (.925,-.15);
\coordinate (5) at (-3,0);
\coordinate (5u) at (-2.775,.15);
\coordinate (5d) at (-2.775,-.15);
\coordinate (6) at (3,0);
\coordinate (6u) at (2.775,.15);
\coordinate (6d) at (2.775,-.15);
\fill[fill=gray!20] (1) -- (5u) -- (2u) -- cycle;
\fill[fill=gray!20] (1) -- (4u) -- (6u) -- cycle;
\fill[fill=gray!20] (2d) -- (5d) -- (3) -- cycle;
\fill[fill=gray!20] (3) -- (6d) -- (4d) -- cycle;
\node [fill=white,circle,minimum size=4.5] (1) at (0,2) {\ \quad};
\node [fill=white,circle,minimum size=4.5] (2) at (-1,0) {\ \quad};
\node [fill=white,circle,minimum size=4.5] (3) at (0,-2) {\ \quad};
\node [fill=white,circle,minimum size=4.5] (4) at (1,0) {\ \quad};
\node [fill=white,circle,minimum size=4.5] (5) at (-3,0) {\ \quad};
\node [fill=white,circle,minimum size=4.5] (6) at (3,0) {\ \quad};
\node (1) at (0,2) {1};
\node (2) at (-1,0) {2};
\node (2u) at (-1,0.15) {\ \quad};
\node (2d) at (-1,-0.15) {\ \quad};
\node (3) at (0,-2) {3};
\node (4) at (1,0) {4};
\node (4u) at (1,0.15) {\ \quad};
\node (4d) at (1,-0.15) {\ \quad};
\node (5) at (-3,0) {5};
%\node (5u) at (-3,0.15) {\ \quad};
%\node (5d) at (-3,-0.15) {\ \quad};
\node (5u) at (-2.775,0.15) {};
\node (5d) at (-2.775,-0.15) {};
\node (6) at (3,0) {6};
\node (6u) at (2.775,0.15) {};
\node (6d) at (2.775,-0.15) {};
\draw[thick,->]
(1) edge node[below right]{\footnotesize$\alpha$} (2)
(2u) edge node[above]{\footnotesize$\xi$} (5u)
(5u) edge node[above left]{\footnotesize$\delta$} (1)
(5d) edge node[below]{\footnotesize$\eta$} (2d)
(2) edge node[above right]{\footnotesize$\beta$} (3)
(3) edge node[below left]{\footnotesize$\nu$} (5d)
(1) edge node[above right]{\footnotesize$\varrho$} (6u)
(6u) edge node[above]{\footnotesize$\varepsilon$} (4u)
(4) edge node[below left]{\footnotesize$\sigma$} (1)
(4d) edge node[below]{\footnotesize$\mu$} (6d)
(6d) edge node[below right]{\footnotesize$\omega$} (3)
(3) edge node[above left]{\footnotesize$\gamma$} (4)
;
\end{tikzpicture}
\]
where
the four shaded triangles denote the
$f$-orbits.
Then $g$ has four orbits
\begin{align*}
 \cO(\alpha)  &= ( \alpha\ \beta\ \gamma\ \sigma), &
 \cO(\varrho)  &= ( \varrho\ \omega\ \nu\ \delta), &
 \cO(\xi)  &= ( \xi\ \eta), &
% \cO(\varepsilon)  &= ( \varepsilon\ \mu).
 \cO(\mu)  &= ( \mu\ \varepsilon ).
\end{align*}
Let
$m_{\bullet} : \cO(g) \to \bN^*$
be the weight function which takes all values $1$. Moreover,  
let
$c_{\bullet} : \cO(g) \to K^*$
be an arbitrary parameter function
and
$a = c_{\cO(\alpha)}$,
$b = c_{\cO(\varrho)}$,
$c = c_{\cO(\xi)}$,
$d = c_{\cO(\mu)}$.
Then the associated weighted surface algebra
$S(a,b,c,d) = \Lambda(S^2,\vec{T(2)},m_{\bullet}^r,c_{\bullet})$
is given by the quiver $Q$ and the relations:
\begin{align*}
%&&
  \alpha \xi &= b \varrho \omega \nu ,
  &
  \xi \delta &= a \beta \gamma \sigma ,
  &
  \delta \alpha &= c \eta ,
  &
  \beta \nu &= c \xi ,
  &
  \nu \eta &= a \gamma \sigma \alpha ,
\!\!\!\!\!\!\!\!
%  &
\\
  \eta \beta &=  b \delta \varrho \omega ,
%\!\!\!\!\!\!\!\!
%\\
  &
%&&
  \gamma \mu &=  b \nu \delta \varrho ,
  &
  \mu \omega &= a \sigma \alpha \beta ,
&
  \omega \gamma &= d \varepsilon ,
  &
  \sigma \varrho &= d \mu ,
%  &
\!\!\!\!\!\!\!\!
\\
  \varrho \varepsilon &= a \alpha \beta \gamma ,
  &
  \varepsilon \sigma &= b \omega \nu \delta ,
%\!\!\!\!\!\!\!\!
%\\
  &
  \alpha \xi \eta &= 0 ,
  &
  \xi \delta \varrho &= 0 ,
  &
  \nu \eta \xi &= 0 ,
  &
  \eta \beta \gamma &= 0 ,
%  &
\\
  \gamma \mu \varepsilon &= 0 ,
  &
  \mu \omega \nu &= 0 ,
  &
  \varrho \varepsilon \mu &= 0 ,
  &
  \varepsilon \sigma \alpha &= 0 ,
  &
%\\
  \beta \gamma \mu &= 0 ,
  &
  \sigma \alpha \xi &= 0 ,
%  &
\\
  \delta \varrho \varepsilon &= 0 ,
  &
  \omega \nu \eta &= 0 ,
  &
  \xi \eta \beta &= 0 ,
  &
  \eta \xi \delta &= 0 ,
  &
  \mu \varepsilon \sigma &= 0 ,
  &
  \varepsilon \mu \omega &= 0 .
\end{align*}
An algebra $S(a,b,c,d)$
with $a,b,c,d \in K^*$
is said to be a \emph{spherical algebra}.
We observe now that the algebra $S(a,b,c,d)$
is isomorphic to the algebra $S(abcd,1,1,1)$.
Indeed, there is an isomorphism of algebras
$\varphi : S(abcd,1,1,1) \to S(a,b,c,d)$
given by
\begin{align*}
 \varphi(\alpha) &= \alpha , &
 \varphi(\beta) &= \beta , &
 \varphi(\gamma) &= \gamma , &
 \varphi(\sigma) &= \sigma ,
\\
 \varphi(\varrho) &= (b c)^{\tfrac{1}{2}} \varrho , &
 \varphi(\omega) &= (b c)^{\tfrac{1}{2}} \omega , &
 \varphi(\nu) &= (b d)^{\tfrac{1}{2}} \nu , &
 \varphi(\delta) &= (b d)^{\tfrac{1}{2}} \delta ,
\\
 \varphi(\xi) &= (b d)^{\tfrac{1}{2}} c \xi , &
 \varphi(\eta) &= (b d)^{\tfrac{1}{2}} c \eta , &
 \varphi(\mu) &= (b c)^{\tfrac{1}{2}} d \mu , &
 \varphi(\varepsilon) &= (b c)^{\tfrac{1}{2}} d \varepsilon .
\end{align*}
For $\lambda \in K^*$,
we set
$S(\lambda) = S(\lambda,1,1,1)$.
A spherical algebra $S(\lambda)$
with $\lambda \in K \setminus \{0,1\}$
is said to be a \emph{non-singular spherical algebra},
and $S(1)$ the \emph{singular spherical algebra}.

We observe now that a spherical algebra $S(\lambda)$
is isomorphic to the algebra $S(\lambda)^0$
given by the Gabriel quiver $Q_{S(\lambda)}$
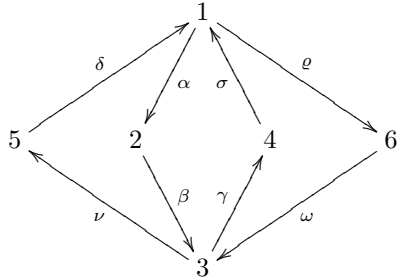
\[
%  \xymatrix@R=2pc@C=1.5pc{
%  \xymatrix@R=3.5pc@C=1.8pc{
  \xymatrix@R=3.pc@C=1.2pc{
%  \xymatrix@C=.8pc{
    &&& 1
    \ar[ld]^{\alpha}
    \ar[rrrd]^{\varrho}
    \\   
    5
    \ar[rrru]^{\delta}
%    \ar@<-.5ex>[rr]_(.6){\eta}
    && 2
%    \ar@<-.5ex>[ll]_(.4){\xi}
    \ar[rd]^{\beta}
    && 4
    \ar[lu]^{\sigma}
%    \ar@<-.5ex>[rr]_(.4){\mu}
    && 6
%    \ar@<-.5ex>[ll]_(.6){\varepsilon}
    \ar[llld]^{\omega}
    \\
   &&& 3
    \ar[lllu]^{\nu}
    \ar[ur]^{\gamma}
  }
\]
of $S(\lambda)$
and the induced relations:
\begin{align*}
  \alpha \beta \nu &= \varrho \omega \nu ,
  &
  \beta \nu \delta &= \lambda \beta \gamma \sigma ,
  &
  \nu \delta \alpha &= \lambda \gamma \sigma \alpha ,
  &
  \delta \alpha \beta &=  \delta \varrho \omega ,
\\
  \gamma \sigma \varrho &=  \nu \delta \varrho ,
&
  \sigma \varrho \omega &= \lambda \sigma \alpha \beta ,
  &
 \varrho \omega \gamma &= \lambda \alpha \beta \gamma ,
  &
  \omega \gamma \sigma &= \omega \nu \delta ,
\\
  \alpha \beta \nu \delta \alpha &= 0 ,
  &
  \beta \nu \delta \varrho &= 0 ,
  &
  \nu \delta \alpha \beta \nu &= 0 ,
  &
  \delta \alpha \beta \gamma &= 0 ,
\\
 \gamma \sigma \varrho \omega \gamma &= 0 ,
  &
  \sigma \varrho \omega \nu &= 0 ,
  &
  \varrho \omega \gamma \sigma \varrho &= 0 ,
  &
  \omega \gamma \sigma \alpha &= 0 ,
\\
  \beta \gamma \sigma \varrho &= 0 ,
  &
  \sigma \alpha \beta \nu &= 0 ,
  &
  \delta \varrho \omega \gamma &= 0 ,
  &
  \omega \nu \delta \alpha &= 0 ,
\\
  \beta \nu \delta \alpha \beta &= 0 ,
  &
  \delta \alpha \beta \nu \delta &= 0 ,
 &
  \sigma \varrho \omega \gamma \sigma &= 0 ,
  &
  \omega \gamma \sigma \varrho \omega &= 0 .
\end{align*}
Moreover, a minimal set of relations defining
$S(\lambda)^0$
is given by the above eight commutativity relations and the
four zero relations:
\begin{align*}
  \beta \nu \delta \varrho &= 0 ,
  &
  \delta \alpha \beta \gamma &= 0 ,
  &
  \sigma \varrho \omega \nu &= 0 ,
  &
  \omega \gamma \sigma \alpha &= 0 .
\end{align*}

We also note that $S(1)^0$,
and hence $S(1)$,
is not a symmetric (even self-injective) algebra.
Indeed, if $\lambda = 1$,
then
$\alpha \beta - \varrho \omega$
and
$\alpha \beta \nu \delta = \delta \omega \gamma \sigma$
are independent elements of the indecomposable
projective module $P_1$ at the vertex $1$,
which are annihilated by the radical of $S(1)^0$,
and hence are in the socle of $P_2$.
Therefore, we exclude  $S(1)$.

For each $\lambda \in K \setminus \{ 0,1 \}$,
we denote by $C(\lambda)$ the $K$-algebra
given by the quiver
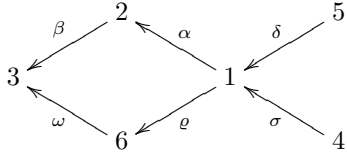
\[
%  \xymatrix{
%  \xymatrix@C=2.5pc{
  \xymatrix@C=2.5pc@R=1pc{
%  \xymatrix@C=4.5pc@R=3pc{
    & 2 \ar[ld]_{\beta} && 5 \ar[ld]_{\delta} \\
    3 && 1 \ar[lu]_{\alpha} \ar[ld]^{\varrho} \\
    & 6 \ar[lu]^{\omega} && 4 \ar[lu]^{\sigma} 
  }
\]
and the relations:
$\delta \alpha \beta = \delta \varrho \omega$
and
$\delta \varrho \omega = \lambda \sigma \alpha \beta$.
We note that $C(\lambda)$ is the double
one-point extension algebra of the path algebra
$H = K \Delta$ of the quiver $\Delta$
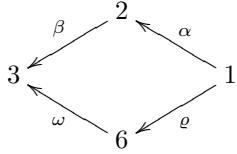
\[
%  \xymatrix{
%  \xymatrix@C=2.5pc{
  \xymatrix@C=2.5pc@R=1pc{
%  \xymatrix@C=4.5pc@R=3pc{
    & 2 \ar[ld]_{\beta} \\
    3 && 1 \ar[lu]_{\alpha} \ar[ld]^{\varrho} \\
    & 6 \ar[lu]^{\omega}  
  }
\]
of Euclidean type $\widetilde{\bA}_3$
by two indecomposable modules
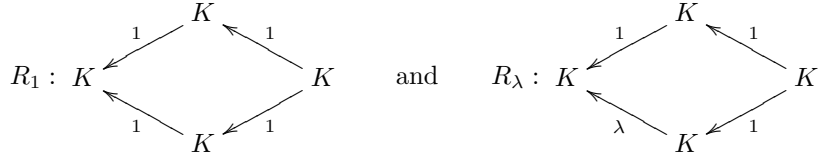
\[
  R_{1} :
\vcenter{
%  \xymatrix{
%  \xymatrix@C=2.5pc{
  \xymatrix@C=2.5pc@R=1pc{
%  \xymatrix@C=4.5pc@R=3pc{
    & K \ar[ld]_{1} \\
    K && K \ar[lu]_{1} \ar[ld]^{1} \\
    & K \ar[lu]^{1}  
  }
}
\qquad
 \mbox{and}
 \qquad
  R_{\lambda} :
%  \xymatrix@C=1.5pc{
\vcenter{
%  \xymatrix{
%  \xymatrix@C=2.5pc{
  \xymatrix@C=2.5pc@R=1pc{
%  \xymatrix@C=4.5pc@R=3pc{
    & K \ar[ld]_{1} \\
    K && K \ar[lu]_{1} \ar[ld]^{1} \\
    & K \ar[lu]^{\lambda}  
  }
}
\]
lying on the mouth of stable tubes of rank $1$ in $\Gamma_H$.
For $\lambda \in K \setminus \{0,1\}$,
the modules $R_1$ and $R_{\lambda}$ are not isomorphic,
and then $C(\lambda)$ is a tubular algebra of type
$(2,2,2,2)$ in the sense of \cite{R},
and consequently it is an algebra of polynomial growth.
On the other hand, $C(1)$ is a tame algebra
of non-polynomial growth
(see \cite{NoS}).
\end{example}

\begin{lemma}
\label{lem:3.7}
For any $\lambda \in K \setminus \{0,1\}$,
the algebras $S(\lambda)$
and $\T(C(\lambda))$ are isomorphic.
\end{lemma}

\begin{proof}
By general theory (see \cite{Sk2}),
the trivial extension algebra $\T(C(\lambda))$
is isomorphic to the orbit algebra
$\widehat{C(\lambda)} / (\nu_{\widehat{C(\lambda)}})$
of the repetitive category $\widehat{C(\lambda)}$
of $C(\lambda)$ with respect to the infinite
cyclic group $(\nu_{\widehat{C(\lambda)}})$
generated by the Nakayama automorphism
$\nu_{\widehat{C(\lambda)}}$ of $\widehat{C(\lambda)}$.
One checks directly  that $\widehat{C(\lambda)}$
contains the full convex subcategory
given by the quiver
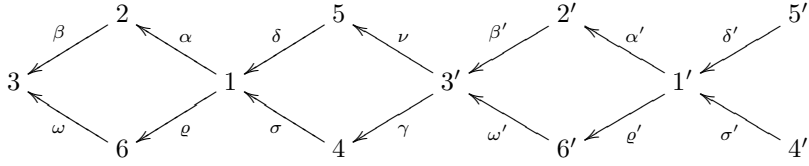
\[
%  \xymatrix{
%  \xymatrix@C=2.5pc{
  \xymatrix@C=2.5pc@R=1pc{
%  \xymatrix@C=4.5pc@R=3pc{
    & 2 \ar[ld]_{\beta} && 5 \ar[ld]_{\delta} 
      && 2' \ar[ld]_{\beta'} && 5' \ar[ld]_{\delta'} \\
    3 && 1 \ar[lu]_{\alpha} \ar[ld]^{\varrho} 
&&  3' \ar[lu]_{\nu} \ar[ld]^{\gamma} 
      && 1' \ar[lu]_{\alpha'} \ar[ld]^{\varrho'} \\
    & 6 \ar[lu]^{\omega} && 4 \ar[lu]^{\sigma} 
      && 6' \ar[lu]^{\omega'} && 4' \ar[lu]^{\sigma'} 
  }
\]
and the relations:
\begin{align*}
  \delta \alpha \beta &= \delta \varrho \omega ,
  &
  \sigma \varrho \omega &= \lambda \sigma \alpha \beta ,
  &
  \nu \delta \alpha &= \lambda \gamma \sigma \alpha ,
  &
  \gamma \sigma \varrho &= \nu \delta \varrho ,
 \\
  \beta' \nu \delta &= \lambda \beta' \gamma \sigma ,
  &
  \omega' \gamma \delta &= \omega' \nu \delta ,
  &
  \alpha' \beta' \nu &= \varrho' \omega' \nu ,
 &
  \varrho' \omega' \gamma &= \lambda \alpha' \beta' \gamma ,
 \\
  \beta' \nu \delta \varrho &= 0 ,
  &
  \omega' \gamma \delta \alpha &= 0 ,
  &
  \delta' \alpha' \beta' \gamma &= 0 ,
  &
  \sigma' \varrho' \omega' \nu &= 0 ,
\end{align*}
where
$\nu_{\widehat{C(\lambda)}}(i) = i'$ for any vertex $i \in \{1,2,3,4,5,6\}$
and
$\nu_{\widehat{C(\lambda)}}(\theta) = \theta'$ for any arrow
$\theta \in \{\alpha,\beta,\omega,\varrho,\delta,\sigma\}$.
We conclude that $\T(C(\lambda))$ is isomorphic
to the algebra
$S(\lambda)^{0}$,
and hence to the spherical algebra
$S(\lambda)$.
\end{proof}

We also note that
there is a natural action of the cyclic group $G$
of order $2$ on $S(\lambda)^{0}$
given by the cyclic rotation of vertices and arrows
of the quiver
$Q_{S(\lambda)}$:
\begin{align*}
 (1\ 3),
  &&
 (2\ 4),
  &&
 (5\ 6),
  &&
 (\alpha\ \gamma),
  &&
  (\beta\ \sigma),
  &&
  (\varrho\ \nu),
&&
  (\omega\ \delta) .
\end{align*}
Then the orbit algebra $S(\lambda)^0/G$
is isomorphic to the basic algebra
$T(\lambda)^{0}$ of the triangle algebra $T(\lambda)$,
for any  $\lambda \in K \setminus \{0,1\}$.

\smallskip

We describe now some 
special properties of the exceptional weighted surface algebras
introduced above.

\begin{proposition}
\label{prop:3.8}
Let $\Lambda$ be a non-singular algebra
$D(\lambda)$,
$\Lambda(\lambda)$,
$T(\lambda)$,
$S(\lambda)$,
$\lambda \in K \setminus \{0,1\}$.
Then the following  hold:
\begin{enumerate}[(i)]
 \item
  $\Lambda$ is an algebra of polynomial growth.
 \item
  The simple modules in $\mod \Lambda$ are periodic of period $4$.
 \item
  $\Lambda$ is a periodic algebra of period $4$.
\end{enumerate}
\end{proposition}

\begin{proof}
(i)
It follows from the above discussion that
$\Lambda(\lambda) \cong \T(B(\lambda))$,
$D(\lambda) \cong \Lambda(\lambda)/H$,
$S(\lambda) \cong \T(C(\lambda))$,
$T(\lambda) \cong S(\lambda)/G$,
where
$B(\lambda)$ and $C(\lambda)$
are tubular algebras of type $(2,2,2,2)$,
and $G$ and $H$ are cyclic groups of orders
$2$ and $3$, respectively.
Then the fact that $\Lambda$ is of polynomial growth
follows from \cite[Theorem]{Sk1}.

\smallskip

(ii)
It follows from general theory of self-injective
algebras of type $(2,2,2,2)$ that all simple modules
in $\mod \Lambda$ lie in stable tubes of rank $2$
in $\Gamma^s_{\Lambda}$
(see \cite[Section~3]{NeS} and \cite[Section~3]{Sk1}).
Since $\Lambda$ is a symmetric algebra,
we conclude that all simple modules in $\mod \Lambda$
are periodic of period $4$.

\smallskip

(iii)
It has been proved in \cite[Proposition~7.1]{BES4}
that $D(\lambda)$ is a periodic algebra of period $4$.
Then, applying \cite[Theorem~3.7]{Du1},
	we concluded in \cite[Proposition~5.8]{ESk-WSA}
that $\Lambda(\lambda)$ is a periodic algebra of period $4$.
Further, it follows from \cite{HR}
(see also \cite[5.2(5)]{R})
that the tubular algebras
$B(\lambda)$ and $C(\lambda)$
are derived equivalent,
and hence their trivial extension algebras
$\T(B(\lambda))$ and $\T(C(\lambda))$
are derived equivalent,
by \cite[Theorem~3.1]{Ric2}.
Then, since
$\Lambda(\lambda) \cong \T(B(\lambda))$
is a periodic algebra of period $4$,
we conclude that
$S(\lambda) \cong \T(C(\lambda))$
is also a periodic algebra of period $4$
(see \cite[Theorem~2.9]{ESk3}).
Finally, applying again \cite[Theorem~3.7]{Du1},
we infer that
$T(\lambda) \cong S(\lambda)/G$
is also a periodic algebra of period $4$.
\end{proof}

\begin{proposition}
\label{prop:3.9}
Let $\Lambda$ be a singular algebra
$D(1)$ or $\Lambda(1)$.
Then the following hold:
\begin{enumerate}[(i)]
 \item
  $\Lambda$ is a tame algebra of non-polynomial growth.
 \item
$\mod \Lambda$ does not have a simple periodic module.
 \item
  $\Lambda$ is not a periodic algebra.
\end{enumerate}
\end{proposition}

\begin{proof}
(i)
The fact that $\Lambda(1)$
is tame algebra of non-polynomial growth
follows from \cite[Theorem~2]{ESk-HTA}.
Applying arguments from
\cite[Section~5]{ESk-HTA}, we conclude
similarly that the orbit algebra
$D(1) \cong \Lambda(1)/H$
is also a tame algebra of non-polynomial growth.
For  $\Lambda = \Lambda(1)$,
the statement (ii) follows from
\cite[Proposition~6.4]{ESk-WSA}.
Let $\Lambda = D(1)$.
We note that for the indecomposable projective
$\Lambda$-modules $P_1$ and $P_2$
at  vertices $1$ and $2$,
we have
$\rad P_1/S_1 \cong \rad P_2/S_2$,
and hence the simple modules
$S_1$ and $S_2$ are non-periodic.
Part (iii) follows from (ii) and general theory 
(see Theorem IV.11.19 of \cite{SY}.
 \end{proof}

\section{Properties of general weighted surface algebras}%
\label{sec:properties}

We will first discuss the assumptions, and special cases, 
and then analyse positions of virtual arrows. 
We determine a basis of a  weighted surface algebra. 
Then we prove that a weighted
surface algebra is, 
other than the singular triangle,
or spherical algebra,
is symmetric.

%\bigskip
\smallskip

Let $\Lambda = \Lambda(Q,f,m_{\bullet},c_{\bullet})$
be a  weighted triangulation algebra.

\begin{remark}\label{rem:4.1}
%\normalfont 
(i) 
We have excluded Example \ref{ex:3.1}, part (2). 
Namely  part (3) of Assumption~\ref{ass} requires that $m_{\beta}n_{\beta}\geq 4$
since the arrow $\alpha$ of the example  is a virtual loop.

(ii) 
In Example \ref{ex:3.3} we have that arrows $\alpha_3, \beta_3$ are
virtual and $m_{\alpha}n_{\alpha}=4$ for any arrow $\alpha \neq 
\alpha_3, \beta_3$. If one would allow (say) $m_{{\alpha}_2}=1$ then 
by Lemma \ref{lem:4.4} (below) 
the  algebra would not be finite-dimensional. 
Therefore we exclude it, which is done by condition (2) of 
Assumption \ref{ass}.

(iii) 
There are some special cases when certain parameters must
be excluded: For the triangle algebra, we exclude $1$. 
   In Example 3.4, we must exclude parameters $\pm 1$.
For the spherical algebra, we also exclude $1$. For these
parameters the algebras are not symmetric, this will be proved 
in Proposition  \ref{prop:4.9}. 
\end{remark}

\medskip

\begin{remark}\label{rem:4.2}
\normalfont
We analyse possible configurations near some virtual arrow.
As we have already seen, conditions (2) and (3) of \ref{ass} show that
it is not possible that $\alpha, \ba$ are both virtual, and using that
$g$ takes virtual arrows to virtual arrows, also $f^2(\alpha)$ and
$f^2(\ba)$ cannot be both virtual.

\smallskip

%(a) 
(i) 
Assume $\ba$ is a virtual loop. 
Then by the above, no other $f^j(\alpha)$ or $f^j(\ba)$ is virtual. 
In fact, the $g$-cycle of $\alpha$ has length
at least three. 
(If it has length three so that $Q$ has two vertices
then we have by condition (3) of \ref{ass}, $m_{\alpha}\geq 2$.)
The arrows $g(\alpha)$ and $g^{-1}(f^2(\ba)$ are therefore not virtual. 
We may have that $f(g(\alpha))$ is virtual. 
If it is a loop then $Q$ is the quiver as in Example \ref{ex:3.4}. 
Otherwise $|Q_0|\geq 5$.

\smallskip

%(b) 
(ii) 
Now assume $\ba$ is virtual but not a loop, then $g$ has cycle
$(\ba \ f^2(\alpha))$ and also $f^2(\alpha)$ is virtual. 
By conditions	(2) and (3) of \ref{ass} no other 
$f^j(\alpha)$ or $f^j(\ba)$ is virtual. 
Also none of these arrows can be a loop since otherwise 
$Q$ would not be 2-regular. 
So $Q$ has a subquiver
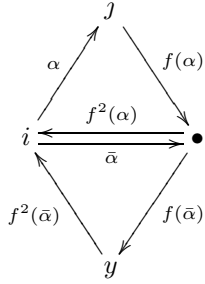
\[
%  \xymatrix@R=2pc@C=1.5pc{
%  \xymatrix@R=3.5pc@C=1.8pc{
  \xymatrix@R=3.pc@C=1.8pc{
%  \xymatrix@C=.8pc{
    & j
    \ar[rd]^{f(\alpha)}
    \\
    i
    \ar[ru]^{\alpha}
    \ar@<-.5ex>[rr]_{\ba}
    && \bullet
    \ar@<-.5ex>[ll]_{f^2(\alpha)}
 \ar[ld]^{f(\ba)}
    \\
    & y
    \ar[lu]^{f^2(\ba)}
 }
\]
where $j$ and $y$ could be equal. 
If $j=y$ then $|Q_0|=3$ and there are no further virtual arrow. 
If $j\neq y$ there may or may not be loops at $j$
or/and $y$ but they cannot be virtual, for example by part (i).

\smallskip

%(c) 
(iii) 
We note that if $\ba$ is a loop fixed by $g$, then in any 
case $n_{\alpha}\geq 3$. 
\end{remark}

We mention a few consequences.

\begin{lemma}\label{lem:4.3} 
Let $i \in Q_0$ and let $\alpha, \ba$ 
be the arrows starting at $i$. 
%\begin{enumerate}[(a)]
\begin{enumerate}[(i)]
 \item
	Assume $f(\alpha)$ is virtual, then $\ba$ is not virtual.
 \item
	If $f^2(\alpha)$ is virtual then $\ba$ is virtual and $g(f(\alpha))= f(\ba)$.
 \item
	If $\alpha, \ba$ are double arrows then they are both not virtual.
\end{enumerate}
\end{lemma}

\begin{proof}
%	(a) 
	(i) 
We  know that $f^2(\alpha)$ is not virtual  and then
$g(f^2(\alpha)) = \ba$ also is not virtual.

\smallskip

	(ii) 
Let $j=t(\ba)$, then $g(f(\alpha), f(\ba)$ start at $j$
and also $f^2(\alpha)$ starts at $j$. 
Now, $f^2(\alpha)$ is virtual but $g(f(\alpha)), f(\ba)$ 
are not virtual, so they must be equal.

\smallskip
	
	(iii) 
Assume $\alpha, \ba$ are double arrows. 
Assume for a contradiction that $\ba$ (say) is virtual.
Recall $\ba = g(f^2(\alpha))$, so $g$ has a 2-cycle 
$(\ba \ f^2(\alpha))$. 
It follows that $f(\alpha)$ is a loop at $t(\alpha)$. 
It is necessarily fixed	by $g$ and we have a contradiction to 
\ref{rem:4.2} part (iii). 
\end{proof}

The following, already announced, explains why condition (2) 
of Assumption~\ref{ass} is necessary.

\begin{lemma}\label{lem:4.4} 
Suppose there exists a pair of virtual arrows $\alpha$ and $\ba$.
Then $\La$ is not finite-dimensional.
\end{lemma}

\begin{proof}
By \ref{rem:4.2} and \ref{lem:4.3}, the
arrows $\alpha, \ba$ are not loops or double arrows. Then using
Remark \ref{rem:4.2}(ii), we see that $Q$ has a subquiver
\[
  \xymatrix{
%  \xymatrix@C=1pc{
    1
    \ar@<.5ex>[r]^{f^2(\ba)}
    & 2
    \ar@<.5ex>[l]^{\alpha}
    \ar@<.5ex>[r]^{\ba}
    & 3
    \ar@<.5ex>[l]^{f^2(\alpha)}
  }
.
\]

Since  $g(f^2(\alpha)) = \ba$, it follows that  
$f(\alpha)$ is an arrow $1\to 3$, and similarly
$f(\ba)$ is an arrow $3\to 1$.
The subquiver with vertices $ 1, 2, 3$ is 2-regular 
and hence it is equal to $Q$. 
This is the triangulation quiver as in Example \ref{ex:3.3}, 
and we use the labelling from \ref{ex:3.3}, and with this
we have 
$$
   g= (\alpha_1 \ \beta_1)(\alpha_2 \ \beta_2)(\alpha_3 \ \beta_3).
$$
Assume (say) $\alpha_2, \beta_2$ and $\alpha_3, \beta_3$ are virtual. 
If $\alpha_1, \beta_1$ are also virtual then we 
do not have any zero relations and the algebra 
is not finite-dimensional. 
Suppose now that $\alpha_1, \beta_1$ are not virtual. 
Then the zero-relations are
\begin{align*} 
 \alpha_2\alpha_3\beta_3 &=0, & \beta_3\beta_2\alpha_2 &=0, \cr
 \alpha_3\beta_3\beta_2&=0, & \beta_2\alpha_2\alpha_3 &= 0.
\end{align*}
Hence $\alpha_1, \beta_1$ do not occur in a zero relation. 
We have other relations, in particular
$$
 \alpha_2\alpha_3 = c_{\alpha_1}(\beta_1\alpha_1)^{m-1}\beta_1, \ \
 \beta_3\beta_2 = c_{\alpha_1}(\alpha_1\beta_1)^{m-1}\alpha_1
$$
(where $m=m_{\alpha_1}$),
and one checks that they are consistent with the other relations, 
and do not cause a zero relation for $\beta_1\alpha_1$. 
It follows that the powers $(\beta_1\alpha_1)^r$ for $r=1, 2, \ldots$ are
linearly independent in $\La$ and the algebra is not finite-dimensional.
%\end{proof}

\medskip

We note that directly 
imposing nilpotence relations for $\alpha_1, \beta_1$ would not
produce an algebra as we wish. Namely, suppose we add the relations
$$
  (\beta_1\alpha_1)^{m-1}\beta_1=0, \ \ (\alpha_1\beta_1)^{m-1}\alpha_1=0.
$$
Then the resulting algebra has Gabriel quiver consisting of one isolated vertex
together with  
$$
\xymatrix{
  1  \ar@<+.5ex>[r]^{}
   & 2 \ar@<+.5ex>[l]^{}
 }
$$
and the algebra is the product of an algebra of finite type with a 1-dimensional
simple algebra.
\end{proof}

Let $\Lambda = \Lambda(Q,f,m_{\bullet},c_{\bullet})$ 
be a weighted surface algebra, and
$I = I(Q,f,m_{\bullet},c_{\bullet})$.
In order to study properties of $\Lambda$, and modules,
we work towards  specifying  a suitable basis of the algebra $\Lambda$,
defined in terms of cycles of  $g$. In the following,
we will as usual identify an element of $K Q$ with its
residue class in $\Lambda = KQ/I$.
In addition to the elements  $A_{\alpha}$ occuring in the
definition, we will also use monomials   of length 
$m_{\alpha}n_{\alpha}-2 (\geq 0)$. 
If $\alpha$ is an arrow, define $A_{\alpha}'$ by
$$\alpha A_{\alpha}' = A_{\alpha}.$$
If $\alpha$ is virtual then $A_{\alpha}'$ is the idempotent 
$e_{t(\alpha)}$.

\begin{lemma}
\label{lem:4.5}
Let $\alpha$ be an arrow in $Q$.
Then the following hold:
\begin{enumerate}[(i)]
 \item
  $B_{\alpha} \rad \Lambda = 0$.
 \item
  $B_{\alpha}$ is non-zero.
\item If $\alpha$ is not virtual then $A_{\alpha} \rad^2\La = 0$.
\end{enumerate}
\end{lemma}

\begin{proof}
%$(i)$
(i)
We must show that $B_{\alpha} \ba = 0$
and $B_{{\alpha}} \alpha = 0$
in $\Lambda$.
It follows from (i) and (iv) of Lemma~\ref{lem:2.9}
and the relations in $\Lambda$ that
\begin{align*}
 c_{{\alpha}} B_{{\alpha}} \ba
  &= \alpha f(\alpha) f^2(\alpha) \ba.
\end{align*}
If $\alpha$ is not virtual then $f(\alpha)f^2(\alpha)\ba=0$
since  $\ba = g(f^2(\alpha))$. 	
Now assume $\alpha$ is virtual. Then 
	$$B_{\alpha}\ba = \left\{\begin{array}{ll} \alpha^2\ba & \alpha \ \mbox{is a loop}  \cr
\alpha g(\alpha) \ba & \mbox{else.}
\end{array}\right.
$$
In the first case, $\alpha = g(\alpha)$ and $f(\alpha) = \ba$, and 
$$
  \alpha^2\ba = \alpha g(\alpha) f(g(\alpha)) = 0
$$
	noting that $f(\alpha) = \ba$ cannot be virtual (by the general
	assumption).
In the second case, $g(\alpha) = f^2(\ba)$ which is virtual, and then 
$f(\alpha)$ is not virtual since there would be otherwise two 
virtual arrows starting at the same vertex. Therefore
$$\alpha g(\alpha)\ba = \alpha f^2(\ba) f^3(\ba) = 0.$$
Furthermore, by interchanging the roles of $\alpha$ and $\ba$ 
we obtain also $c_{\alpha}B_{\alpha}\alpha = c_{\ba}B_{\ba} \alpha = 0$.

\smallskip

%\noindent 
(ii)
This follows from the relations defining $\Lambda$.

\smallskip

%	\noindent 
(iii) 
We have $A_{\alpha}f^2(\ba) = B_{\alpha}$ which is in the
socle by the previous. 
It remains to show that $A_{\alpha}g(f(\ba))=0$.
By the relations this is a non-zero scalar multiple of
$\ba f(\ba) g(f(\ba))$.
Since $f^2(\ba)$ is in the $g$-orbit of $\alpha$, it is not
virtual by the assumption. 
Hence by the relation (2) of the definition
we have $\ba f(\ba)g(f(\ba)) = 0$ as required.
\end{proof}

%\bigskip
\begin{remark}\label{rem:4.6} 
\normalfont
We can motivate the zero relations of Definition \ref{def:2.8}, 
and also see that the relations give rise to zero conditions.

\smallskip

(i) 
\ Consider $\alpha f(\alpha) g(f(\alpha))$ when $f^2(\alpha)$ is virtual. 
Then also $g(f^2(\alpha))$ is virtual but $g(f^2(\alpha)) = \ba$, 
and then $g(f(\alpha))= f(\ba)$.  
So we have
$$
  \alpha f(\alpha)g(f(\alpha))
   = c_{\ba}\ba g(f(\alpha)) 
   = c_{\ba}\ba f(\ba) 
   = c_{\ba}c_{\alpha}A_{\alpha},
$$
and $A_{\alpha}$ is non-zero, because $B_{\alpha}$ is non-zero.
Since $\ba$ is virtual, $\alpha$ is not virtual (see Assumption \ref{ass}).  
So $A_{\alpha} \rad^2\La = 0$, by Lemma \ref{lem:4.5}.

\smallskip

(ii) 
In the original  version, the relation $\alpha g(\alpha)f(g(\alpha)) = 0$ 
in $\La$ is a consequence of the definition. 
This is now
not the case, and we must add but not always.
Suppose $f(\alpha)$ is virtual. 
By relation
(1) we have
$$\alpha g(\alpha) f(g(\alpha)) = \alpha c_{f(\alpha)}A_{f(\alpha)} = c_{f(\alpha)}c_{\ba}A_{\ba},$$
and $A_{\ba}$ is non-zero, because $B_{\ba}$ is non-zero.
We know $\ba$ is not virtual (since $f(\alpha)$ is virtual, see Lemma~\ref{lem:4.3}). 
So again $A_{\ba}\rad^2\La = 0$. 
\end{remark}

\smallskip

%\bigskip

Lemma \ref{lem:4.5} only shows that
$\langle B_{\alpha} \rangle \subseteq \soc \La$.
We will now prove that equality holds. On the way, we see that
for some of the algebras certain parameters need to be excluded.

%\medskip

\begin{lemma} \label{lem:4.7} 
%\begin{enumerate}[(a)]
\begin{enumerate}[(i)]
 \item
%	(a) 
\ Assume $\alpha$ starting at $i$ is virtual. 
Then $\soc_2(e_i\La)$ is generated by $A_{\ba}$.
The module $e_i\La$ has basis 
$$
 \{ e_i, \ba, \ba g(\ba), \ldots, \  A_{\ba}, B_{\ba}, \ \ba f(\ba)\}.
$$

 \item
%	(b) 
\ Assume $\alpha, \ba$ are not virtual and 
$\soc (e_i\La)= \langle B_{\alpha}\rangle$. 
Then $e_i\La$ has basis all proper initial
submonomials of $B_{\alpha}$ and $B_{\ba}$ 
together with $e_i$ and $B_{\alpha}$.
\end{enumerate}
\end{lemma}

\begin{proof} 
%	(a)
	(i)
Let $\alpha$ be a virtual arrow starting at $i$, 
then $\ba$ is not virtual (see \ref{ass}).
Since virtual arrows are unions of $g$-cycles, 
no virtual arrow occurs as a factor of $B_{\ba}$.
We express $B_{\alpha}$ in terms of arrows 
of the Gabriel quiver. 
If $\alpha$ is a loop then we get
$$
  B_{\alpha} = \alpha^2  = (c_{\alpha}^{-1}\ba f(\ba))^2.
$$
Otherwise we get
$$
  B_{\alpha} = \alpha g(\alpha) = 
  (c_{\alpha}^{-1}\ba f(\ba))(c_{g(\alpha)}^{-1}f(\alpha)f^2(\alpha)).
$$
We claim that
$e_i\La$ has basis
$$
  \{ e_i, \ba, \ba g(\ba), \ldots, \ A_{\ba}, B_{\ba}, \ \ba f(\ba)\}.
$$
Assume first $\alpha$ is a loop.  Then
$$
  \ba f(\ba)\ba 
   = c_{\alpha}\alpha\ba
   = c_{\alpha} \alpha f(\alpha) 
   = c_{\alpha}c_{\ba}A_{\ba}
$$
and it follows
that the set spans $e_i\La$.
Suppose $\alpha$ is not a loop then
$$
 \ba f(\ba) f(\alpha) 
  = c_{\alpha}\alpha f(\alpha) 
  = c_{\alpha}c_{\ba}A_{\ba}
$$
and the set spans $e_i\La$. One checks that the set is linearly independent.

\smallskip

%    (b) 
    (ii) 
The set is a spanning set, by the relations. 
Using the assumption on the socle, one checks that
it is linearly independent.
\end{proof}

We will now show that the condition on the socle in part (b) is satisfied
except for those algebras which we have excluded anyway, in particular
this will give us the required basis. For the proof,
we use the following preparation.

%\medskip

\begin{lemma}\label{lem:4.8} 
Assume $\alpha, \ba$ start at vertex $i$ and are both not virtual. 
Then $A_{\alpha}$ and $A_{\ba}$ are linearly independent in $\La$.
\end{lemma}

\begin{proof}
Assume (for a contradiction) that for some $t\in K$ we have
$$A_{\alpha} + tA_{\ba} \in I. \leqno(*)$$
We have $c_{\ba}A_{\ba} - \alpha f(\alpha)\in I$ and $c_{\alpha}A_{\alpha} - \ba f(\ba) \in I$.  
Note that arrows along $A_{\alpha}, A_{\ba}$ are not virtual.
If an identity (*) exists, and given that $\alpha, \ba$ are not virtual, 
it follows that at least one of $f(\alpha)$ or  $f(\ba)$ is virtual.

%\medskip

\smallskip

The last arrows of $A_{\alpha}$ and $A_{\ba}$ end at $t(f(\ba))$ and $t(f(\alpha))$, 
which must be the same  (by (*)).
Suppose (say) $f(\alpha)$ is virtual, then also $g(f(\alpha))$ is virtual. 
The arrows starting at $t(f(\alpha))$ are $f^2(\alpha)$ and $f^2(\ba)$, and
$g(f(\alpha))\neq f^2(\alpha)$. So $g(f(\alpha)) = f^2(\ba)$. 
That is, $f^2(\ba)$ is virtual, and then also $g(f^2(\ba))$.  
But $g(f^2(\ba)) = \alpha$ and $\alpha$ is
not virtual by assumption, a contradiction.
Similarly $f(\ba)$ cannot be virtual. Hence  there is no identity (*).
\end{proof}

%\bigskip

\begin{proposition}\label{prop:4.9} 
Assume $e_i\La e_j$ has an element $\zeta$ with $\zeta J=0$ 
but $\zeta\not\in \langle B_{\alpha}\rangle$.
Then $\La$ is isomorphic 
to the singular triangular algebra $T(1)$ 
or the singular spherical algebra $S(1)$.
Conversely, $T(1)$ and $S(1)$ have this property.
\end{proposition}

%\bigskip

We split the proof into three parts, first a reduction, and then
two lemmas.

\medskip

For the reduction, if
one of the arrows starting at $i$ is virtual,  then no such $\zeta$ exists,
we can see this from the basis of $e_i\La$.
So
assume the arrows  $\alpha$ and $\ba$ starting at $i$ are both not virtual. 
Then there are no virtual arrows occuring
in $B_{\alpha},  B_{\ba}$ and we get a  spanning set  of $e_i\La$  
consisting of initial subwords of $B_{\alpha}, B_{\ba}$ (in particular
the last arrows which are $f^2(\ba)$ and $f^2(\alpha)$ are also not virtual).
Then we can write $\zeta$ in terms of the spanning set as
$$\zeta = \zeta_1 + a\zeta_2\in e_i\La e_j$$
with $a\in K$ and $\zeta_1$ a linear combination of paths along $B_{\alpha}$ and 
$\zeta_2$ a linear combination of paths along $B_{\ba}$.
Then the lowest terms of $\zeta_1, \zeta_2$  satisfy the same property, 
so we may assume that the $\zeta_i$ are monomials.
They are then two   paths from $i$ to $j$, and are linearly independent in $\La$.

%\bigskip
\begin{lemma}\label{lem:4.10} 
Assume $\alpha, \ba$ are not virtual. 
Let $\zeta = \zeta_1 + a\zeta_2$ in $e_i\La e_j$, 
such that $\zeta \rad \La = 0$ but 
$\zeta \not\in \langle B_{\alpha}\rangle$. 
Assume $\zeta_1$ is a monomial along $B_{\alpha}$ 
and $\zeta_2$ is a monomial along $B_{\ba}$. 
Then
$$
 \zeta_1 = \alpha g(\alpha), \ \ \zeta_2 = \ba g(\ba),
$$ 
and moreover $f(\alpha)$ and $f(\ba)$ are virtual.
\end{lemma}

\begin{proof}
Write $\delta_i$ for the last arrow in $\zeta_i$.
Let $\beta = g(\delta_1)$, this is not virtual and it starts at $j$.
Then $\zeta_1\beta$ is an initial submonomial of $B_{\alpha}$ 
and it occurs in some relation.
Therefore $\zeta_1\beta = A_{\alpha}$ or possibly $B_{\alpha}$.

\smallskip

{\sc Claim: } $\zeta_1\beta\neq B_{\alpha}$, that is $\zeta_1 \neq A_{\alpha}$.
%
%\smallskip
%
 Assume  $\zeta_1=A_{\alpha}$, then  $\beta = f^2(\ba)$.
We also have either $\beta = g(\delta_2)$ or $\beta = f(\delta_2)$ and we will
show that both lead to contradictions.

\smallskip

(i) Suppose $\beta = g(\delta_2)$, then $\zeta_2\beta$ 
is a monomial along $B_{\ba}$, and occurs in a relation. 
So it is either $A_{\ba}$ or $B_{\ba}$ and
then is $B_{\alpha}$. 
But then $\beta = f^2(\alpha)$ and $f^2(\ba)=f^2(\alpha)$,
a contradiction.

\smallskip

(ii) Suppose $\beta = f(\delta_2)$, then $\delta_2 = f(\ba)$ and we have
$$\zeta_2\beta = \zeta_2' f(\ba)f^2(\ba) = \zeta_2'\cdot c_{g(\ba)}A_{g(\ba)}
$$
and $\zeta_2'$ must be equal to $\ba$. 
However then we
have $\zeta_2 = \ba \delta_2 = \ba f(\ba)$ which is not a path along
a $g$-cycle, a contradiction.
So we must have that  $\zeta_1\beta = A_{\alpha}$.

\smallskip

{\sc Claim: } $g(\delta_2) = \bar{\beta}$ (and is not virtual).
Assume not, then  $g(\delta_2) = \beta$.  
This means that $\zeta_2\beta$ is an initial submonomial 
of $B_{\ba}$ and occurs in a relation. 
So it must be $A_{\ba}$ or $B_{\ba}$, and then 
by using the previous argument it must be $A_{\ba}$.  
But it is also a scalar multiple of $A_{\alpha}$, and this  
contradicts Lemma~\ref{lem:4.8}.
So $g(\delta_2) = \bar{\beta}$.

\smallskip

We can now use the previous argument again, and get that $\zeta_2\bar{\beta} = A_{\ba}$.

As well $\beta = f(\delta_2)$. 
Therefore
$$
  \zeta_2\beta = \zeta_2' \delta_2f(\delta_2)
    = c_{\bar{\delta_2}}\zeta_2' A_{\bar{\delta}_2}.
$$

We claim that $\bar{\delta}_2$ (and $\bar{\delta}_1$) are virtual.

Suppose $\bar{\delta}_2$ is not virtual, then 
$A_{\bar{\delta}_2}$ is in the second socle. 
It follows that
$\zeta_2' = e_i$ and $\zeta_2 = \ba$, 
and $\beta = f(\ba)$.
This implies $f(\beta) = f^2(\ba)$.
On the other hand, $\beta$ is the last arrow 
of $A_{\alpha}$ and therefore $g(\beta) = f^2(\ba)$,  
that is $f(\beta) = g(\beta)$, a contradiction.
The same reasoning using $\zeta_1$ and 
$\bar{\beta} = f(\delta_1)$ shows that
$\bar{\delta}_1$ is virtual.

\smallskip

Let $\eta = g^{-1}(\delta_2) = f^{-1}(\bar{\delta}_2)$. We have
$$\zeta_2\beta = c_{\bar{\delta}} \zeta_2'\bar{\delta}_2 = c_{\bar{\delta}}\zeta_2''\eta f(\eta)
= c_{\bar{\delta}_2}c_{\bar{\eta}} \zeta_2''A_{\bar{\eta}}
\leqno{(*)}$$
and this is a scalar multiple of $A_{\alpha} = \zeta_1 \beta$.

The arrow $\bar{\eta}$ is not virtual.
If $\bar{\delta}_2$ is a virtual loop 
then $g$ has cycle 
$(\eta \ \delta_2 \ \bar{\eta}\ \ldots )$ 
and $\bar{\eta}$ is not virtual.
If $\bar{\delta}_2$ is virtual but not a loop, 
we see that $f$ has cycle 
$(\eta \ \bar{\delta}_2 \ f^2(\eta))$ 
and $f^2(\eta)$ is not virtual. 
In this case $\bar{\eta}= g(f^2(\eta))$ 
and hence it is not virtual.

Now $A_{\bar{\eta}}$ is in the second socle, 
and therefore $\zeta_2''=e_i$ 
and $\eta = g^{-1}(\delta_2) = \ba$, 
and $\delta_2= g(\ba)$. 
We have proved
$\zeta_2= \ba g(\ba)$, 
and it follows that $f(\ba) = \bar{\delta}_2$ 
and we have proved that it is virtual.

The same reasoning, for $\zeta_1$ and $\bar{\beta}$ 
shows $\zeta_1 = \alpha g(\alpha)$ 
and it follows that $f(\alpha) = \bar{\delta}_1$, 
so it is virtual.
\end{proof}

\begin{lemma}\label{lem:4.11} 
Assume $\zeta = \zeta_1+a\zeta_2$ with $\zeta J=0$ but
$\zeta \not\in\langle B_{\alpha}\rangle$. 
Assume $\zeta_1=\alpha g(\alpha)$ and $\zeta_2=\ba g(\ba)$, and
moreover $f(\alpha)$ and $f(\ba)$ are
virtual, and $A_{\alpha}, A_{\ba}$ have length three. 
Then $\La$ is isomorphic to $T(1)$ or $S(1)$.
\end{lemma}

\begin{proof}
Since $A_{\alpha}$ and $A_{\ba}$ have length 3, 
we have $m_{\alpha}n_{\alpha}=4 = m_{\ba}n_{\ba}$.

\smallskip

{\sc Case 1}.
Assume $i=j$. 
Then the arrows ending at $i$ are
$\{ g(\alpha), g(\ba)\} = \{ f^2(\alpha), f^2(\ba)\}$. 
So we have two cases
to consider.

\smallskip

(a) 
Assume $g(\alpha) = f^2(\ba)$ and $g(\ba) = f^2(\alpha)$. 
Let $x=t(\ba)$ and $y=t(\alpha)$. 
Then $f(\ba)$ is an arrow $x\to y$ and $f(\alpha)$ is an
arrow $y\to x$. 
The full subquiver with vertices $i, x, y$ is $2$-regular 
and hence is equal to $Q$.
Moreover we have
$g = (\alpha \ f^2(\ba))(\ba \  f^2(\alpha)) (f(\alpha) \ f(\ba))$ 
and hence $m_{\alpha}=2$ and $m_{\ba}=2$.
Also  $f= (\alpha \ f(\alpha) \ \delta_2)(\ba \ f(\ba) \ \delta_1)$, 
so these are precisely the data for the triangle algebra 
(Example~\ref{ex:3.3}),
with the quiver
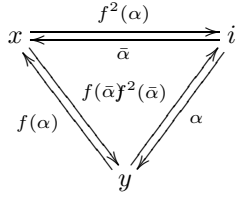
\[
%  \xymatrix@R=2pc@C=1.5pc{
%  \xymatrix@R=3.5pc@C=1.8pc{
%  \xymatrix@R=3.pc@C=1.8pc{
  \xymatrix@R=3.5pc@C=2.5pc{
%  \xymatrix@C=.8pc{
    x
    \ar@<.35ex>[rr]^{f^2(\alpha)}
    \ar@<.35ex>[rd]^{f(\bar{\alpha})}
    && i
   \ar@<.35ex>[ll]^{\bar{\alpha}}
    \ar@<.35ex>[ld]^{\alpha}
    \\
    & y
    \ar@<.35ex>[lu]^{f(\alpha)}
    \ar@<.35ex>[ru]^{f^2(\bar{\alpha})}
  }
\]
Moreover, it follows from Example~\ref{ex:3.3}
that we may take (up to algebra isomorphism)
$c_{\ba}=\lambda$, 
$c_{\alpha}=1$,
$c_{f(\alpha)}=1$,
for some $\lambda \in K^*$.
Since $\zeta J = 0$,
we obtain the equalities
\begin{align*}
 0&= \big( \alpha f^2(\ba) + a \ba f^2 (\alpha)\big) \alpha
   =  \alpha f^2(\ba)  \alpha + a \ba f^2 (\alpha) \alpha
   =  \ba f(\ba)  + a \ba f (\ba) 
\\&
   = (1+a) \ba f(\ba) , 
\\
 0&= \big( \alpha f^2(\ba) + a \ba f^2 (\alpha)\big) \ba
   =  \alpha f^2(\ba)  \ba + a \ba f^2 (\alpha) \ba
   =  \alpha f(\alpha)  + a \ba f^2 (\alpha) \ba
\\&
   = (1+a \lambda^{-1}) \alpha f(\alpha) ,
\end{align*}
and hence $a = -1$ and $\lambda = 1$.
In particular, we conclude that 
$\Lambda$ is isomorphic to $T(1)$.

\smallskip

%\noindent
(b)  Assume $g(\alpha) = f^2(\alpha)$ and $g(\ba) = f^2(\ba)$.
Then $f(\alpha)$ and $f(\ba)$ are virtual loops and again $Q$ has
three vertices, and moreover 
$g = (\alpha \ f^2(\alpha) \ \ba \ f^2(\ba))(f(\alpha))(f(\ba))$ and $m_{\alpha}=1$. 
These are the data which 
determine the algebra $\Sigma$ introduced in Example~\ref{ex:3.4},
with the quiver
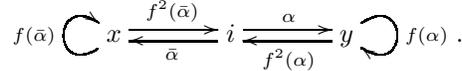
\[
%  \xymatrix{
  \xymatrix@C=2.8pc{
%  \xymatrix@C=1pc{
    x
%    \ar `ld_u[] `_rd[]^{\alpha} []
    \ar@(ld,ul)^{f(\bar{\alpha})}[]
    \ar@<.5ex>[r]^{f^2(\bar{\alpha})}
    & i
    \ar@<.5ex>[l]^{\bar{\alpha}}
%    \ar `ru_d[] `_lu[]^{\eta} [] &
    \ar@<.5ex>[r]^{\alpha}
    & y
    \ar@<.5ex>[l]^{f^2(\alpha)}
%    \ar `ru_d[] `_lu[]^{\xi} [] &
    \ar@(ru,dr)^{f(\alpha)}[]
  } 
.
\]
Moreover, it follows from Example~\ref{ex:3.4}
that we may take (up to algebra isomorphism)
$c_{\alpha}=\lambda$, 
$c_{f(\alpha)}=1$,
$c_{f(\ba)}=1$,
for some $\lambda \in K^*$,
that is
$\Lambda$ is isomorphic to $\Sigma(\lambda)$.
Since $\zeta J = 0$,
we obtain the equalities
\begin{align*}
 0&= \big( \alpha f^2(\alpha) + a \ba f^2 (\ba)\big) \alpha
   =  \alpha f^2(\alpha)  \alpha + a \ba f^2 (\ba) \alpha
   =  \alpha f(\alpha)  + a \lambda^{-1} \alpha f (\alpha) 
\\&
   = (1+a \lambda^{-1}) \alpha f(\alpha) ,
\\
 0&= \big( \alpha f^2(\alpha) + a \ba f^2 (\ba)\big) \ba
   =  \alpha f^2(\alpha)  \ba + a \ba f^2 (\ba) \ba
   =  \lambda^{-1} \ba f(\ba)  + a \ba f (\ba) 
\\&
   = (\lambda^{-1}+a) \ba f(\ba) , 
\end{align*}
and hence $a = -\lambda^{-1}$ and $\lambda^2 = 1$.
In particular, we conclude that 
$\Lambda$ is isomorphic to $\Sigma(1) \cong \Sigma(-1)$,
and consequently to $T(1)$ (see  Example~\ref{ex:3.4}).

\medskip

%{\sc Case $i\neq j$}. \
{\sc Case 2}.
Assume $i \neq j$. 
Note first that $\alpha, f^2(\alpha)$  and $\ba, f^2(\ba)$ cannot be loops by
\ref{rem:4.2} and \ref{lem:4.3}.

We claim that arrows $\alpha, g(\alpha), f(\alpha)$ and $f^2(\alpha)$ are pairwise distinct.
Clearly $f(\alpha)$ is different from the other three since
it is the only one of these which is virtual.
If $\alpha = g(\alpha)$ then $\alpha$ ends at $j$ and $g(\alpha)$ is a loop,
and $\alpha$ is a loop, a contradiction.
Clearly $\alpha\neq f^2(\alpha)$ since otherwise
$f(\alpha) =\alpha$.
Finally $f^2(\alpha)$ ends at $i$ and $g(\alpha)$ does not end at $i$.
We observe that $f(\alpha)$ is not a (virtual) loop.
Otherwise we would have $f^2(\alpha)=g(\alpha)$ which we had excluded.
Similarly $f(\ba)$ is not a loop.
Similarly, we show that 
$\ba$, $g(\ba)$, $f(\ba)$ and $f^2(\ba)$ 
are pairwise distinct.

Let $x=t(f(\alpha))$ and $y=t(f(\ba))$. 
Since $m_{\alpha}n_{\alpha}=4$ we have the
arrow $g^2(\alpha): j\to y$. 
Similarly $g^2(\ba): j\to x$. 

It follows that $g^2(\alpha) = f(g(\ba))$  
and $g^2(\ba) = f(g(\alpha))$, 
and $f(g^2(\alpha)) = g(f(\ba))$ is virtual, 
and $f(g^2(\ba)) = g(f(\alpha))$ is virtual. 
By \ref{rem:4.2} and \ref{lem:4.3}, 
the arrows $g^2(\alpha)$ and $g^2(\ba)$ are not loops. 
Hence $Q$ is the quiver
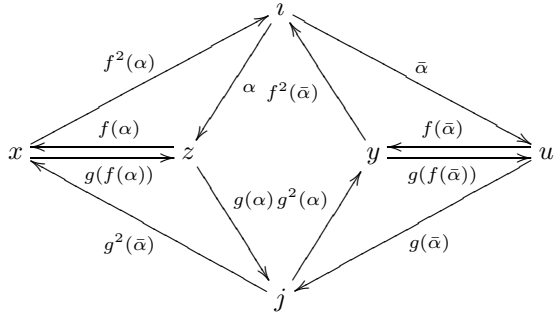
\[
%  \xymatrix@R=2pc@C=1.5pc{
%  \xymatrix@R=3.5pc@C=1.8pc{
%  \xymatrix@R=3.pc@C=1.2pc{
%  \xymatrix@R=3.5pc@C=1.8pc{
  \xymatrix@R=3.5pc@C=2pc{
%  \xymatrix@C=.8pc{
    &&& i
    \ar[ld]^{\alpha}
    \ar[rrrd]^{\bar{\alpha}}
    \\   
    x
    \ar[rrru]^{f^2(\alpha)}
    \ar@<-.5ex>[rr]_(.6){g(f(\alpha))}
    && z
    \ar@<-.5ex>[ll]_(.4){f(\alpha)}
    \ar[rd]^(.4){g(\alpha)}
    && y
    \ar[lu]^{f^2(\bar{\alpha})}
    \ar@<-.5ex>[rr]_(.4){g(f(\bar{\alpha}))}
    && u
    \ar@<-.5ex>[ll]_(.6){f(\bar{\alpha})}
    \ar[llld]^{g(\bar{\alpha})}
%    &&&& \bullet
%    \ar[llld]^{\nu}
    \\
    &&& j
    \ar[lllu]^{g^2(\bar{\alpha})}
    \ar[ur]^(.6){g^2(\alpha)}
  }
\]
of a spherical algebra,
with two pairs of virtual arrows,
and in fact the data 
are those for the spherical algebra,
namely $m_{\alpha} = 1$ and $m_{\bar{\alpha}} = 1$.
Moreover, it follows from Example~\ref{ex:3.6}
that we may take (up to algebra isomorphism)
$c_{\alpha}=\lambda$, 
$c_{\ba}=1$,
$c_{f(\alpha)}=1$,
$c_{f(\ba)}=1$,
for some $\lambda \in K^*$.
Now, since $\zeta J = 0$,
we obtain the equalities
\begin{align*}
 0&= \big( \alpha g(\alpha) + a \ba g (\ba)\big) g^2(\alpha)
   =  \alpha g(\alpha)  g^2(\alpha) + a \ba g(\ba) g^2(\alpha)
\\&
   =  \lambda^{-1} \ba f(\ba)  + a \ba f (\ba) 
   = (\lambda^{-1}+a) \ba f(\ba) , 
\\
 0&= \big( \alpha g(\alpha) + a \ba g(\ba)\big) g^2(\ba)
   =  \alpha g(\alpha)  g^2(\ba) + a \ba g(\ba) g^2(\ba)
   =  \alpha f(\alpha) + a \alpha f(\alpha)
\\&
   = (1+a) \alpha f(\alpha) ,
\end{align*}
and hence $a = -1$ and $\lambda = 1$.
In particular, we conclude that 
$\Lambda$ is isomorphic to $S(1)$.
\end{proof}

Since we exclude these algebras,  Lemma~\ref{lem:4.7} 
shows that we have a basis of $e_i\La$ in terms of initial submonomials
of $B_{\alpha}$ and $B_{\ba}$. 
Hence we get the expected formula for the dimension.

\begin{corollary}
\label{cor:4.12}
Let $i$ be a vertex of $Q$ and
$\alpha,\bar{\alpha}$ the two arrows in $Q$
with source $i$.
Then
$\dim_K e_i \Lambda =
  m_{\alpha} n_{\alpha}
  + m_{\bar{\alpha}} n_{\bar{\alpha}}$.
\end{corollary}

\begin{proposition}
\label{prop:4.13}
Let $(Q,f)$ be a triangulation quiver,
$m_{\bullet}$ and $c_{\bullet}$
weight and parameter functions of $(Q,f)$,
%and $\Lambda$ be a weighted surface algebra.
and 
$\Lambda = \Lambda(Q,f,m_{\bullet},c_{\bullet})$. 
Then the following statements hold:
\begin{enumerate}[(i)]
 \item
  $\Lambda$ is  finite-dimensional  with
  $\dim_K \Lambda = \sum_{\cO \in \cO(g)} m_{\cO} n_{\cO}^2$.
 \item
  $\Lambda$ is a symmetric algebra, except when $\La$ is the
		singular triangle, or spherical algebra. 
\end{enumerate}
\end{proposition}

\begin{proof}
Let 
$I = I(Q,f,m_{\bullet},c_{\bullet})$.

\smallskip

%$(i)$
(i)
It follows from Corollary~\ref{cor:4.12} that,
for each vertex $i$ of $Q$,
the indecomposable projective right
$\Lambda$-module $P_i$ at the vertex $i$
has the dimension
$\dim_K P_i = m_{\alpha} n_{\alpha} + m_{\bar{\alpha}} n_{\bar{\alpha}}$,
where $\alpha, \bar{\alpha}$ are the two arrows in $Q$ with source $i$.
Then we get
\[
  \dim_K \Lambda = \sum_{\cO \in \cO(g)} m_{\cO} n_{\cO}^2 .
\]

\smallskip

%$(ii)$
(ii)
	For each vertex $i \in Q_0$, we denote by
$\cB_i$ the basis of $e_i \Lambda$ consisting of $e_i$,
all initial subwords of $A_{\alpha}$ and $A_{\bar{\alpha}}$,
and
$\omega_i = c_{\alpha} B_{\alpha} = c_{\bar{\alpha}} B_{\bar{\alpha}}$
(see Lemma~\ref{lem:4.7} and Corollary~\ref{cor:4.12}).
We know that $\omega_i$ generates the socle of $e_i \Lambda$.
Then $\cB = \bigcup_{i \in Q_0} \cB_i$ is a $K$-linear basis of $\Lambda$.
We  defined a symmetrizing $K$-linear form
$\varphi : \Lambda \to K$ which assigns
to the coset $u + I$ of a path $u$ in $Q$ the element in $K$
\[
   \varphi(u+I) = \left\{ \begin{array}{cl}
      c_{\alpha}^{-1} & \mbox{if $u = \cB_{\alpha}$ for an arrow $\alpha \in Q_1$}, \\
	   0 & \mbox{if $u \in \cB$ otherwise}
   \end{array} \right.
\]
and extending linearly.
\end{proof}

\section{Periodicity of weighted surface algebras}\label{sec:per}

%\noindent
In this section we will prove that every weighted surface algebra
with at least  two simple modules,
not isomorphic to a 
disc, triangle, 
tetrahedral, spherical 
algebra, 
is a periodic algebra
of period $4$.

We recall briefly what we need from \cite{ESk-WSA}, for proofs and details
we refer to  \cite{ESk-WSA}.
Let $A = KQ/I$ be a bound quiver algebra, and let $A^e$ be the enveloping algebra.
Then  the $(e_i\otimes e_j)A^e= A(e_i\otimes e_j)A$
for $i,j \in Q_0$, form a complete set of pairwise non-isomorphic
indecomposable projective modules in $\mod A^e$
(see \cite[Proposition~IV.11.3]{SY}).

The following result by Happel \cite[Lemma~1.5]{Ha2} describes
the terms of a minimal projective resolution of $A$ in $\mod A^e$.

\begin{proposition}
\label{prop:5.1}
Let $A = K Q/I$ be a bound quiver algebra, where $Q$ is the Gabriel quiver of $A$.
Then there is in $\mod A^e$ a minimal projective resolution of $A$
of the form
\[
  \cdots \rightarrow
  \bP_n \xrightarrow{d_n}
\bP_{n-1} \xrightarrow{ }
  \cdots \rightarrow
  \bP_1 \xrightarrow{d_1}
  \bP_0 \xrightarrow{d_0}
  A \rightarrow 0,
\]
where
\[
  \bP_n = \bigoplus_{i,j \in Q_0}
	  A(e_i\otimes e_j)A^{\dim_K \Ext_A^n(S_i,S_j)}
\]
for any $n \in \bN$.
\end{proposition}

We will need details for the first  three differentials.
We have
\[
  \bP_0  = \bigoplus_{i \in Q_0} A e_i \otimes e_i A .
\]
The homomorphism $d_0 : \bP_0 \to A$ in $\mod A^e$ defined by
$d_0 (e_i \otimes e_i) = e_i$ for all $i \in Q_0$
is a minimal projective cover of $A$ in $\mod A^e$.
Recall that, for two vertices $i$ and $j$ in $Q$,
the number of arrows from $i$ to $j$ in $Q$ is equal
to $\dim_K \Ext_A^1(S_i,S_j)$
(see \cite[Lemma~III.2.12]{ASS}).
Hence we have
\[
  \bP_1
= \bigoplus_{\alpha \in Q_1} A e_{s(\alpha)} \otimes e_{t(\alpha)} A
        .
\]
Then we have the following known fact (see \cite[Lemma~3.3]{BES4}
for a proof).

\begin{lemma}
\label{lem:5.2}
Let $A = K Q/I$ be a bound quiver algebra, and
$d_1 : \bP_1 \to \bP_0$ the homomorphism in $\mod A^e$
defined by
\[
 d_1(e_{s(\alpha)} \otimes e_{t(\alpha)}) =
   \alpha \otimes e_{t(\alpha)} - e_{s(\alpha)} \otimes \alpha
\]
for any arrow $\alpha$ in $Q$.
Then $d_1$ induces a minimal projective cover
$d_1 : \bP_1 \to \Omega_{A^e}^1(A)$ of
$\Omega_{A^e}^1(A) = \Ker d_0$ in $\mod A^e$.
In particular, we have
$\Omega_{A^e}^2(A) \cong \Ker d_1$ in $\mod A^e$.
\end{lemma}

For the algebras $A$ we will consider, the kernel
$\Omega_{A^e}^2(A)$ of $d$ will be generated,
as an $A$-$A$-bimodule, by some elements of $\bP_1$
associated to a set of relations generating the
admissible ideal $I$.
Recall that a relation in the path algebra $KQ$
is an element of the form
\[
  \mu = \sum_{r=1}^n c_r \mu_r
  ,
\]
where $c_1, \dots, c_r$ are non-zero elements of $K$ and
$\mu_r = \alpha_1^{(r)} \alpha_2^{(r)} \dots \alpha_{m_r}^{(r)}$
are paths in $Q$ of length $m_r \geq 2$, $r \in \{1,\dots,n\}$,
having a common source and a common target.
The admissible ideal $I$ can be generated by a finite set
of relations in $K Q$ (see \cite[Corollary~II.2.9]{ASS}).
In particular, the bound quiver algebra $A = K Q/I$ is given
by the path algebra $K Q$ and a finite number of identities
$\sum_{r=1}^n c_r \mu_r = 0$ given by a finite set of generators of
the ideal $I$.
Consider the $K$-linear homomorphism $\varrho : K Q \to \bP_1$
which assigns to a path $\alpha_1 \alpha_2 \dots \alpha_m$ in $Q$
the element
\[
  \varrho(\alpha_1 \alpha_2 \dots \alpha_m)
   = \sum_{k=1}^m \alpha_1 \alpha_2 \dots \alpha_{k-1}
                  \otimes \alpha_{k+1} \dots \alpha_m
\]
in $\bP_1$, where $\alpha_0 = e_{s(\alpha_1)}$
and $\alpha_{m+1} = e_{t(\alpha_m)}$.
Observe that
$\varrho(\alpha_1 \alpha_2 \dots \alpha_m) \in e_{s(\alpha_1)} \bP_1 e_{t(\alpha_m)}$.
Then, for a relation $\mu = \sum_{r=1}^n c_r \mu_r$
in $K Q$ lying in $I$, we have an element
\[
  \varrho(\mu) = \sum_{r=1}^n c_r \varrho(\mu_r) \in e_i \bP_1 e_j ,
\]
where $i$ is the common source and $j$ is the common
target of the paths $\mu_1,\dots,\mu_r$.
The following lemma shows that relations always
produce elements in the kernel of $d_1$;  the proof
is straightforward.

\begin{lemma}
\label{lem:5.3}
Let $A = K Q/I$ be a bound quiver algebra and
$d_1 : \bP_1 \to \bP_0$ the homomorphism in $\mod A^e$
defined in Lemma~\ref{lem:5.2}.
Then for any relation $\mu$ in $K Q$ lying in $I$,
we have $d_1(\varrho(\mu)) = 0$.
\end{lemma}

For an algebra $A = K Q/I$ in our context, we will define 
a family of relations $\mu^{(1)},\dots,\mu^{(q)}$
such that the associated elements
$\varrho(\mu^{(1)}), \dots, \varrho(\mu^{(q)})$ generate
the $A$-$A$-bimodule $\Omega_{A^e}^2(A) = \Ker d_1$.
In fact, using Lemma~\ref{lem:5.3},
we will  show that
\[
  \bP_2 = \bigoplus_{j = 1}^q A e_{s(\mu^{(j)})} \otimes e_{t(\mu^{(j)})} A
        ,
\]
and the homomorphism $d_2 : \bP_2 \to \bP_1$ in $\mod A^e$ such that
\[
  d_2 \big(e_{s(\mu^{(j)})} \otimes e_{t(\mu^{(j)})}\big) = \varrho(\mu^{(j)})
  ,
\]
for $j \in \{1,\dots,q\}$, defines a projective cover
of $\Omega_{A^e}^2(A)$ in $\mod A^e$.
In particular, we have $\Omega_{A^e}^3(A) \cong \Ker d_2$ in $\mod A^e$.
We will denote  this homomorphism $d_2$ by $R$.
The differential  $S:=d_3 : \bP_3 \to \bP_2$ will be defined later.

\bigskip

Now we fix 
a weighted surface algebra $\Lambda = \Lambda(Q,f,m_{\bullet},c_{\bullet})$
for a triangulation quiver $(Q,f)$
with at least  two vertices,
a weight function
$m_{\bullet}$
and
a parameter function
$c_{\bullet}$.
Moreover, we assume that 
$\Lambda$ is not a (non-singular) tetrahedral, or disc, or triangle, or spherical 
algebra (they are already dealt with in Proposition \ref{prop:3.8}).

We fix a vertex $i$ of $Q$, and we show that the simple module
$S_i$ is periodic of period four. 
Let $\alpha$ and $\ba$ be the arrows starting at $i$.

\begin{proposition}
\label{prop:5.4}
Assume that the arrows $\alpha, \ba$ are not virtual.
Then there is an exact sequence in $\mod \Lambda$
\[
  0 \rightarrow
  S_i \rightarrow
  P_i \xrightarrow{\pi_3}
  P_{t(f(\alpha))} \oplus P_{t(f(\bar{\alpha}))} \xrightarrow{\pi_2}
  P_{t(\alpha)} \oplus P_{t(\bar{\alpha})} \xrightarrow{\pi_1}
  P_i \rightarrow
  S_i \rightarrow
  0,
\]
which give rise to a minimal projective resolution of $S_i$ in $\mod \Lambda$.
In particular, $S_i$ is a periodic module of period $4$.
\end{proposition}

If the arrows 
$\alpha$ and $\ba$ 
are both not virtual then this is
Proposition~7.1 of \cite{ESk-WSA}.

%\bigskip

\medskip

Now assume that $\ba$ is  a virtual loop, then $\alpha$ is not virtual.
Note that by Assumption \ref{ass} we have $m_{\alpha}n_{\alpha}\geq 4$. 
The quiver $Q$ contains a subquiver
\[
 \xymatrix{
  i \ar@(dl,ul)[]^{\ba} \ar@<+.5ex>[r]^{\alpha}
   & j \ar@<+.5ex>[l]^{f(\alpha)}  
 }
\]
and $f$ has a cycle $(\ba \ \alpha \ f(\alpha))$. Let $\gamma$ be the other arrow 
starting at vertex $j$, and $\delta$ be the other arrow ending at $j$. 

\begin{lemma}\label{lem:5.5} There is an exact sequence of $\La$-modules
	$$0 \to \Omega^{-1}(S_i) \to P_j \to P_j \to  \Omega(S_i)\to 0
	$$
	which gives rise to a periodic minimal projective resolution of $S_i$ 
	in ${\rm mod}\La$. In particular $S_i$ is periodic of period $4$.
\end{lemma}

\begin{proof}
	We have $\Omega(S_i) = \alpha \La$, and we take $\Omega^2(S_i)$ as 
$$
  \Omega^2(S_i) = \{ x\in e_j\La \mid \alpha x = 0\}.
$$
We have the following relations in $\La$:
\begin{enumerate}[(i)]
 \item
    $\alpha f(\alpha) = c_{\ba} \ba$, 
 \item
    $\ba \alpha = c_{\alpha}A_{\alpha}$.
\end{enumerate}
Hence 
$c_{\alpha}A_{\alpha} = \ba\alpha  = c_{\ba}^{-1} \alpha f(\alpha)\alpha$ 
and if we set
$$\vf:= f(\alpha)\alpha  - c_{\ba}c_{\alpha} A_{\alpha}'$$
(where $\alpha A_{\alpha}' = A_{\alpha}$), 
then $\vf \La \subseteq \Omega^2(S_i)$.
The module $\Omega^2(S_i)$ has dimension $m_{\alpha}n_{\alpha}-1$. 
We will now show that 
$\vf \La$ has the same dimension which will give equality.

\smallskip

(1) \ First we observe that $\vf f(\alpha)=0$. Namely
$$\vf f(\alpha) = f(\alpha)\alpha f(\alpha) - c_{\alpha}c_{\ba}A_{\alpha}'f(\alpha)
= f(\alpha)c_{\ba}\ba - c_{\alpha}c_{\ba}A_{\gamma} =0$$
since $f(\alpha)\ba = c_{\gamma}A_{\gamma}$ and $c_{\alpha}=c_{\gamma}$. 
Hence $\vf \rad \La$ is generated by $\vf \gamma$. 

\smallskip

(2) \ We show that $A_{\alpha}'\gamma$ lies in the socle. We have
$\gamma = f(\delta)$ and $A_{\alpha}'$ has length $\geq 2$. 
Therefore  $A_{\alpha}'\gamma = A_{\alpha}'' g^{-1}(\delta)\delta f(\delta)$.
The product of the last three arrows is zero unless possibly 
$f(g^{-1}(\delta))$ is virtual, and if so then it lies in the second socle,
by Lemma~\ref{lem:4.5}~(iii). 
Moreover, in this case,
$A_{\alpha}''$ is in the radical,
because $\Lambda$ is not the triangle
algebra considered in Example~\ref{ex:3.4}.
Hence, in any case 
$A_{\alpha}'\gamma$ lies in the socle. 

\smallskip

(3) \ We can now compute the dimension of $\vf \La$. 
By (1) and (2) the radical of $\vf \La$ is generated by
$\vf \gamma = f(\alpha) \alpha\gamma + u$, for an element $u$ in the socle. 
Now $f(\alpha)\alpha\gamma$ is a monomial along $B_{f(\alpha)}$ 
which has length $\geq 4$ and hence
$f(\alpha)\alpha \gamma$ is not in the socle. 
It follows that $\vf \La$ has basis
$$\{ \vf, \ \vf\gamma, \ \vf\gamma g(\gamma) = f(\alpha)\alpha\gamma g(\gamma), \ldots, B_{f(\alpha)}\}
$$
of size $m_{\alpha}n_{\alpha}-1$. 
Hence $\vf \La = \Omega^2(S_i)$. Note that this is not simple.

\smallskip

(4) \ We identify $\Omega^3(S_i)$ with $\{ x\in e_j\La\mid \vf x=0\}$, 
which has dimension $m_{\alpha}n_{\alpha}-1$.
By (1) we know that this contains $f(\alpha)\La$. 
Moreover $f(\alpha)\La$ is isomorphic to $\Omega^{-1}(S_i)$ 
which has the same dimension.
Hence we have $\Omega^3(S_i)\cong \Omega^{-1}(S_i)$ and $S_i$ has period $4$.
%$\Box$
\end{proof}

Now assume $\ba$ is virtual but not a loop. Then $\alpha$ is not virtual,
and it cannot be a loop (see \ref{rem:4.2} and \ref{lem:4.3}).
We have the following diagram:
\[
%  \xymatrix@R=2pc@C=1.5pc{
%  \xymatrix@R=3.5pc@C=1.8pc{
  \xymatrix@R=3.pc@C=1.8pc{
%  \xymatrix@C=.8pc{
    & j
    \ar[rd]^{f(\alpha)}
    \\
    i
    \ar[ru]^{\alpha}
    \ar@<-.5ex>[rr]_{\ba}
    && \bullet
    \ar@<-.5ex>[ll]_{f^2(\alpha)}
 \ar[ld]^{f(\ba)}
    \\
    & y
    \ar[lu]^{f^2(\ba)}
  }
\]

\medskip

\begin{lemma}\label{lem:5.6}
There is an exact sequence of $\La$-modules
	$$0\to \Omega^{-1}(S_i)\to P_y \to P_j \to \Omega(S_i)\to 0
	$$
	which gives rise to a minimal projective resolution of period $4$.
\end{lemma}

%\bigskip

\begin{proof} 
	We identify $\Omega(S_i) = \alpha \La$ and then 
$\Omega^2(S_i) = \{ x\in e_j\La\mid \alpha x = 0\}$. 
We have the following relations in $\La$:
\begin{enumerate}[(i)]
 \item
    $\alpha f(\alpha) = c_{\ba} \ba$,
 \item
    $\ba f(\ba) = c_{\alpha}A_{\alpha}$.
\end{enumerate}
Hence 
$c_{\alpha}A_{\alpha} = \ba f(\ba) 
 = c_{\ba}^{-1} \alpha f(\alpha)f(\ba)$ 
and if we set
$$\vf:= f(\alpha)f(\ba)   - c_{\ba}c_{\alpha} A_{\alpha}'$$
(where $\alpha A_{\alpha}' = A_{\alpha}$), 
then $\vf \La \subseteq \Omega^2(S_i)$.

The module $\Omega^2(S_i)$ has dimension 
$m_{f(\alpha)}n_{f(\alpha)}-1$. 
We want to show that $\vf \La$ has the same dimension.
Assume first that  $j\neq y$. 
Let $\gamma = g(f(\bar{\alpha}))$ and $\delta = f^{-1}(\gamma)$.

\smallskip

	(1) \ First we observe that $\vf f^2(\ba)=0$: Namely
	$$\vf f^2(\ba) = f(\alpha)f(\ba) f^2(\ba) - c_{\alpha}c_{\ba}A_{\alpha}'f^2(\ba).$$
The first term is 
	$$f(\alpha) c_{f^2(\alpha)}A_{f^2(\alpha)} = 
	c_{f^2(\alpha)}f(\alpha)f^2(\alpha) = c_{f^2(\alpha)}c_{g(\alpha)}A_{g(\alpha)}$$
	and the second term is the same since $A_{\alpha}'f^2(\ba)= A_{g(\alpha)}$
and $c_{\alpha}=c_{g(\alpha)}$, and $c_{f^2(\alpha)} = c_{\ba}$. 
So we get zero.
Hence $\vf \rad \La$ is generated by $\vf \gamma$.

\smallskip

(2) \ We analyse  $A_{\alpha}'\gamma$ and compute the dimension of
	$\vf\La$.

\smallskip
	
(2a) \ Assume $m_{\alpha}n_{\alpha}=3$. 
Then we claim that $A_{\alpha}'\gamma$ is in the second socle, 
and moreover $m_{f(\alpha)}n_{f(\alpha)} \geq 5$. 

	With this assumption, $A_{\alpha} = \alpha\delta$ where $\delta=g(\alpha)$, 
         so $A_{\alpha}'\gamma = \delta\gamma = c_{f(\alpha)}A_{f(\alpha)}$ 
	which is in the second socle. Consider the $g$-orbit of $f(\alpha)$, 
	it is $(f(\alpha) \ f(\ba) \ \gamma \  \ldots)$ and the last
	arrow in this orbit is the arrow $g^{-1}(f(\alpha)) \neq \alpha$ ending at $j$.

	We know that $\gamma$ cannot end at $j$ since then 
	the $f$-cycle of $g(\alpha)$ would not have length three. 
	Let $z=t(\gamma)$. Since $\gamma = f(g(\alpha))$, the arrow $f(\gamma)$ is an  arrow
	$z\to j$ and this is then the last arrow in the $g$-cycle in question. 
	This cannot be $g(\gamma)$ since it is $f(\gamma)$. 
	At best, $g(\gamma)$ is a loop at $z$ and then the $g$-cycle of $f(\alpha)$ has length 5, and
	in general it has length $\geq 5$.

	In this case $\vf\gamma = f(\alpha)f(\ba)\gamma + v$, 
         for an element $v$ in the second socle. 
	We postmultiply by $g(\gamma)g^2(\gamma)$ and get
	$$f(\alpha)f(\ba)\gamma g(\gamma)g^2(\gamma) + 0$$
	and this is a non-zero monomial along $B_{f(\alpha)}$. 
	We deduce using also (1) that $\dim \vf\La = m_{f(\alpha)}n_{f(\alpha)} - 1$,  and 
	we are done in this case.

\smallskip

(2b)  
Assume $m_{\alpha}n_{\alpha}\geq 5$, then $A_{\alpha}'\gamma
= A_{\alpha}'' g^{-1}(\delta)\delta f(\delta) = 0$
since $A_{\alpha}''$ is in the square of the radical and the other factor
is at least in the second socle.

It follows that the dimension of $\vf\La$ is equal to 
$m_{f(\alpha)}n_{f(\alpha)}-1$
 as required.

\smallskip

(2c) 
Assume $m_{\alpha}n_{\alpha} = 4$. 
Consider $A_{\alpha}'\gamma$. 

In this case $A_{\alpha}' = g(\alpha)\delta$ and
$A_{\alpha}'\gamma = g(\alpha)\delta f(\delta)$.
This is zero if $f(g(\alpha))$ is not virtual. 
We are left with the case where
$f(g(\alpha))$ is virtual,  and then 
$g(\alpha)\delta f(\delta)$ is a scalar multiple of
$g(\alpha)\bar{\delta} = c_{f(\alpha)}A_{f(\alpha)}$.

So $\vf \gamma = f(\alpha)f(\ba)\gamma - \lambda A_{f(\alpha)}$ for
a non-zero scalar $\lambda$. 
Because $\Lambda$ is not a spherical algebra, we have
$m_{f(\alpha)}n_{f(\alpha)}\geq 5$, then it follows as before
that $\dim \vf\La = m_{f(\alpha)}n_{f(\alpha)} -1$.

\smallskip

(3) By (1) we know $\Omega^3(S_i)$ contains $f^2(\ba)\La$ 
and this is isomorphic to $\Omega^{-1}(S_i)$. 
One sees that they have the same dimension and hence
are equal. 
This completes the proof in the case $i\neq j$.

\smallskip

Assume now that $i=j$, then $Q$ has three vertices.
Then $g$ is the product of three $2$-cycles, namely
$$g = \big(\alpha \ f^2(\ba)\big)\big(f(\alpha) \ f(\ba)\big)\big(\ba \ f^2(\alpha)\big).$$ 
By Assumption~\ref{ass}, we have $m_{\alpha}\geq 2$
and $m_{f(\alpha)}\geq 2$. 
Moreover, since $\Lambda$ is not a triangle algebra, we have
$m_{\alpha}+m_{f(\alpha)}\geq 5$. 
One sees similarly as before that $\vf \La$ has dimension$m_{f(\alpha)}n_{f(\alpha)}-1$, 
and then it follows again that $S_i$ has period four.
\end{proof}

%\bigskip

To identify the projective $\mathbb{P}_2$
of a minimal bimodule resolution, we need $\Ext^2_{\La}(S_i, S_j)$ for simple modules $S_i, S_j$.
	
\begin{lemma}\label{lem:5.7}
The dimension of $\Ext_{\Lambda}^2(S_i, S_j)$ is equal to the number of arrows
$j\to i$ in the Gabriel quiver of $\La$.
\end{lemma}

%\bigskip

\begin{proof} 
This follows from the calculation of
syzygies of the simple modules.
\end{proof}

%\bigskip

Now we construct the first steps of a minimal
projective bimodule resolution of $\Lambda$.
Then we will show that
$\Omega_{\Lambda}^4(\Lambda) \cong \Lambda$
in $\mod \Lambda^e$.
We shall use the notation introduced in earlier in this section.
Recall the first few steps of a minimal projective resolution
of $\Lambda$ in $\mod \Lambda^e$,
\[
  \bP_3 \xrightarrow{S}
  \bP_2 \xrightarrow{R}
  \bP_1 \xrightarrow{d}
  \bP_0 \xrightarrow{d_0}
  \Lambda \to 0
\]
where

\begin{align*}
  \bP_0
     &= \bigoplus_{i \in Q_0} \Lambda e_i \otimes e_i \Lambda ,
  \\
  \bP_1
        &= \bigoplus_{\alpha \in (Q_{\La})_1} \Lambda e_{s(\alpha)} \otimes e_{t(\alpha)} \Lambda ,
\end{align*}
the homomorphism $d_0$
is defined by
$d_0 ( e_i \otimes e_i ) = e_i$ for all $i \in Q_0$,
and the homomorphism $d : \bP_1 \to \bP_0$
is defined by
\[
  d \big( e_{s(\alpha)} \otimes e_{t(\alpha)} \big)
    = \alpha \otimes e_{t(\alpha)} - e_{s(\alpha)} \otimes \alpha
\]
for any arrow $\alpha$ of the Gabriel quiver $Q_{\La}$.
In particular, we have
$\Omega_{\Lambda}^1(\Lambda) = \Ker d_0$
and
$\Omega_{\Lambda}^2(\Lambda) = \Ker d$.

\medskip

We define now the homomorphism $R : \bP_2 \to \bP_1$.
For each arrow $\alpha$ of $Q$ which is not virtual,
that is, it is an arrow of the Gabriel quiver
$Q_{\La}$, consider the  element in $K Q$
\[
  \mu_{\alpha}: = \ba f(\ba) - c_{\alpha} A_{\alpha} .
\]
Since we work with the Gabriel quiver, we must make substitutions.
Note first that since $\alpha$ is not virtual, no arrow in the $g$-cycle
of $\alpha$ is virtual and therefore $A_{\alpha}$ is a path in $Q_{\La}$.
If
$\ba$ is virtual then we substitute
$\ba = \alpha f(\alpha)$ (using that $c_{\ba}=1$ by assumption).
Note that
if $\ba$ is virtual then $f(\alpha)$ is not virtual (see \ref{rem:4.2} and
\ref{lem:4.3}). 
Similary we substitute  $f(\alpha) = g(\alpha)f(g(\alpha))$ if $f(\alpha)$ is
virtual. \ Recall that $\alpha, f(\alpha)$ cannot be both virtual.

Note also that $\mu_{\alpha} = e_{s(\ba)} \mu_{\alpha} e_{t(f(\ba))}$.
It follows from Proposition \ref{prop:5.1} and Lemma \ref{lem:5.7}
that $\bP_2$ is of the form
\[
  \bP_2
      = \bigoplus_{\alpha \in (Q_{\La})_1} 
       \Lambda e_{s(\ba)} \otimes e_{t(f(\ba))} \Lambda .
\]
We define the homomorphism $R : \bP_2 \to \bP_1$
in $\mod \Lambda^e$ by
\[
  R\big( e_{s(\ba)} \otimes e_{t(f(\ba))}\big) := \varrho (\mu_{\alpha})
\]
for any arrow $\alpha$ of the Gabriel quiver of $\La$,
where $\varrho : K Q \to \bP_1$
is the $K$-linear homomorphism defined 
%in \ref{gen:per}.
earlier this section.

It follows from Lemma~\ref{lem:5.3} that
$\Im R \subseteq \Ker d$.

\begin{lemma}
\label{lem:5.8}
The homomorphism $R : \bP_2 \to \bP_1$
induces a projective cover
$\Omega_{\Lambda^e}^2(\Lambda)$
in $\mod \Lambda^e$.
In particular, we have $\Omega_{\Lambda^e}^3(\Lambda) = \Ker R$.
\end{lemma}

\begin{proof}
         We know that
$\rad \Lambda^e
 = \rad \Lambda^{\op} \otimes \Lambda
   + \Lambda^{\op} \otimes \rad \Lambda$
(see \cite[Corollary~IV.11.4]{SY}).
It follows from the definition that the generators
$\varrho(\mu_{\alpha})$
of the image of $R$ are elements of $\rad \bP_1$
which are linearly independent in
        $\rad \bP_1 / \rad^2 \bP_1$, provided
        both $\alpha, \ba$ are not virtual.
        Suppose (say) $\ba$ is virtual, then we consider 
	$S_i\otimes_{\La} \varrho(\mu_{\alpha})$. 
	This is precisely the generator of the module $\Omega^2(S_i)$
	as constructed in Lemmas \ref{lem:5.5} and \ref{lem:5.6}.
	Therefore $\varrho(\mu_{\alpha})$ is a generator for the image
        of $R$. 
We  conclude that
$\varrho(\mu_{\alpha})$, $\alpha \in Q_1$,
form a minimal set of generators of the right
$\Lambda^e$-module
$\Omega_{\Lambda^e}^2(\Lambda)$.
Summing up, we obtain that $R : \bP_2 \to \Omega_{\Lambda^e}^2(\Lambda)$
is a projective cover of
$\Omega_{\Lambda^e}^2(\Lambda)$
in $\mod \Lambda^e$.
\end{proof}

By Proposition \ref{prop:5.1}, and the result that simple modules
have $\Omega$-period four, 
we have that $\bP_3$ is of the form
\[
  \bP_3 = \bigoplus_{i \in Q_0} \Lambda e_i \otimes e_i \Lambda .
\]
For each vertex $i \in Q_0$, we define
an element in $\bP_2$. If both arrows $\alpha, \ba$ starting at $i$ are not
virtual (so that also the arrows $f^2(\alpha), f^2(\ba)$ ending at $i$
are not virtual) then we define (as in \cite{ESk-WSA})
\begin{align*}
  \psi_i
   &= \big( e_i \otimes e_{t(f(\alpha))} \big) f^2(\alpha)
     + \big( e_i \otimes e_{t(f(\bar{\alpha}))} \big) f^2(\bar{\alpha})
     - \alpha \big( e_{t(\alpha)} \otimes e_i \big)
     - \bar{\alpha} \big( e_{t(\bar{\alpha})} \otimes e_i \big)
  \\&
    = \big(e_{s(\alpha)} \otimes e_{t(f(\alpha))}\big) f^2(\alpha)
      + \big(e_{s(\bar{\alpha})} \otimes e_{t(f(\bar{\alpha}))}\big) f^2(\bar{\alpha})
      - \alpha \big(e_{s(f(\alpha))} \otimes e_{t(f^2(\alpha))}\big)
  \\& \quad\,
- \bar{\alpha} \big(e_{s(f(\bar{\alpha}))} \otimes e_{t(f^2(\bar{\alpha}))}\big)
     .
\end{align*}
Suppose $\ba$ is virtual and then also 
$f^2(\alpha)$ is virtual, and
$\alpha, f^2(\ba)$ are not virtual. 
In this case, we take the
same formula but omit the terms 
which have virtual arrows
(and idempotents which do not occur in $\bP_2$). 
That is we define
\begin{align*}
        \psi_i &=
      \big( e_i \otimes e_{t(f(\bar{\alpha}))} \big) f^2(\bar{\alpha})
     - \alpha \big( e_{t(\alpha)} \otimes e_i \big)
     .
\end{align*}
Similarly, if $\alpha$ is virtual and then $f^2(\ba)$ is virtual, and
then $\ba, f^2(\alpha)$ are not virtual, we define
\begin{align*} \psi_i
  &= \big( e_i \otimes e_{t(f(\alpha))} \big) f^2(\alpha)
     - \bar{\alpha} \big( e_{t(\bar{\alpha})} \otimes e_i \big).
\end{align*}

We  define now a  homomorphism $S : \bP_3 \to \bP_2$
in $\mod \Lambda^e$ by
\[
  S( e_i \otimes e_i ) = \psi_i
\]
for any vertex $i \in Q_0$.

\begin{lemma}
	\label{lem:5.9}
The homomorphism $S : \bP_3 \to \bP_2$
induces a projective cover of
$\Omega_{\Lambda^e}^3(\Lambda)$
in $\mod \Lambda^e$.
In particular, we have
$\Omega_{\Lambda^e}^4(\Lambda) = \Ker S$.
\end{lemma}

\begin{proof}
We will prove first that
$R(\psi_i) = 0$ for any $i \in Q_0$.
Fix a vertex $i \in Q_0$.
If both arrows starting at $i$ are not virtual, this is identical
        with the calculation in \cite{ESk-WSA}, we will not repeat this
        (note that only arrows in $g$-orbits occur which are not virtual).
Suppose now  that $\alpha$ is not virtual and $\ba$ is virtual.
Then we have, in $\bP_1$, that
\begin{align*}
  R(\psi_i)
   &=  \varrho(\mu_{\bar{\alpha}}) f^2(\bar{\alpha})
        - \alpha \varrho \big(\mu_{f(\alpha)}\big)
   \\&
        = \varrho\big(\alpha f(\alpha)f(\ba)\big)f^2(\ba)
        - \alpha\varrho\big(f(\alpha)f(\ba)f^2(\ba)\big)\\
        &
        + c_{\alpha}\big(-\varrho(A_{\alpha})f^2(\ba)
 + \alpha\varrho(A_{g(\alpha)})\big)
\\ &
       = e_i\otimes f(\alpha) f(\ba)f^2(\ba) - \alpha f(\alpha)f(\ba)\otimes e_i
        + c_{\alpha}(-e_i\otimes A_{g(\alpha)} + A_{\alpha}\otimes e_i)\\
& =0 ,
\end{align*}
because 
$f^2(\ba) = g^{n_{\alpha}-1}(\alpha)$,
$f^2(\alpha) = g^{n_{\ba}-1}(\ba)$,
$f(\alpha)f(\ba)f^2(\ba) = f(\alpha)f^2(\alpha) = c_{g(\alpha)}A_{g(\alpha)}$ 
and
$\alpha f(\alpha)f(\ba) = \ba f(\ba) = c_{\alpha}A_{\alpha}$.
Similarly one shows that
$R(\psi_i)=0$ 
if $\alpha$  is virtual and $\bar{\alpha}$ not virtual.
Hence  $\Im S \subseteq \Ker R$.
Further, it follows from the definition
that the generators $\psi_i$, $i \in Q_0$,
of the image of $S$ are elements of $\rad \bP_2$
which are linearly independent in
$\rad \bP_2 / \rad^2 \bP_2$.
We conclude from the form of $\bP_2$ that
these elements form a minimal set of generators of
$\Ker R = \Omega_{\Lambda^e}^3(\Lambda)$.
Hence $S : \bP_3 \to \Omega_{\Lambda^e}^3(\Lambda)$
is a projective cover of $\Omega_{\Lambda^e}^3(\Lambda)$
in $\mod \Lambda^e$.
\end{proof}

\begin{theorem}
\label{th:5.10}
There is an isomorphism
$\Omega_{\Lambda^e}^4(\Lambda) \cong \Lambda$
in $\mod \Lambda^e$.
In particular, $\Lambda$ is a periodic algebra
of period $4$.
\end{theorem}

\begin{proof}
This  is very similar to the proof of \cite[Theorem~5.9]{ESk2}.
	Let $\vf$ be the symmetrizing $K$-linear 
	form as defined in Proposition \ref{prop:4.13}. 
Then, by general theory, we have
the 
symmetrizing bilinear form
$(-,-) : \Lambda \times \Lambda \to K$ such that
$(x,y) = \varphi(x y)$ for any $x,y \in \Lambda$.
Observe that, for any elements $x \in \cB_i$ and $y \in \cB$,
we have
\[
  (x,y) = \mbox{ the coefficient of $\omega_i$ in $x y$},
\]
when $x y$ is expressed
as a linear combination of the elements of $e_i \cB$ over $K$.
Consider also the dual basis
$\cB^* = \{ b^* \,|\, b \in \cB \}$ of $\Lambda$
such that
$(b,c^*) = \delta_{b c}$ for $b,c \in \cB$.
Observe that, for $x \in e_i \cB$ and $y \in \cB$, the element
$(x,y)$ can only be non-zero if $y = y e_i$.
In particular, if $b \in e_i \cB e_j$ then
$b^* \in e_j \cB e_i$.

For each vertex $i \in Q_0$, we define the element of $\bP_3$
\[
  \xi_i = \sum_{b \in \cB_i} b \otimes b^* .
\]
We note that $\xi_i$ is independent of the basis
of $\Lambda$
(see \cite[part (2a) on the page 119]{ESk2}).
It follows from \cite[part (2b) on the page 119]{ESk2}
that, for any element
$a \in e_i (\rad \Lambda) e_j \setminus e_i (\rad \Lambda)^2 e_j$,
we have
\[
  a \xi_i = \xi_j a .
\]
Consider now the homomorphism
\[
  \theta : \Lambda \to \bP_3
\]
in $\mod \Lambda^e$
such that
$\theta (e_i) = \xi_i$
for any $i \in Q_0$.
Then
$\theta (1_{\Lambda}) = \sum_{i \in Q_0} \xi_i$,
and consequently we have
\[
  a \Big( \sum_{i \in Q_0} \xi_i \Big)
  = \theta(a)
  = \Big( \sum_{i \in Q_0} \xi_i \Big) a
\]
for any element $a \in \Lambda$.
We claim that $\theta$ is a monomorphism.
It is enough to show that $\theta$ is a monomorphism
of right $\Lambda$-modules.
We know that $\Lambda = \bigoplus_{i \in Q_0} e_i \Lambda$
and each $e_i \Lambda$ has simple socle
generated by $\omega_i$.
For each $i \in Q_0$, we have
\begin{align*}
  \theta(\omega_i)
     &= \Big( \sum_{j \in Q_0} \xi_j \Big) \omega_i
     = \xi_i \omega_i
     = \sum_{b \in \cB_i} ( b \otimes b^* ) \omega_i
%  \\&
     = \sum_{b \in \cB_i} b \otimes b^* \omega_i
     = \omega_i \otimes \omega_i
     \neq 0
    .
\end{align*}
Hence the claim follows.
Our next aim is to show that
$S(\xi_i) = 0$ for any $i \in Q_0$,
or equivalently, that
$\Im \theta \subseteq \Ker S = \Omega_{\Lambda^e}^4(\Lambda)$.
Applying arguments from
\cite[part (3) on the pages 119 and 120]{ESk2},
we obtain that
\[
  \sum_{b \in \cB} b ( a^r \otimes a^s ) b^*
     = \sum_{b \in \cB} b \otimes a^{r+s} b^*
\]
for all integers $r,s \geq 0$ and any element
$a = e_p a e_q$ in $\rad \Lambda$, with $p,q \in Q_0$.
In particular, for each arrow $\alpha$ in $Q_1$, we have
\[
  \sum_{b \in \cB} b \alpha \otimes  b^*
     = \sum_{b \in \cB} b \otimes \alpha b^* ,
\]
and hence
\[
  \sum_{b \in \cB_i} b \alpha \otimes  b^*
     = \sum_{b \in \cB_i} b \otimes \alpha b^*
\]
for any $i \in Q_0$.
We note that every arrow $\beta$ in $Q$
occurs once as a left factor of some $\psi_j$
(with negative sign) and once a right factor
of some $\psi_k$ (with positive sign),
because $\beta = f^2 (\alpha)$ for a unique
arrow $\alpha$.
Then, for any $i \in Q_0$,
the following equalities hold
\begin{align*}
  S(\xi_i)
     &= \sum_{b \in \cB_i} S ( b \otimes b^* )
     = \sum_{b \in \cB_i} \sum_{j \in Q_0} S ( b e_j \otimes e_j b^* )
     = \sum_{b \in \cB_i} \sum_{j \in Q_0} b S ( e_j \otimes e_j ) b^*
  \\&
     = \sum_{b \in \cB_i} \sum_{j \in Q_0} b \psi_j b^*
     = \sum_{\alpha \in (Q_{\Lambda})_1} \Bigg[
             \sum_{b \in \cB_i} - ( b \alpha \otimes  b^* )
            + \sum_{b \in \cB_i} b \otimes \alpha b^*
        \Bigg]
     = 0
     .
%  \\&
\end{align*}
Hence, indeed
$\Im \theta \subseteq \Ker S = \Omega_{\Lambda^e}^4(\Lambda)$,
and we obtain a monomorphism
$\theta : \Lambda \to \Omega_{\Lambda^e}^4(\Lambda)$
in $\mod \Lambda^e$.

Finally, it follows from
Theorem~\ref{th:2.4}
and
Proposition~\ref{prop:5.4}
that
$\Omega_{\Lambda^e}^4(\Lambda) \cong {}_1 \Lambda_{\sigma}$
in $\mod \Lambda^e$
for some $K$-algebra automorphism $\sigma$ of $\Lambda$.
Then
$\dim_K \Lambda = \dim_K \Omega_{\Lambda^e}^4(\Lambda)$,
and consequently $\theta$ is an isomorphism.
Therefore, we have
$\Omega_{\Lambda^e}^4(\Lambda) \cong \Lambda$
in $\mod \Lambda^e$.
Clearly, then $\Lambda$ is a periodic algebra
of period $4$.
\end{proof}

\begin{corollary}
\label{cor:5.11}
Let $(Q,f)$ be a triangulation quiver with at least four vertices,
let $m_{\bullet}$ and $c_{\bullet}$ be weight and parameter
functions of $(Q,f)$, 
and let $\Lambda = \Lambda(Q,f,m_{\bullet},c_{\bullet})$
be the associated weighted triangulation algebra.
Then the Cartan matrix $C_{\Lambda}$ of $\Lambda$
is singular.
\end{corollary}

\begin{proof}
This follows from
Theorems \ref{th:2.5}
and \ref{th:5.10}.
\end{proof}

%\bigskip

\section{The representation type}\label{sec:reptype}

The aim of this section is to prove Theorem~\ref{th:main2}. 
We start by describing the general strategy. We are given
a weighted triangulation algebra 
$\Lambda = \Lambda(Q,f,m_{\bullet},c_{\bullet})$ of dimension $d$. 
We aim to define an algebraic
family of algebras $\La(t)$ for $t\in K$ in the variety ${\rm alg}_d(K)$
such that for every arrow $\alpha$ of $Q$ we have
\begin{enumerate}[(i)]
 \item
   $\alpha f(\alpha) = c_{\ba}t^{v(\alpha)} A_{\ba} $
   with $v(\alpha)$ a natural number $\geq 1$,
 \item
    $c_{\alpha}B_{\alpha} = c_{\ba}B_{\ba}$, and
 \item
  the zero relations as in Definition \ref{def:2.8} hold.
\end{enumerate}

Then the algebra $A(0)$ is the biserial weighted triangulation algebra
associated to $\La$. Theorem \ref{th:main2} for $\La$ will follow if we can
make sure that $A(t)\cong A(1)$ for any non-zero $t\in K$. 
We will define a map $\vf_t: A(1) \to A(t)$ such 
that $\vf_t(\alpha) = t^{u(\alpha)}\alpha \in A(t)$ where
$u(\alpha)\geq 1$ is a natural number, and extend to products and
linear combinations. 
This will define an algebra isomorphism 
$A(1)\to A(t)$ if and only if for all arrows $\alpha$ 
the following identity holds:
$$v(\alpha) + u(\alpha) + u(f(\alpha)) = u(A_{\ba})
\leqno{(\dagger)}$$
where we define $u(\mu) = \sum_{i=1}^r u(\alpha_i)$
for a monomial $\mu=\alpha_1\ldots \alpha_r$ in $KQ$. 
We deal first with some of the exceptions, in Proposition \ref{prop:6.4} it will be clear why these need to be treated separately.

\bigskip

Let $(Q,f)$ be the triangulation quiver
(as in Example \ref{ex:3.6})
\[
\begin{tikzpicture}
[->,scale=.9]
\coordinate (1) at (0,2);
\coordinate (2) at (-1,0);
\coordinate (2u) at (-.925,.15);
\coordinate (2d) at (-.925,-.15);
\coordinate (3) at (0,-2);
\coordinate (4) at (1,0);
\coordinate (4u) at (.925,.15);
\coordinate (4d) at (.925,-.15);
\coordinate (5) at (-3,0);
\coordinate (5u) at (-2.775,.15);
\coordinate (5d) at (-2.775,-.15);
\coordinate (6) at (3,0);
\coordinate (6u) at (2.775,.15);
\coordinate (6d) at (2.775,-.15);
\fill[fill=gray!20] (1) -- (5u) -- (2u) -- cycle;
\fill[fill=gray!20] (1) -- (4u) -- (6u) -- cycle;
\fill[fill=gray!20] (2d) -- (5d) -- (3) -- cycle;
\fill[fill=gray!20] (3) -- (6d) -- (4d) -- cycle;
\node [fill=white,circle,minimum size=4.5] (1) at (0,2) {\ \quad};
\node [fill=white,circle,minimum size=4.5] (2) at (-1,0) {\ \quad};
\node [fill=white,circle,minimum size=4.5] (3) at (0,-2) {\ \quad};
\node [fill=white,circle,minimum size=4.5] (4) at (1,0) {\ \quad};
\node [fill=white,circle,minimum size=4.5] (5) at (-3,0) {\ \quad};
\node [fill=white,circle,minimum size=4.5] (6) at (3,0) {\ \quad};
\node (1) at (0,2) {1};
\node (2) at (-1,0) {2};
\node (2u) at (-1,0.15) {\ \quad};
\node (2d) at (-1,-0.15) {\ \quad};
\node (3) at (0,-2) {3};
\node (4) at (1,0) {4};
\node (4u) at (1,0.15) {\ \quad};
\node (4d) at (1,-0.15) {\ \quad};
\node (5) at (-3,0) {5};
%%%\node (5u) at (-3,0.15) {\ \quad};
%%%\node (5d) at (-3,-0.15) {\ \quad};
\node (5u) at (-2.775,0.15) {};
\node (5d) at (-2.775,-0.15) {};
\node (6) at (3,0) {6};
\node (6u) at (2.775,0.15) {};
\node (6d) at (2.775,-0.15) {};
\draw[thick,->]
(1) edge node[below right]{\footnotesize$\alpha$} (2)
(2u) edge node[above]{\footnotesize$\xi$} (5u)
(5u) edge node[above left]{\footnotesize$\delta$} (1)
(5d) edge node[below]{\footnotesize$\eta$} (2d)
(2) edge node[above right]{\footnotesize$\beta$} (3)
(3) edge node[below left]{\footnotesize$\nu$} (5d)
(1) edge node[above right]{\footnotesize$\varrho$} (6u)
(6u) edge node[above]{\footnotesize$\varepsilon$} (4u)
(4) edge node[below left]{\footnotesize$\sigma$} (1)
(4d) edge node[below]{\footnotesize$\mu$} (6d)
(6d) edge node[below right]{\footnotesize$\omega$} (3)
(3) edge node[above left]{\footnotesize$\gamma$} (4)
;
\end{tikzpicture}
\]
where $f$ is the permutation of arrows of order $3$
described by the shaded subquivers.
Then $g$ has four orbits
\begin{align*}
 \cO(\alpha)  &= ( \alpha\ \beta\ \gamma\ \sigma), & 
 \cO(\varrho)  &= ( \varrho\ \omega\ \nu\ \delta), & 
 \cO(\xi)  &= ( \xi\ \eta), & 
 \cO(\varepsilon)  &= ( \varepsilon\ \mu).
\end{align*}
Let $r \geq 2$ be a natural number and
let $m_{\bullet}^r : \cO(g) \to \bN^*$
 be the weight function
given by 
$m_{\cO(\alpha)} = m_{\cO(\varrho)} = m_{\cO(\xi)} = 1$
and
$m_{\cO(\varepsilon)} = r$.
Moreover, let
$c_{\bullet} : \cO(g) \to K^*$
be an arbitrary parameter function.
We consider the weighted triangulation algebra  
\[
  S(r,c_{\bullet}) = \Lambda(Q,f,m_{\bullet}^r,c_{\bullet}).
\]

\begin{lemma}
\label{lem:6.1}
The algebra $S(r,c_{\bullet})$ degenerates to the  biserial
weighted triangulation algebra
$B(Q,f,m_{\bullet}^r,c_{\bullet})$.
In particular, $S(r,c_{\bullet})$ is a tame algebra.
\end{lemma}

\begin{proof}
We write $a = c_{\cO(\alpha)}$,
$b = c_{\cO(\varrho)}$,
$c = c_{\cO(\xi)}$,
$d = c_{\cO(\varepsilon)}$.
For each  $t \in K$,
consider the algebra $A(t)$
given by the quiver $Q$
and the relations:
\begin{align*}
 \alpha \xi &= b t^r \varrho \omega \nu , &
 \xi \delta &= a t^r \beta \gamma \sigma , &
 \delta \alpha &= c t^r \eta ,
 \\
 \beta \nu &= c t^r \xi , &
 \nu \eta &= a t^r \gamma \sigma \alpha , &
 \eta \beta &= b t^r \delta \varrho \omega ,
 \\
 \sigma \varrho &= d t^{3r-4} (\mu \varepsilon)^{r-1} \mu , &
 \varrho \varepsilon &= a t^{3r-4} \alpha \beta \gamma , 
&
 \varepsilon \sigma &= b t^{3r-4} \omega \nu \delta ,
 \\
 \gamma \mu &= b t^{3r-4} \nu \delta \varrho , &
 \mu \omega &= a t^{3r-4} \sigma \alpha \beta , &
 \omega \gamma &= d t^{3r-4} (\varepsilon \mu)^{r-1} \varepsilon ,
\\
 a \alpha \beta \gamma \sigma &= b \varrho \omega \nu \delta , &
 \!\!\!\!\!
 a \beta \gamma \sigma \alpha &= c \xi \eta , &
 a \gamma \sigma \alpha \beta &= b \nu \delta \varrho \omega , 
 \\
 a \sigma \alpha \beta \gamma &= d (\mu \varepsilon)^{r} , 
&
 \!\!\!\!\!
 b \delta \varrho \omega \nu &= c \eta \xi , &
 b \omega \nu \delta \varrho &= d (\varepsilon \mu)^r , 
\\
&&
   \!\!\!\!\!\!\!\!\!\!\!\!\!\!\!\theta f(\theta)g\big(f(\theta)\big) &= 0  
 \mbox{ for } \theta \in Q_1 \setminus \{\delta , \beta \}, 
 \!\!\!\!\!\!\!\!\!\!\!\!\!\!\!\!\!\!\!\!\!\!\!\!\!\!\!\!\!\!\!\!\!\!\!
 \\
&&
   \!\!\!\!\!\!\!\!\!\!\!\!\!\!\!\theta g(\theta)f\big(g(\theta)\big) &= 0 
 \mbox{ for } \theta \in Q_1 \setminus \{\alpha , \nu \}. 
 \!\!\!\!\!\!\!\!\!\!\!\!\!\!\!\!\!\!\!\!\!\!\!\!\!\!\!\!\!\!\!\!\!\!\!
\end{align*}
Note that $3r-4 \geq 1$ because $r \geq 2$.
Then $A(t)$, $t \in K$,
is an algebraic family in the variety $\alg_d(K)$, 
with $d = 4 r + 36$.
Observe also that 
$A(1) = S(r,c_{\bullet})$
and
$A(0) = B(Q,f,m_{\bullet}^r,c_{\bullet})$.
Fix $t \in K \setminus \{ 0 \}$.
Then there exists an isomorphism of $K$-algebras
$\varphi_t : A(1) \to A(t)$ given by
\begin{align*}
 \varphi_t(\alpha) &= t^r \alpha , &
 \varphi_t(\beta) &= t^{2r} \beta , &
 \varphi_t(\gamma) &= t^r \gamma , &
 \varphi_t(\sigma) &= t^{4r} \sigma , 
\\
 \varphi_t(\varrho) &= t^r \varrho , &
 \varphi_t(\omega) &= t^{4r} \omega , &
 \varphi_t(\nu) &= t^r \nu , &
 \varphi_t(\delta) &= t^{2r} \delta , 
\\
 \varphi_t(\xi) &= t^{4r} \xi , &
 \varphi_t(\eta) &= t^{4r} \eta , &
 \varphi_t(\varepsilon) &= t^{4} \varepsilon , &
 \varphi_t(\mu) &= t^{4} \mu .
\end{align*}
Therefore, applying 
Proposition~\ref{prop:2.2},
we conclude that $S(r,c_{\bullet})$
degenerates to 
\linebreak
$B(Q,f,m_{\bullet}^r,c_{\bullet})$,
and $S(r,c_{\bullet})$ is a tame algebra.
\end{proof}

Let $(Q,f)$ be the triangulation quiver
\[
  \xymatrix{
%  \xymatrix@C=1pc{
    1
%    \ar `ld_u[] `_rd[]^{\alpha} []
    \ar@(ld,ul)^{\alpha}[]
    \ar@<.5ex>[r]^{\beta}
    & 2
    \ar@<.5ex>[l]^{\gamma}
%    \ar `ru_d[] `_lu[]^{\eta} [] &
    \ar@<.5ex>[r]^{\sigma}
    & 3
    \ar@<.5ex>[l]^{\delta}
%    \ar `ru_d[] `_lu[]^{\xi} [] &
    \ar@(ru,dr)^{\eta}[]
  } 
%,
\]
(as in Example \ref{ex:3.4}) with $f$-orbits 
$(\alpha\ \beta\ \gamma)$
and
$(\eta\ \delta\ \sigma)$.
Then $g$ has orbits 
\begin{align*}
 \cO(\alpha)  &= (\alpha), & 
 \cO(\beta)  &= (\beta\ \sigma\ \delta\ \gamma), & 
 \cO(\eta)  &= ( \eta ).
\end{align*}
Let $r \geq 3$ be a natural number and
$m_{\bullet}^r : \cO(g) \to \bN^*$
the weight function
given by 
$m_{\cO(\alpha)}^r = 2$, 
$m_{\cO(\beta)}^r = 1$
and
$m_{\cO(\eta)}^r = r$.
Moreover, let
$c_{\bullet} : \cO(g) \to K^*$
be an arbitrary parameter function.
We consider the weighted triangulation algebra
\[
  \Sigma(r,c_{\bullet}) = \Lambda(Q,f,m_{\bullet}^r,c_{\bullet}).
\]

\begin{lemma}
\label{lem:6.2}
$\Sigma(r,c_{\bullet})$ degenerates to the  biserial
weighted triangulation algebra
$B(Q,f,m_{\bullet}^r,c_{\bullet})$.
In particular, $\Sigma(r,c_{\bullet})$ is a tame algebra.
\end{lemma}

\begin{proof}
We abbreviate $a = c_{\cO(\alpha)}$,
$b = c_{\cO(\beta)}$ and $c = c_{\cO(\eta)}$.
We consider two cases.

\smallskip

(1)
Assume $r = 3$.
For each $t \in K$, we denote
by $\Lambda(t)$ the algebra
given by the quiver $Q$
and the relations:
\begin{align*}
  \alpha \beta &= b t \beta \sigma \delta ,&
  \beta \gamma &= a t \alpha ,&
  \gamma \alpha &= b t \sigma \delta \gamma ,
 \\
  \eta \delta &= b t \delta \gamma \beta ,&
	a\alpha^2 &= b(\beta\sigma\delta\gamma), &
	b(\gamma\beta\sigma\delta) &= b(\sigma\delta\gamma\beta), &
\\
	b(\delta\gamma\beta\sigma) &= c\eta^r,&
  \delta \sigma &= c t \eta^2 ,&
  \sigma \eta &= b t \gamma \beta \sigma, &
\\
	\theta f(\theta)g(f(\theta)) &= 0 \mbox{ for } \beta \neq \theta \in Q_1,
      \!\!\!\!\!\!\!\!
& 
&
         \theta g(\theta)f(g(\theta)) = 0 \mbox{ for } \gamma \neq \theta \in Q_1.
      \!\!\!\!\!\!\!\!\!\!\!\!\!\!\!\!\!\!\!\!
      \!\!\!\!\!\!\!\!\!\!\!\!\!\!\!\!\!
&
\end{align*}
Then $\Lambda(t)$, $t \in K$,
is an algebraic family in the variety $\alg_{21}(K)$, 
with $\Lambda(1) = \Sigma(3,c_{\bullet})$,
and 
$\Lambda(0) =  B(Q,f,m_{\bullet}^3,c_{\bullet})$.
Moreover, for each $t \in K \setminus \{0 \}$,
there exists an isomorphism of $K$-algebras
$\varphi_t : \Lambda(1) \to \Lambda(t)$ given by
\begin{align*}
 \varphi_t(\alpha) &= t^6 \alpha , &
 \varphi_t(\beta) &= t^{2} \beta , &
 \varphi_t(\gamma) &= t^3 \gamma , 
\\
 \varphi_t(\eta) &= t^{4} \eta , &
 \varphi_t(\delta) &= t^{4} \delta , &
 \varphi_t(\sigma) &= t^{3} \sigma . 
\end{align*}
Therefore, it follows from 
Proposition~\ref{prop:2.2}
that $\Sigma(3,c_{\bullet})$
degenerates to $B(Q,f,m_{\bullet}^3,c_{\bullet})$,
and $\Sigma(3,c_{\bullet})$ is  tame.

\smallskip

(2)
Assume $r \geq 4$.
For each $t \in K$, we denote
by $A(t)$ the algebra
given by the quiver $Q$
and the relations:
\begin{align*}
  \alpha \beta &= b t^r \beta \sigma \delta ,&
  \beta \gamma &= a t^r \alpha ,&
  \gamma \alpha &= b t^r \sigma \delta \gamma ,
 \\
  \eta \delta &= b t^{2r-6} \delta \gamma \beta ,&
%   &&
        a\alpha^2 &= b(\beta\sigma\delta\gamma), &
        b(\gamma\beta\sigma\delta) &= b(\sigma\delta\gamma\beta), &
 \\
        b(\delta\gamma\beta\sigma) &= c\eta^r,&
  \delta \sigma &= c t^{2r-6} \eta^{r-1} ,&
  \sigma \eta &= b t^{2r-6} \gamma \beta \sigma, &
 \\
        \theta f(\theta)g(f(\theta)) &= 0 \mbox{ for } \beta \neq \theta \in Q_1,
      \!\!\!\!\!\!\!\!
&& 
        \theta g(\theta)f(g(\theta)) = 0 \mbox{ for } \gamma \neq \theta \in Q_1.
      \!\!\!\!\!\!\!\!\!\!\!\!\!\!\!\!\!\!\!\!
      \!\!\!\!\!\!\!\!\!\!\!\!\!\!\!\!\!
\end{align*}
We note that $2r-6 \geq 1$, because $r \geq 4$.
Then $A(t)$, $t \in K$,
is an algebraic family in the variety $\alg_{d}(K)$, 
with $d = r+18$.
Observe also that 
$A(1) = \Sigma(r,c_{\bullet})$
and 
$A(0) =  B(Q,f,m_{\bullet}^r,c_{\bullet})$.
Further, for each $t \in K \setminus \{0 \}$,
there exists an isomorphism of $K$-algebras
$\psi_t : A(1) \to A(t)$ given by
\begin{align*}
 \psi_t(\alpha) &= t^{3r} \alpha , &
 \psi_t(\beta) &= t^{r} \beta , &
 \psi_t(\gamma) &= t^r \gamma ,
\\
 \psi_t(\eta) &= t^{6} \eta , &
 \psi_t(\delta) &= t^{2r} \delta , &
 \psi_t(\sigma) &= t^{2r} \sigma . 
\end{align*}
Therefore, applying
Proposition~\ref{prop:2.2},
we conclude that $\Sigma(r,c_{\bullet})$
degenerates to 
\linebreak
$B(Q,f,m_{\bullet}^r,c_{\bullet})$,
and $\Sigma(r,c_{\bullet})$ is  tame.
\end{proof}

Let $(Q,f)$ be the triangulation quiver (as in Example \ref{ex:3.1})
\[
  \xymatrix{
%  \xymatrix@C=1pc{
    1
    \ar@(ld,ul)^{\alpha}[]
    \ar@<.5ex>[r]^{\beta}
    & 2
    \ar@<.5ex>[l]^{\gamma}
    \ar@(ru,dr)^{\sigma}[]
  } 
%,
\]
with  $f$-orbits
$(\alpha \ \beta \ \gamma)$
and
$(\sigma)$, so that 
the  $g$-orbits are $(\alpha)$ and 
$(\beta \ \sigma \ \gamma)$.
We fix  a natural number $r \geq 4$, and we take
$m_{\bullet}$ to be 
the weight function  
$m_{\alpha} = r$
and
$m_{\beta}  = 1$.
Moreover, for 
$c_{\bullet}$ we take
an 
arbitrary parameter function.
We consider the weighted triangulation algebra 
\[
  \Omega(r,c_{\bullet}) = \Lambda(Q,f,m_{\bullet}^r,c_{\bullet}).
\]

\begin{lemma}
\label{lem:6.3} The algebra
$\Omega(r,c_{\bullet})$ degenerates to the biserial
weighted triangulation algebra
$B(Q,f,m_{\bullet}^r,c_{\bullet})$.
In particular, $\Omega(r,c_{\bullet})$ is  tame.
\end{lemma}

\begin{proof}
We abbreviate $a = c_{\alpha}$
and $b = c_{\beta}$.
We consider two cases.

\smallskip

(1)
Assume $r = 4$.
For each $t \in K$, we denote
by $\Lambda(t)$ the algebra
given by the quiver $Q$
and the relations:
\begin{align*}
 \sigma^2 &= b t \gamma \beta,
\!\!\!\!\!\!\!\!\!
 &
	\gamma\alpha  &= b t \sigma \gamma ,
 &
 \alpha \beta &= b t \beta \sigma,\!\!\!
 &
 \beta \gamma &= a t \alpha^3 ,
 \\
	&& 
	a\alpha^4 &= b(\beta\sigma\gamma), &
	b(\gamma\beta\sigma)&= b(\sigma\gamma\beta), \!\!\!\!\!\!&
\\
	&&
	\theta f(\theta)g(f(\theta))&=0  \ (\theta \in Q_1),&
	\theta g(\theta)f(g(\theta))&=0 \ (\theta \in Q_1).
\end{align*}
Then $\Lambda(t)$, $t \in K$,
is an algebraic family in the variety $\alg_{13}(K)$, 
with $\Lambda(1) = \Sigma(4,c_{\bullet})$
and 
$\Lambda(0) =  B(Q,f,m_{\bullet}^4,c_{\bullet})$.
Moreover, for each $t \in K^*$,
there exists an isomorphism of $K$-algebras
$\varphi_t : \Lambda(1) \to \Lambda(t)$ given by
\begin{align*}
 \varphi_t(\sigma) &= t^5 \sigma , &
 \varphi_t(\gamma) &= t^{5} \gamma , &
 \varphi_t(\beta) &= t^6 \beta , &
 \varphi_t(\alpha) &= t^{4} \alpha . 
\end{align*}
Therefore, it follows from 
Proposition~\ref{prop:2.2}
that $\Omega(4,c_{\bullet})$
degenerates to $B(Q,f,m_{\bullet}^4,c_{\bullet})$,
and $\Omega(4,c_{\bullet})$ is a tame algebra.

\smallskip

(2)
Assume $r \geq 5$.
For each $t \in K$, we denote
by $A(t)$ the algebra
given by the quiver $Q$
and the relations:
\begin{align*}
\sigma^2 &= b t^r \gamma \beta,
 &
	\gamma\alpha  &= b t^{r-4} \sigma \gamma ,
 &
	\alpha \beta &= b t^{r-4} \beta \sigma,
 \\
	\beta \gamma &= a t^{r-4} \alpha^{r-1} ,
 &
        a\alpha^r &= b(\beta\sigma\gamma), &
        b(\gamma\beta\sigma)&= b(\sigma\gamma\beta),
\\
        &&
\!\!\!\!\!\!\!\!\!\!\!\!
        \theta f(\theta)g(f(\theta))=0  &\ (\theta \in Q_1),
&
        \theta g(\theta)f(g(\theta))&=0 \ (\theta \in Q_1) .
\end{align*}
Then $A(t)$, $t \in K$,
is an algebraic family in the variety $\alg_{d}(K)$, 
with $d = 9+r$.
Observe also that 
$A(1) = \Omega(r,c_{\bullet})$
and 
$A(0) =  B(Q,f,m_{\bullet}^r,c_{\bullet})$.
Further, for each $t \in K \setminus \{0 \}$,
there exists an isomorphism of $K$-algebras
$\psi_t : A(1) \to A(t)$ given by
\begin{align*}
	\psi_t(\sigma) &= t^{r} \sigma , &
 \psi_t(\gamma) &= t^{r} \gamma , &
 \psi_t(\beta) &= t^{2r} \beta , &
 \psi_t(\alpha) &= t^{4} \alpha  . 
\end{align*}
Therefore, applying
Proposition~\ref{prop:2.2},
we conclude that $\Omega(r,c_{\bullet})$
degenerates to 
\linebreak
$B(Q,f,m_{\bullet}^r,c_{\bullet})$,
and $\Omega(r,c_{\bullet})$ is a tame algebra.
\end{proof}

%\bigskip

Let $(Q, f)$ be the triangulation quiver
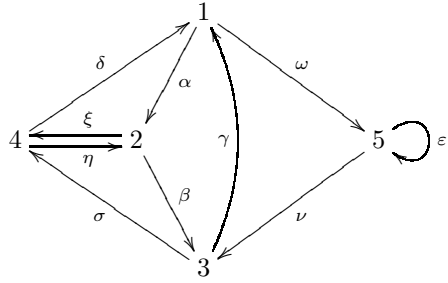
\[
%  \xymatrix@R=2pc@C=1.5pc{
%  \xymatrix@R=3.5pc@C=1.8pc{
  \xymatrix@R=3.pc@C=1.2pc{
%  \xymatrix@C=.8pc{
    &&& 1
    \ar[ld]^{\alpha}
    \ar[rrrd]^{\omega}
    \\   
    4
    \ar[rrru]^{\delta}
    \ar@<-.5ex>[rr]_(.6){\eta}
    && 2
    \ar@<-.5ex>[ll]_(.4){\xi}
    \ar[rd]^{\beta}
    &&&& 5  \ar@(ru,dr)^{\varepsilon}[]
    \ar[llld]^{\nu}
    \\
    &&& 3
    \ar[lllu]^{\sigma}
    \ar@/_3ex/[uu]^{\gamma}
  }
\]
with $f$-orbits 
$(\alpha\ \xi\ \delta)$,
$(\beta\ \sigma\ \eta)$, 
$(\gamma\ \omega\ \nu)$
and $(\varepsilon)$.
Then  $g$ has orbits 
\begin{align*}
  &&
  \cO(\alpha) &= (\alpha \ \beta \ \gamma), &
  \cO(\delta) &= (\delta \ \omega \ \varepsilon \ \nu \ \sigma),&
  \cO(\xi) &= (\xi \ \eta).
  &&
\end{align*}
Let $m_{\bullet} : \cO(g) \to \bN^*$
be the weight function
given by 
$m_{\cO(\alpha)} = m_{\cO(\delta)} = m_{\cO(\xi)} = 1$
and let 
$c_{\bullet} : \cO(g) \to K^*$
be an arbitrary parameter function.
We consider the weighted triangulation algebra  
\[
  \Phi(c_{\bullet}) = \Lambda(Q,f,m_{\bullet},c_{\bullet}).
\]

\begin{lemma}
\label{lem:6.4-new} 
The algebra
$\Phi(c_{\bullet})$ degenerates to the  biserial
weighted triangulation algebra
$B(Q,f,m_{\bullet},c_{\bullet})$.
In particular, $\Phi(c_{\bullet})$ is a tame algebra.
\end{lemma}

\begin{proof}
We write 
$a = c_{\cO(\alpha)}$,
$b = c_{\cO(\delta)}$,
$c = c_{\cO(\xi)}$.
For each $t \in K$, 
consider the algebra $A(t)$ 
given by the quiver $Q$
and the relations:
\begin{align*}
 \alpha \xi &= b t \omega \varepsilon \nu \sigma 
 , &
 \xi \delta &= a t \beta \gamma
 , &
 \sigma \eta &= a t \gamma \alpha
 , &
 \omega \nu &= a t^2 \alpha \beta
, \\
 \eta \beta &= b t \delta \omega \varepsilon \nu  
 , &
 \delta \alpha &= c t \eta
 , &
 \beta \sigma &= c t \xi
 , &
 \nu \gamma &= b t^2 \varepsilon \nu \sigma \delta
, \\
 \gamma \omega &= b t^2 \sigma \delta \omega \varepsilon   
 , &&& 
\!\!\!\!\!\!\!\!\!\!\!\!\!\!\!\!\!\!\!\!\!\!\!\!\!\!\!\!\!\!\!\!\!\!\!\!\!\!\!\!\!\!\!
\theta f(\theta) g\big(f(\theta)\big) &= 0 
\mbox{ for } \theta \in Q_1 \setminus \{ \beta , \delta \} ,
\!\!\!\!\!\!\!\!\!\!\!\!\!\!\!\!\!\!\!\!\!\!\!\!\!\!\!\!\!\!\!\!\!\!\!\!\!\!\!\!\!\!\!
\!\!\!\!\!\!\!\!\!\!\!\!\!\!\!\!\!\!\!\!\!\!\!\!\!\!\!\!\!\!\!\!\!\!\!\!\!\!\!\!\!\!\!
\\
 \varepsilon^2 &= b t^8 \nu \sigma \delta \omega , 
 &&& 
\!\!\!\!\!\!\!\!\!\!\!\!\!\!\!\!\!\!\!\!\!\!\!\!\!\!\!\!\!\!\!\!\!\!\!\!\!\!\!\!\!\!\!
\theta g(\theta) f\big(g(\theta)\big) &= 0 
\mbox{ for } \theta \in Q_1 \setminus \{ \alpha , \sigma \} ,
\!\!\!\!\!\!\!\!\!\!\!\!\!\!\!\!\!\!\!\!\!\!\!\!\!\!\!\!\!\!\!\!\!\!\!\!\!\!\!\!\!\!\!
\!\!\!\!\!\!\!\!\!\!\!\!\!\!\!\!\!\!\!\!\!\!\!\!\!\!\!\!\!\!\!\!\!\!\!\!\!\!\!\!\!\!\!
\end{align*}
and in addition
\begin{align*}
	a(\alpha\beta\gamma)&= b(\omega\varepsilon\nu\sigma\delta),  &
	a(\beta\gamma\alpha)&= c(\xi\eta), & 
	b(\sigma\delta\omega\varepsilon\nu)&= a(\gamma\alpha\beta), &
\\ 
b(\delta\omega\varepsilon\nu\sigma)&= c(\eta\xi), &
b(\varepsilon\nu\sigma\delta\omega)&= b(\nu\sigma\delta\omega\varepsilon).
\end{align*}
	Then $A(t)$, $t \in K$,
is an algebraic family in the variety $\alg_{38}(K)$, 
with $A(1) = \Phi(c_{\bullet})$
and 
$A(0) =  B(Q,f,m_{\bullet},c_{\bullet})$.
Moreover, for each $t \in K^*$,
there exists an isomorphism of $K$-algebras
$\varphi_t : A(1) \to A(t)$ given by
\begin{align*}
 \varphi_t(\alpha) &= t^{5} \alpha , & 
 \varphi_t(\beta) &= t^5 \beta , &
 \varphi_t(\gamma) &= t^{10} \gamma , &
 \varphi_t(\xi) &= t^{10} \xi , &
 \varphi_t(\eta) &= t^{10} \eta , \\
 \varphi_t(\delta) &= t^{4} \delta , &
 \varphi_t(\omega) &= t^{4} \omega , &
 \varphi_t(\varepsilon) &= t^{4} \varepsilon , &
 \varphi_t(\nu) &= t^{4} \nu , &
 \varphi_t(\sigma) &= t^4 \sigma . 
\end{align*}
Therefore, it follows from 
Proposition~\ref{prop:2.2}
that $\Phi(c_{\bullet})$
degenerates to $B(Q,f,m_{\bullet},c_{\bullet})$,
and $\Phi(c_{\bullet})$ is tame.
\end{proof}

%\bigskip

Let $(Q, f)$ be the triangulation quiver
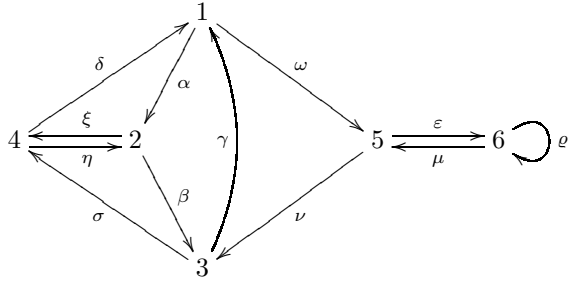
\[
%  \xymatrix@R=2pc@C=1.5pc{
%  \xymatrix@R=3.5pc@C=1.8pc{
  \xymatrix@R=3.pc@C=1.2pc{
%  \xymatrix@C=.8pc{
    &&& 1
    \ar[ld]^{\alpha}
    \ar[rrrd]^{\omega}
    \\   
    4
    \ar[rrru]^{\delta}
    \ar@<-.5ex>[rr]_(.6){\eta}
    && 2
    \ar@<-.5ex>[ll]_(.4){\xi}
    \ar[rd]^{\beta}
    &&&& 5     \ar[llld]^{\nu}
      \ar@<.5ex>[rr]^(.5){\varepsilon}
      && 6
      \ar@<.5ex>[ll]^(.5){\mu}
      \ar@(ru,dr)^{\varrho}[]
    \\
    &&& 3
    \ar[lllu]^{\sigma}
    \ar@/_3ex/[uu]^{\gamma}
  }
\]
with $f$-orbits 
$(\alpha\ \xi\ \delta)$,
$(\beta\ \sigma\ \eta)$, 
$(\gamma\ \omega\ \nu)$,
$(\varepsilon\ \varrho\ \mu)$ .
Then  $g$ has orbits 
\begin{align*}
  &&
  \cO(\alpha) &= (\alpha \ \beta \ \gamma), &
  \cO(\delta) &= (\delta \ \omega \ \varepsilon \ \mu \ \nu \ \sigma),&
  \cO(\xi) &= (\xi \ \eta), &
  \cO(\varrho) &= (\varrho).
  &&
\end{align*}
Let $r \geq 2$ be a natural number and
$m_{\bullet}^r : \cO(g) \to \bN^*$
be the weight function
given by 
$m_{\cO(\alpha)}^r = m_{\cO(\delta)}^r = m_{\cO(\xi)}^r = 1$
and
$m_{\cO(\varrho)}^r = r$.
Moreover, let
$c_{\bullet} : \cO(g) \to K^*$
be an arbitrary parameter function.
We consider the weighted triangulation algebra  
\[
  \Psi(r,c_{\bullet}) = \Lambda(Q,f,m_{\bullet}^r,c_{\bullet}).
\]

\begin{lemma}
\label{lem:6.5-new} 
The algebra
$\Psi(r,c_{\bullet})$ degenerates to the associated biserial
weighted triangulation algebra
$B(Q,f,m_{\bullet}^r,c_{\bullet})$.
In particular, $\Psi(r,c_{\bullet})$ is a tame algebra.
\end{lemma}

\begin{proof}
We write 
$a = c_{\cO(\alpha)}$,
$b = c_{\cO(\delta)}$,
$c = c_{\cO(\xi)}$,
$d = c_{\cO(\varrho)}$.
For each $t \in K$, 
consider the algebra $A(t)$ 
given by the quiver $Q$
and the relations:
\begin{align*}
 \alpha \xi &= b t^r \omega \varepsilon \mu \nu \sigma 
 , &
 \xi \delta &= a t^r \beta \gamma
 , &
 \delta \alpha &= c t^r \eta
, \\
 \eta \beta &= b t^r \delta \omega \varepsilon \mu \nu  
 , &
 \sigma \eta &= a t^r \gamma \alpha
 , &
 \beta \sigma &= c t^r \xi
, \\
 \gamma \omega &= b t^{2r} \sigma \delta \omega \varepsilon \mu   
 , &
 \omega \nu &= a t^{r} \alpha \beta
 , &
 \nu \gamma &= b t^{2r} \varepsilon \mu \nu \sigma \delta
, \\
 \varepsilon \varrho &= b t^{8r-12} \nu \sigma \delta \omega \varepsilon 
 , &
 \varrho \mu &= b t^{8r-12} \mu \nu \sigma \delta \omega    
 , &
 \mu \varepsilon &= d t^{8r-12} \varrho^{r-1} 
 , \\
 & 
\theta f(\theta) g\big(f(\theta)\big) = 0 
\mbox{ for } \theta \in Q_1 \setminus \{ \beta , \delta, \mu \} ,
\!\!\!\!\!\!\!\!\!\!\!\!\!\!\!\!\!\!\!\!\!\!\!\!\!\!\!\!\!\!\!\!\!\!\!\!\!\!\!\!\!\!\!
\\
 %& &&
%\!\!\!\!\!\!\!\!\!\!\!\!\!\!\!\!\!\!\!\!\!\!\!\!\!\!\!\!\!\!\!\!\!\!\!\!\!\!\!\!\!\!\!
%\!\!\!\!\!\!\!\!\!\!\!\!\!\!\!\!
 & 
\theta g(\theta) f\big(g(\theta)\big) = 0 
\mbox{ for } \theta \in Q_1 \setminus \{ \alpha , \sigma, \varepsilon \} ,
%\!\!\!\!\!\!\!\!\!\!\!\!\!\!\!\!\!\!\!\!\!
\!\!\!%\!\!\!
\!\!\!\!\!\!\!\!\!\!\!\!\!\!\!\!\!\!\!\!\!\!\!\!\!\!\!\!\!\!\!\!\!\!\!\!\!%\!\!\!\!\!\!
\\
	&
\mu\varepsilon \mu = 0, 
\ \
\varepsilon\mu\varepsilon = 0, 
\ \ 
	\mbox{ if } r\geq 3.
\!\!\!\!\!\!\!\!\!\!\!\!\!\!\!\!\!\!\!\!\!\!\!\!\!\!\!\!\!\!\!\!\!\!\!\!\!\!\!\!\!\!\!
\!\!\!\!\!\!\!\!\!\!\!\!\!\!\!\!\!\!\!\!\!\!\!\!\!\!\!\!\!\!\!\!\!\!\!\!\!\!\!\!\!\!\!
\!\!\!\!\!\!\!\!\!\!\!\!\!\!\!\!\!\!\!\!\!\!\!\!\!\!\!\!\!\!\!\!\!\!\!\!\!\!\!\!\!\!\!
\end{align*}
In addition we have the relations
\begin{align*}
	a(\alpha\beta\gamma)&= b(\omega\varepsilon\mu\nu\sigma\delta), &
	a(\beta\gamma\alpha)&=c(\xi\eta), &
	b(\sigma\delta\omega\varepsilon\mu\nu)&= a(\gamma\alpha\beta), &
\\
	b(\delta\omega\varepsilon\mu\nu\sigma) &= c(\eta\xi),&
	b(\varepsilon\mu\nu\sigma\delta\omega)&= b(\nu\sigma\delta\omega\varepsilon\mu), & 
	b(\mu\nu\sigma\delta\omega\varepsilon)&= d\varrho^r .
\end{align*}
We note that $8 r - 12 \geq 1$, because $r \geq 2$.
Then $A(t)$, $t \in K$,
is an algebraic family in the variety $\alg_{n}(K)$, 
with $n = r+49$.
Observe also that 
$A(1) = \Psi(r,c_{\bullet})$
and 
$A(0) =  B(Q,f,m_{\bullet}^r,c_{\bullet})$.
Further, for each $t \in K \setminus \{0\}$,
there is an isomorphism of $K$-algebras
$\varphi_t : A(1) \to A(t)$ given by
\begin{align*}
 \varphi_t(\alpha) &= t^{3r} \alpha , & 
 \varphi_t(\beta) &= t^{3r} \beta , &
 \varphi_t(\gamma) &= t^{6r} \gamma , &
 \varphi_t(\xi) &= t^{6r} \xi , \\
 \varphi_t(\eta) &= t^{6r} \eta , &
 \varphi_t(\varrho) &= t^{12} \varrho , &
 \varphi_t(\delta) &= t^{2r} \delta , &
 \varphi_t(\omega) &= t^{2r} \omega , \\
 \varphi_t(\varepsilon) &= t^{2r} \varepsilon , &
 \varphi_t(\mu) &= t^{2r} \mu , &
 \varphi_t(\nu) &= t^{2r} \nu , &
 \varphi_t(\sigma) &= t^{2r} \sigma . 
\end{align*}
Therefore, it follows from 
Proposition~\ref{prop:2.2}
that $\Psi(r,c_{\bullet})$
degenerates to $B(Q,f,m_{\bullet}^r,c_{\bullet})$,
and $\Psi(r,c_{\bullet})$ is tame.
\end{proof}

Towards the general case, let
$\Lambda = \Lambda(Q,f,m_{\bullet},c_{\bullet})$ 
be an arbitrary weighted triangulation algebra.
We define
\begin{align*} 
	M &:= M_{\Lambda} = 2 \lcm \{ m_{\cO} n_{\cO}  \,|\,  \cO \in \cO(g) \}, \cr
	q(\alpha) &: = m_{\alpha}n_{\alpha}, \ q: Q_1\to \bN^*, \cr
	v(\alpha) &:= M\bigg(1- \frac{1}{q(\alpha)} - \frac{1}{q(f(\alpha))} - 
	\frac{1}{q\big(f^2(\alpha)\big)} \bigg)
\end{align*}
for any $\alpha \in Q_1$. We will see in Proposition \ref{prop:6.5} that as
long as 
$v(\alpha)\geq 1$ for all arrows $\alpha$, we can define 
algebraic family of algebras and show that $\La$ degenerates to the
associated biserial triangulation algebra. First we determine
quivers which have arrows $\alpha$ with  $v(\alpha)\leq 0$.

\medskip

\begin{proposition}
\label{prop:6.4}
Assume that $v(\alpha) \leq 0$ 
for some arrow $\alpha \in Q_1$.
Then $\Lambda$ is isomorphic to one of the algebras
$D(\lambda)$,
$T(\lambda)$,
$S(\lambda)$,
$\Lambda(\lambda)$,
with $\lambda \in K^*$,
$S(r,c_{\bullet})$,
with $r \geq 2$,
$\Sigma(r,c_{\bullet})$, with $r \geq 3$, 
$\Omega(r,c_{\bullet})$, with $r \geq 4$, 
$\Phi(c_{\bullet})$, 
or $\Psi(r,c_{\bullet})$, with $r \geq 2$. 
\end{proposition}

\begin{proof}
We  have
$v(\alpha) \leq 0$
if and only if 
$\tfrac{1}{q(\alpha)} + \tfrac{1}{q(f(\alpha))} 
 + \tfrac{1}{q(f^2(\alpha))} \geq 1$.
That is,  (up to rotation)
$\big(q(\alpha) , q(f(\alpha)) , q(f^2(\alpha))\big)$
is one of the triples:
$(2,2,n)$, $n \geq 2$,
$(2,3,3)$,
$(2,3,4)$,
$(2,3,5)$,
$(2,3,6)$,
$(2,4,4)$,
and
$(3,3,3)$.

\smallskip

(1) \ The case $(2, 2, n)$ does not occur:
	Suppose we have $\alpha\neq f(\alpha)$ and they are
	both virtual. Then $g(\alpha)$ and $f(\alpha)$ are
	virtual starting at the same vertex, which contradicts 
	Assumption \ref{ass}.
\smallskip

(2)
	We determine when $q(\alpha)=2$ and $v(\alpha)\leq 0$ when
	$\alpha$ is a loop. 
Then $(Q,f)$ has a subquiver of the form
\[
  \xymatrix{
%  \xymatrix@C=1pc{
    \bullet
    \ar@(ld,ul)^{\alpha}[]
    \ar@<.5ex>[r]^{\beta}
    & \bullet
    \ar@<.5ex>[l]^{\gamma}
  } 
%,
\]
with $f$-orbit
$(\alpha\ \beta\ \gamma)$
and  $g(\alpha)=\alpha$.
Moreover,
$m_{\alpha} = 2$.
Further, $\cO(\beta) = \cO(\gamma)$
has length  at least $3$.
Hence we have 
$q(\beta) = m_{\beta} n_{\beta}  = m_{\gamma} n_{\gamma} = q(\gamma)\geq 3$, and
since $v(\alpha) \leq 0$, it is equal to $3$ or $4$, and then $m_{\beta}=1$. 
By Assumption \ref{ass} we can only have $m_{\beta}n_{\beta}=4$. 
Then $(Q,f)$ is the quiver
\[
  \xymatrix{
%  \xymatrix@C=1pc{
    \bullet
%    \ar `ld_u[] `_rd[]^{\alpha} []
    \ar@(ld,ul)^{\alpha}[]
    \ar@<.5ex>[r]^{\beta}
    & \bullet
    \ar@<.5ex>[l]^{\gamma}
%    \ar `ru_d[] `_lu[]^{\eta} [] &
    \ar@<.5ex>[r]^{\sigma}
    & \bullet
    \ar@<.5ex>[l]^{\delta}
%    \ar `ru_d[] `_lu[]^{\xi} [] &
    \ar@(ru,dr)^{\eta}[]
  } 
%,
\]
with $f$-orbits
$(\alpha\ \beta\ \gamma)$
and
$(\eta\ \delta\ \sigma)$,
and  $g$-orbits
$\cO(\alpha) = (\alpha)$,
$\cO(\beta) = (\beta\ \sigma\ \delta\ \gamma)$,
$\cO(\eta) = (\eta)$.
Then $\Lambda$ is isomorphic
to one of the algebras $T(\lambda)$,
for some $\lambda \in K^*$,
or 
$\Sigma(r, c_{\bullet})$ for some $r \geq 3$.

\smallskip

(3) We determine when $q(\xi)=2$ and $v(\xi)\leq 0$
where $\xi$ is not a loop. 
Then $(Q,f)$ contains a subquiver of the form
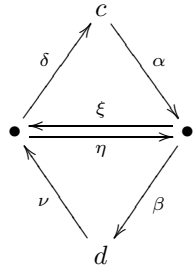
\[
%  \xymatrix@R=2pc@C=1.5pc{
%  \xymatrix@R=3.5pc@C=1.8pc{
  \xymatrix@R=3.pc@C=1.8pc{
%  \xymatrix@C=.8pc{
    & c
    \ar[rd]^{\alpha}
    \\   
    \bullet
    \ar[ru]^{\delta}
    \ar@<-.5ex>[rr]_{\eta}
    && \bullet
    \ar@<-.5ex>[ll]_{\xi}
    \ar[ld]^{\beta}
    \\
    & d
    \ar[lu]^{\nu}
  }
\]
with  $f$-orbits
$(\xi \ \delta \  \alpha)$
and
$(\beta\ \nu\ \eta)$
and  $g$-orbit
$(\xi \ \eta)$.
Since $q(\xi) = 2$,
we have that $\xi$, $\eta$ are virtual arrows.

Assume first that $c = d$.
Clearly, then we have $g$-orbits
$(\alpha \ \beta)$
and
$(\nu\ \delta)$.
It follows from Assumption \ref{ass}
that $\alpha$, $\beta$, $\gamma$, $\delta$
are not virtual.
Hence we have 
$m_{\alpha} n_{\alpha}  = m_{\beta} n_{\beta} \geq 4$
and
$m_{\nu} n_{\nu}  = m_{\delta} n_{\delta} \geq 4$.
But then $v(\xi) \leq 0$ implies 
$q(\alpha) = 4$ and $q(\delta) = 4$.
Thus
$m_{\alpha} n_{\alpha}  = m_{\beta} n_{\beta} = 4$
and
$m_{\nu} n_{\nu}  = m_{\delta} n_{\delta} = 4$,
and consequently $\Lambda$ is isomorphic
to $T(\lambda)$ for some $\lambda \in K^*$.

\smallskip

Assume now that $c \neq d$.
Then the $g$-orbits
$\cO(\alpha) = \cO(\beta)$
and
$\cO(\nu) = \cO(\delta)$
have lengths at least $3$.

Suppose (say) 
$m_{\alpha}n_{\alpha}=3$, so that $g$ as a cycle
$(\alpha \ \beta \ \gamma)$ where $\gamma$ is an arrow $d\mapsto c$. 
It follows from this that the $g$-orbit of $\delta$ has the form
\[
  \big(f^2(\gamma) \ \nu \ \delta \ f(\gamma) \ *\big)
\]
of length $\geq 5$. We assume $v(\xi)\leq 0$ and hence the length
can only be $5$ or $6$. Suppose it has length 5, then $*$ is just a loop, 
$\ve = g(f(\gamma))$, 
the quiver has five vertices and $\Lambda$ is an algebra
of the form $\Phi(c_{\bullet})$. 

Suppose the $g$-orbit of $\delta$ has length $6$, say it is
$(f^2(\gamma) \ \nu \ \delta \ f(\gamma) \ \varepsilon \ \mu)$, 
then we
must have $f(\mu)=\varepsilon$, and $f(\varepsilon)$ is a loop. 
It follows that $Q$ has
six vertices and $\La$ is isomorphic to 
$\Psi(r,c_{\bullet})$ for some $r \geq 2$.

\medskip

Otherwise
$q(\alpha) = m_{\alpha} n_{\alpha}  \geq 4$
and
$q(\delta) =  m_{\delta} n_{\delta} \geq 4$.
But then $v(\xi) \leq 0$ implies 
$q(\alpha) = 4$
and
$q(\delta) = 4$,
and therefore we have 
$m_{\alpha} = 1$, $n_{\alpha} = 4$,
$m_{\delta} = 1$, $n_{\delta} = 4$.
As well 
$m_{\beta} = 1$, $n_{\beta} = 4$,
$m_{\nu} = 1$, $n_{\nu} = 4$.
Summing up, we conclude that 
that $(Q,f)$ 
is the triangulation quiver of the form
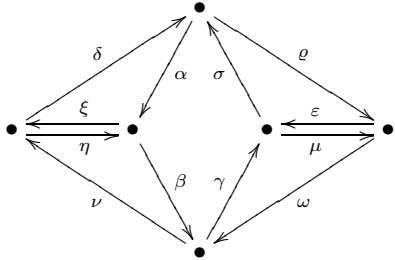
\[
%  \xymatrix@R=2pc@C=1.5pc{
%  \xymatrix@R=3.5pc@C=1.8pc{
  \xymatrix@R=3.pc@C=1.2pc{
%  \xymatrix@C=.8pc{
    &&& \bullet
    \ar[ld]^{\alpha}
    \ar[rrrd]^{\varrho}
    \\   
    \bullet
    \ar[rrru]^{\delta}
    \ar@<-.5ex>[rr]_(.6){\eta}
    && \bullet
    \ar@<-.5ex>[ll]_(.4){\xi}
    \ar[rd]^{\beta}
    && \bullet
    \ar[lu]^{\sigma}
    \ar@<-.5ex>[rr]_(.4){\mu}
    && \bullet
    \ar@<-.5ex>[ll]_(.6){\varepsilon}
    \ar[llld]^{\omega}
%    &&&& \bullet
%    \ar[llld]^{\nu}
    \\
    &&& \bullet
    \ar[lllu]^{\nu}
    \ar[ur]^{\gamma}
  }
\]
with  $f$-orbits
$(\alpha\ \xi\ \delta)$,
$(\beta\ \nu\ \eta)$,
$(\sigma\ \varrho\ \varepsilon)$,
$(\gamma\ \mu\ \omega)$
and the $g$-orbits
$\cO(\xi) = (\xi\ \eta)$,
$\cO(\alpha) = (\alpha\ \beta\ \gamma\ \sigma)$,
$\cO(\delta) = (\delta\ \varrho\ \omega\ \nu)$,
$\cO(\varepsilon) = (\varepsilon\ \mu)$.
Therefore, $\Lambda$ is isomorphic
to one of the algebras 
$S(\lambda)$, for some $\lambda \in K^*$,
or 
$S(r, c_{\bullet})$ for some $r \geq 2$.

\smallskip
(4)
Assume now that there is a loop $\sigma \in Q_1$
with $f(\sigma) = \sigma$ and $v(\sigma) \leq 0$,
and hence $q(\sigma) = 3$ and $v(\sigma) = 0$.
Then $(Q,f)$ is the triangulation quiver
\[
  \xymatrix{
%  \xymatrix@C=1pc{
    1
    \ar@(ld,ul)^{\alpha}[]
    \ar@<.5ex>[r]^{\beta}
    & 2
    \ar@<.5ex>[l]^{\gamma}
    \ar@(ru,dr)^{\sigma}[]
  } 
%,
\]
with  $f$-orbits
$(\alpha \ \beta \ \gamma)$,
$(\sigma)$, 
and 
the $g$-orbits 
$(\alpha)$,
$(\beta \ \sigma \ \gamma)$.
Moreover, we have
$m_{\beta} n_{\beta} = m_{\gamma} n_{\gamma} 
 = m_{\sigma} n_{\sigma} = 3$
and
$m_{\alpha} = m_{\alpha} n_{\alpha}  = r$
for some $r \geq 3$
(by Assumption~\ref{ass}).
Therefore, $\Lambda$ is isomorphic
to one of the algebras 
$D(\lambda)$, for some $\lambda \in K^*$,
or 
$\Omega(r, c_{\bullet})$ for some $r \geq 4$.

\smallskip

(5) The last case to consider is  
where  
$(\eta\ \gamma\ \delta)$
is an $f$-orbit in $Q_1$ of lenght $3$
with 
$v(\eta) \leq 0$
and
$q(\eta) \geq q(\gamma) \geq q(\delta) \geq 3$.
Then
$q(\eta) = q(\gamma) = q(\delta) = 3$. 
We have two cases. 
Assume first that none of $\eta, \gamma, \delta$ is a loop. 
Then by some basic combinatorics, one sees that
$(Q,f)$ is the triangulation quiver 
%of the form.
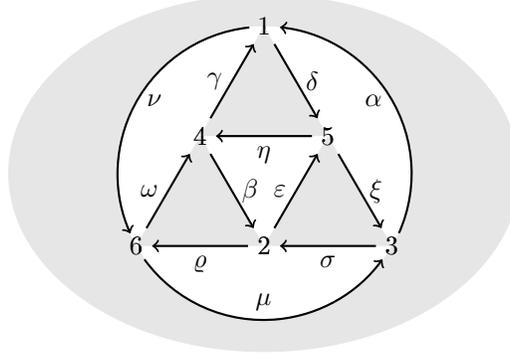
\[
\begin{tikzpicture}
[scale=.85]
\node (1) at (0,1.72) {$1$};
\node (2) at (0,-1.72) {$2$};
\node (3) at (2,-1.72) {$3$};
\node (4) at (-1,0) {$4$};
\node (5) at (1,0) {$5$};
\node (6) at (-2,-1.72) {$6$};
\coordinate (1) at (0,1.72);
\coordinate (2) at (0,-1.72);
\coordinate (3) at (2,-1.72);
\coordinate (4) at (-1,0);
\coordinate (5) at (1,0);
\coordinate (6) at (-2,-1.72);
\fill[fill=gray!20]
      (0,2.22cm) arc [start angle=90, delta angle=-360, x radius=4cm, y radius=2.8cm]
 --  (0,1.72cm) arc [start angle=90, delta angle=360, radius=2.3cm]
     -- cycle;
\fill[fill=gray!20]
    (1) -- (4) -- (5) -- cycle;
\fill[fill=gray!20]
    (2) -- (4) -- (6) -- cycle;
\fill[fill=gray!20]
    (2) -- (3) -- (5) -- cycle;

\node (1) at (0,1.72) {$1$};
\node (2) at (0,-1.72) {$2$};
\node (3) at (2,-1.72) {$3$};
\node (4) at (-1,0) {$4$};
\node (5) at (1,0) {$5$};
\node (6) at (-2,-1.72) {$6$};
\draw[->,thick] (-.23,1.7) arc [start angle=96, delta angle=108, radius=2.3cm] node[midway,right] {$\nu$};
\draw[->,thick] (-1.87,-1.93) arc [start angle=-144, delta angle=108, radius=2.3cm] node[midway,above] {$\mu$};
\draw[->,thick] (2.11,-1.52) arc [start angle=-24, delta angle=108, radius=2.3cm] node[midway,left] {$\alpha$};
\draw[->,thick]
(1) edge node [right] {$\delta$} (5)
(2) edge node [left] {$\varepsilon$} (5)
(2) edge node [below] {$\varrho$} (6)
(3) edge node [below] {$\sigma$} (2)
(4) edge node [left] {$\gamma$} (1)
(4) edge node [right] {$\beta$} (2)
(5) edge node [right] {$\xi$} (3)
(5) edge node [below] {$\eta$} (4)
(6) edge node [left] {$\omega$} (4)
;
\end{tikzpicture}
\]
with  $f$-orbits described by the shaded triangles
and $g$-orbits described by the white triangles,
and the weight function
$m_{\bullet} : \cO(g) \to \bN^*$
taking value $1$.
Then $\Lambda$ is isomorphic
to a tetrahedral algebra $\Lambda(\lambda)$, 
for some $\lambda \in K^*$.

This leaves the case, where one of the three arrows is a loop.
We label the arrows now as $\alpha, \beta, \gamma$ and we assume
$\alpha$ is a loop, and $f$ has the cycle $(\alpha \ \beta \ \gamma)$
and we have $q(\alpha)=q(\beta) = q(\gamma)=3$. 
Then $g(\beta)$ must be a loop, $\sigma$ say, which then has
to be fixed by $f$. Hence  $Q$ has two vertices,
and   the algebra  be isomorphic to $D(\lambda)$ for some $\lambda \in K^*$. 
\end{proof}

\begin{proposition}
\label{prop:6.5}
Let $\Lambda = \Lambda(Q,f,m_{\bullet},c_{\bullet})$
be a weighted triangulation algebra which is not isomorphic
to one of the algebras 
$D(\lambda)$,
$T(\lambda)$,
$S(\lambda)$,
$\Lambda(\lambda)$,
with $\lambda \in K^*$,
$S(r, c_{\bullet})$ with $r \geq 2$,
$\Sigma(r, c_{\bullet})$ with $r \geq 3$,
$\Omega(r, c_{\bullet})$ with $r \geq 4$, 
$\Phi(c_{\bullet})$, or
$\Psi(r,c_{\bullet})$, for some $r \geq 2$.
Then $\Lambda$ degenerates to the biserial
weighted triangulation algebra
$B = B(Q,f,m_{\bullet},c_{\bullet})$.
In particular, $\Lambda$ is tame.
\end{proposition}

\begin{proof}
Let $M, q,v : Q_1 \to \bN^*$
	be the functions defined above. We set now $u(\alpha):= \frac{M}{q(\alpha)}$, this is an integer. 
It follows from Proposition~\ref{prop:6.4}
and the  assumption
that $v(\alpha) \geq 2$ for any arrow
$\alpha \in Q_1$.
For each $t \in K$, consider
the algebra $\Lambda(t)$
given by the quiver $Q$ and the relations:
\begin{itemize}
 \item
  $\alpha f(\alpha) = c_{\bar{\alpha}} t^{v(\alpha)} A_{\bar{\alpha}}$
  for any arrow $\alpha \in Q_1$,
 \item
  $c_{\alpha} B_{\alpha} = c_{\bar{\alpha}} B_{\bar{\alpha}}$
  for any arrow $\alpha \in Q_1$,
 \item
  $\alpha f(\alpha) g(f(\alpha)) = 0$
  for any arrow $\alpha \in Q_1$ with
  $f^{-1}(\alpha)$ not virtual,
 \item
  $\alpha g(\alpha) f(g(\alpha)) = 0$
  for any arrow $\alpha \in Q_1$ with
  $f(\alpha)$ not virtual.
\end{itemize}
Then $\Lambda(\lambda)$, $\lambda \in K$,
is an algebraic family in the variety $\alg_d(K)$, 
with $d = \sum_{\cO \in \cO(g)} m_{\cO} n_{\cO}^2$.
Observe also that 
$\Lambda(1) = \Lambda$
and
$\Lambda(0) = B$.
Further, we claim that for any  $t \in K^*$,
we have an isomorphism of $K$-algebras 
$\varphi_t :\Lambda(1) \to \Lambda(t)$
given by $\varphi_t(\alpha) = t^{u(\alpha)} \alpha$
	for any arrow $\alpha \in Q_1$, and extension to products.
Indeed, it follows from definition of the function $u$
that
$\varphi_t(B_{\sigma}) = t^M B_{\sigma}$
for any arrow $\sigma \in Q_1$,
and hence 
$\varphi_t(A_{\sigma}) = t^{M-u(f^2(\bar{\sigma}))} A_{\sigma}$,
because
$B_{\sigma} = A_{\sigma} g^{n_{\sigma}-1}(\sigma)$
and
$g^{n_{\sigma}-1}(\sigma) = f^2(\bar{\sigma})$.
Then, for any arrow $\alpha \in Q_1$,
we have the equalities
\begin{align*} 
  \varphi_t\big(\alpha f(\alpha)\big) 
   &= t^{u(\alpha) + u(f(\alpha))} \alpha f(\alpha) 
   = t^{u(\alpha) + u(f(\alpha))} t^{v(\alpha)} c_{\bar{\alpha}} A_{\bar{\alpha}}
   = t^{M-u(f^2(\alpha))} c_{\bar{\alpha}} A_{\bar{\alpha}}
\\&
   =  \varphi_t(c_{\bar{\alpha}} A_{\bar{\alpha}}) ,
\end{align*} 
and hence $\varphi_t$ is a well-defined isomorphism of $K$-algebras.
Therefore, by 
Proposition~\ref{prop:2.2},
$\Lambda$ degenerates to $B$,
and $\Lambda$ is a tame algebra. 
\end{proof}

We shall prove now that every weighted surface algebra 
non-isomorphic to a disc algebra, triangle algebra,
tetrahedral algebra, or spherical algebra
is of non-polynomial growth.
We consider first two distinguished cases.

\begin{example}
\label{ex:6.6}
Let $T$ be the triangulation
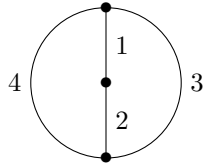
\[
\begin{tikzpicture}[scale=1,auto]
\coordinate (o) at (0,0);
\coordinate (c) at (0,-1);
\coordinate (a) at (0,1);
\coordinate (b) at (1,0);
\coordinate (d) at (-1,0);
\draw (a) to node {$1$} (o);
\draw (o) to node {$2$} (c);
%\draw (b) arc (-90:270:1) node [below] {$b$};
%\draw (.707,.707) arc (45:405:1) node [above right] {$2$};
%\draw (b) arc (-180:180:1) node [left] {$2$};
\draw (b) arc (0:360:1) node [right] {$3$};
\draw (d) node [left] {$4$};
\node (d) at (0,-1) {$\bullet$};
\node (c) at (0,0) {$\bullet$};
\node (b) at (0,1) {$\bullet$};
\end{tikzpicture}
\]
of the unit disc $D = D^2$ in $\bR^2$
by two triangles
and $\vec{T}$ the orientation
$(1\ 3\ 2)$, 
$(2\ 4\ 1)$
of triangles in $T$.
Then the  triangulation quiver
$(Q,f) = (Q(D,\vec{T}),f)$ is the quiver
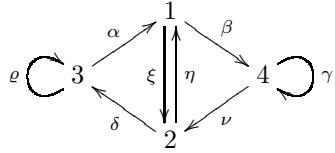
\[
  \xymatrix@R=1pc{
%   \xymatrix{
   & 1 \ar[rd]^{\beta} \ar@<-.5ex>[dd]_{\xi}  \\
   3   \ar@(ld,ul)^{\varrho}[]  \ar[ru]^{\alpha} 
   && 4   \ar@(ru,dr)^{\gamma}[]  \ar[ld]^{\nu}   \\
   & 2 \ar[lu]^{\delta}  \ar@<-.5ex>[uu]_{\eta}
  } 
\]
with $f$-orbits
$(\alpha\ \xi\ \delta)$,
$(\beta\ \nu\ \eta)$,
$(\varrho)$,
$(\gamma)$.
Then $g$ has two 
orbits, 
$\cO(\alpha) = (\alpha\ \beta\ \gamma\ \nu\ \delta\ \varrho)$ 
and
$\cO(\xi) = (\xi\ \eta)$.
Let
$m_{\bullet} : \cO(g) \to \bN^*$
be the trivial multiplicity function
and $c_{\bullet} : \cO(g) \to K^*$
an arbitrary parameter function.
We write
$c_{\cO(\alpha)} = a$
and $c_{\cO(\xi)} = b$.
Then the weighted surface algebra
$D(a,b)^{(2)} = \Lambda(D,\vec{T},m_{\bullet},c_{\bullet})$
is given by the above quiver and the relations:
\begin{align*}
% &&
 \alpha \xi &= a \varrho \alpha \beta \gamma \nu ,
 &
 \xi \delta &= a \beta \gamma \nu \delta \varrho ,
 &
 \delta \alpha &= b \eta ,
 &
 \varrho^2 &= a \alpha \beta \gamma \nu \delta ,
&  \alpha  \xi \eta &= 0 ,
 \\
% &&
 \nu \eta &= a \gamma \nu \delta \varrho \alpha ,
 &
 \eta \beta &= a \delta \varrho \alpha \beta \gamma ,
 &
 \beta \nu &= b \xi ,
 &
  \gamma^2 &= a \nu \delta \varrho \alpha \beta ,
& \xi \eta \beta &= 0 ,
 \\
%  \alpha  \xi \eta &= 0 ,
% &
 \xi \delta \varrho &= 0 ,  
 &
 \nu \eta \xi &= 0 ,
 &
 \eta \beta \gamma &= 0 ,
 &
 \varrho^2 \alpha &= 0 ,
 &
  \gamma^2 \nu &= 0 ,
 \\
% \xi \eta \beta &= 0 ,
% &
 \varrho \alpha \xi &= 0 ,
 &
 \eta \xi \delta &= 0 ,
 &
 \gamma \nu \eta &= 0 ,
 &
 \delta \varrho^2 &= 0 ,
 &
 \beta \gamma^2 &= 0 .
\end{align*}
Observe that the algebra
$D(a,b)^{(2)}$
is isomorphic to the algebra
$D(ab,1)^{(2)}$.
Indeed, there is an isomorphism of algebras
$\varphi : D(ab, 1)^{(2)} \to D(a,b)^{(2)}$
given by
$\varphi(\alpha) = \alpha$,
$\varphi(\xi) = b \xi$,
$\varphi(\delta) = \delta$,
$\varphi(\varrho) = \varrho$,
$\varphi(\beta) = \beta$,
$\varphi(\gamma) = \gamma$,
$\varphi(\nu) = \nu$,
$\varphi(\eta) = b \eta$.
For $\lambda \in K^*$,
we set 
$D(\lambda)^{(2)} = D(\lambda, 1)^{(2)}$.
We see now that $D(\lambda,1)^{(2)}$
is given by its Gabriel quiver $Q^{(2)}$
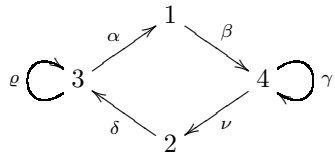
\[
  \xymatrix@R=1pc{
%   \xymatrix{
   & 1 \ar[rd]^{\beta} \\
   3   \ar@(ld,ul)^{\varrho}[]  \ar[ru]^{\alpha} 
   && 4   \ar@(ru,dr)^{\gamma}[]  \ar[ld]^{\nu}   \\
   & 2 \ar[lu]^{\delta} 
  } 
\]
and the relations:
\begin{align*}
 \alpha \beta \nu &= \lambda \varrho \alpha \beta \gamma \nu ,
 \!\!\!\!\!\!\!\!\!\!\!\!\!\!\!\!\!& &
 &
  \beta \nu \delta &= \lambda  \beta \gamma \nu \delta \varrho ,
 \!\!\!\!\!\!\!\!\!\!\!\!\!\!\!\!\!& &
 &
 \varrho^2 &= \lambda \alpha \beta \gamma \nu \delta ,
\!\!\!\!\!\!\!\!\!\!\!\!\!\!\!\!\!
 \\
 \nu \delta \alpha &= \lambda \gamma \nu \delta \varrho \alpha ,
 \!\!\!\!\!\!\!\!\!\!\!\!\!\!\!\!\!& &
 &
 \delta \alpha \beta &= \lambda \delta \varrho \alpha \beta \gamma ,
 \!\!\!\!\!\!\!\!\!\!\!\!\!\!\!\!\!& &
 &
 \gamma^2 &= \lambda \nu \delta \varrho \alpha \beta ,
\!\!\!\!\!\!\!\!\!\!\!\!\!\!\!\!\!
 \\
  \alpha \beta \nu \delta \alpha &= 0 ,
  &
  \beta \nu \delta \varrho &= 0 ,
  &
  \nu \delta \alpha \beta \nu &= 0 ,
  &
  \delta \alpha \beta\gamma &= 0 ,
  &
  \varrho^2 \alpha &= 0 ,
  &
  \gamma^2 \nu &= 0 ,
 \\
  \beta \nu \delta \alpha \beta &= 0 ,
  &
  \varrho \alpha \beta \nu &= 0 ,
  &
  \delta \alpha \beta \nu \delta &= 0 ,
  &
  \gamma \nu \delta \alpha &= 0 ,
  &
  \delta \varrho^2 &= 0 ,
  &
  \beta  \gamma^2 &= 0 .
\end{align*}

We consider also the orbit algebra
$D(a,b)^{(1)} = D(a,b)^{(2)}/H$
of $D(a,b)^{(2)}$
with respect to action of the cyclic group
of order $2$ on $D(a,b)^{(2)}$
given by the cyclic rotation of vertices
and arrows of the quiver $Q$:
\begin{align*}
 &&
 (1\ 2) ,
 &&
 (3\ 4) ,
 &&
 (\alpha\ \nu) ,
 &&
 (\beta\ \delta) ,
 &&
 (\xi\ \eta) ,
 &&
 (\gamma\ \varrho) .
 &&
\end{align*}
Then
$D(a,b)^{(1)}$
is given by the triangulation quiver $(Q',f')$
of the form
\[
  \xymatrix{
%  \xymatrix@C=1pc{
    1
    \ar@(ld,ul)^{\xi}[]
    \ar@<.5ex>[r]^{\beta}
    & 3
    \ar@<.5ex>[l]^{\alpha}
    \ar@(ru,dr)^{\gamma}[]
  } 
%,
\]
with $f'$-orbits
$(\xi\ \beta\ \alpha)$
and
$(\gamma)$,
and the relations:
\begin{align*}
 \alpha \xi &= a \gamma \alpha \beta \gamma \alpha ,
 \!\!\!\!\!\!\!\!\!\!\!\!\!\!\!\!\!& &
 &
 \xi \beta &= a \beta \gamma \alpha \beta \gamma ,
% \!\!\!\!\!\!\!\!\!\!\!\!\!\!\!\!\!& &
 &
 \beta \alpha &= b \xi ,
 &
 \gamma^2 &= a \alpha \beta \gamma \alpha \beta ,
 \!\!\!\!\!\!\!\!\!\!\!\!\!\!\!\!\!& &
 \\ 
 \alpha \xi^2 &= 0 ,
 &
 \xi \beta \gamma &= 0 ,
% &
% \alpha \xi^2 &= 0 ,
 &
 \gamma^2 \alpha &= 0 ,
 &
 \xi^2 \beta &= 0 ,
 &
 \gamma \alpha \xi &= 0 ,
 &
 \beta \gamma^2 &= 0.
\end{align*}
We note that 
$D(a,b)^{(1)}$
is the weighted triangular algebra
$\Lambda(Q',f',m_{\bullet}',c_{\bullet}')$,
where the weight function
$m_{\bullet}' : \cO(g') \to \bN^*$
is given by
$m_{\cO(\alpha)}' = 2 = m_{\cO(\xi)}'$,
and
the parameter function
$c_{\bullet}' : \cO(g') \to K^*$
by
$c_{\cO(\alpha)}' = a$
and
$c_{\cO(\xi)}' = b$.
Similarly as above, we conclude that
$D(a,b)^{(1)}$
is isomorphic to the algebra
$D(ab,1)^{(1)}$.
For $\lambda \in K^*$,
we set 
$D(\lambda)^{(1)} = D(\lambda, 1)^{(1)}$.
Further, the Gabriel quiver $Q^{(1)}$
of $D(\lambda)^{(1)}$
is the orbit quiver $Q^{(2)}/H$
of the Gabriel quiver $Q^{(2)}$
of $D(\lambda)^{(2)}$
with respect to the induced action of $H$,
and is of the form
\[
  \xymatrix{
%  \xymatrix@C=1pc{
    1
    \ar@<.5ex>[r]^{\beta}
    & 3
    \ar@<.5ex>[l]^{\alpha}
    \ar@(ru,dr)^{\gamma}[]
  } 
 .
\]
Hence, $D(\lambda)^{(1)}$ is given by the quiver $Q^{(1)}$
and the relations:
\begin{align*}
 \alpha \beta \alpha &= \lambda \gamma \alpha \beta \gamma \alpha ,
 \!\!\!\!\!\!\!\!\!\!\!\!\!\!\!\!\!& &
 &
 \beta \alpha \beta &= \lambda \beta \gamma \alpha \beta \gamma , 
 \!\!\!\!\!\!\!\!\!\!\!\!\!\!\!\!\!& &
 &
 \gamma^2 &= \lambda \alpha \beta \gamma \alpha \beta ,
 \!\!\!\!\!\!\!\!\!\!\!\!\!\!\!\!\!& &
 \\
 \alpha \beta \alpha \beta \alpha &= 0 ,
 &
 \beta \alpha \beta \gamma &= 0 ,
 &
 \gamma^2 \alpha &= 0 ,
 &
 \beta \alpha \beta \alpha \beta &= 0 ,
 &
 \gamma \alpha \beta \alpha &= 0 ,
 &
 \beta \gamma^2 &= 0 .
\end{align*}
\end{example}

\begin{lemma}
\label{lem:6.7}
For each $\lambda \in K^*$,
the algebras $D(\lambda)^{(1)}$ and $D(\lambda)^{(2)}$
are of non-polynomial growth.
\end{lemma}

\begin{proof}
We fix $\lambda \in K^*$
and consider the quotient algebra $A^{(2)}$ of $D(\lambda)^{(2)}$
given by its Gabriel quiver $Q^{(2)}$
and the zero-relations:
\begin{align*}
 \alpha \beta \nu  &= 0 ,
 &
  \beta \nu \delta  &= 0 ,
 &
 \nu \delta \alpha  &= 0 ,
 &
 \delta \alpha \beta  &= 0 ,
 &
 \varrho^2  &= 0 ,
 &
 \gamma^2  &= 0 ,
 \\
 \varrho \alpha \beta \gamma \nu  &= 0 ,
 &
 \alpha \beta \gamma \nu \delta  &= 0 ,
 &
 \beta \gamma \nu \delta \varrho  &= 0 ,
 &
 \gamma \nu \delta \varrho \alpha  &= 0 ,
 &
 \nu \delta \varrho \alpha \beta  &= 0 ,
 &
 \delta \varrho \alpha \beta \gamma  &= 0 .
\end{align*}
Then $A^{(2)}$ admits the Galois covering
\[
  F^{(2)} : R \to R / G^{(2)} = A^{(2)}
\]
where $R$ is the locally bounded category given by
the infinite quiver
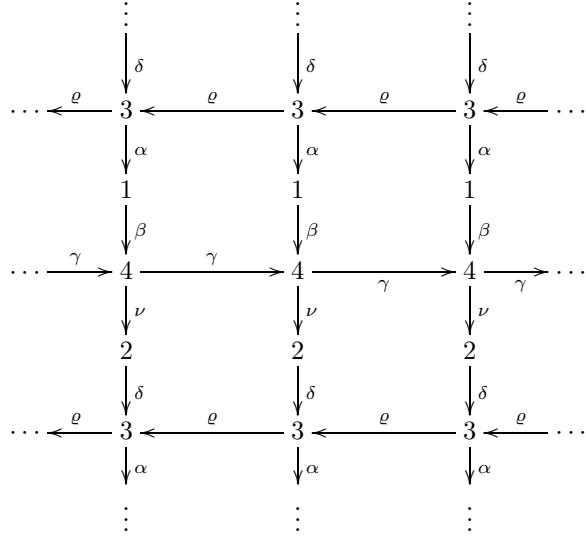
\[
%  \xymatrix{
  \xymatrix@R=1.5pc{
%  \xymatrix@C=1pc{
    & \vdots \ar[d]^{\delta} && \vdots \ar[d]^{\delta} && \vdots \ar[d]^{\delta} \\
    \cdots & 3 \ar[l]_{\varrho} \ar[d]^{\alpha} 
      && 3 \ar[ll]_{\varrho} \ar[d]^{\alpha}
     &&  3 \ar[ll]_{\varrho}  \ar[d]^{\alpha}
    & \cdots \ar[l]_{\varrho} \\
    & 1 \ar[d]^{\beta} && 1 \ar[d]^{\beta} && 1 \ar[d]^{\beta} \\
    \cdots \ar[r]^{\gamma}  & 4 \ar[rr]^{\gamma} \ar[d]^{\nu} 
      && 4 \ar[rr]_{\gamma} \ar[d]^{\nu}
     &&  4 \ar[r]_{\gamma}  \ar[d]^{\nu}
    & \cdots \\
    & 2 \ar[d]^{\delta} && 2 \ar[d]^{\delta} && 2 \ar[d]^{\delta} \\
    \cdots & 3 \ar[l]_{\varrho} \ar[d]^(.4){\alpha} 
      && 3 \ar[ll]_{\varrho} \ar[d]^(.4){\alpha}
     &&  3 \ar[ll]_{\varrho}  \ar[d]^(.4){\alpha}
    & \cdots \ar[l]_{\varrho} \\
    & \vdots && \vdots  && \vdots 
  } 
\]
and the induced relations, and $G^{(2)}$ is the free abelian group 
of rank $2$ generated by the obvious horizontal and vertical shifts. 
We observe now that $R$ contains the full convex subcategory $B$ 
given by the quiver
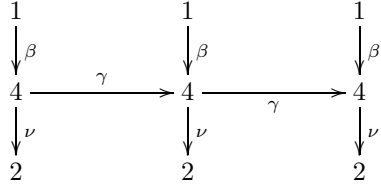
\[
%  \xymatrix{
  \xymatrix@R=1.5pc{
%  \xymatrix@C=1pc{
    1 \ar[d]^{\beta} && 1 \ar[d]^{\beta} && 1 \ar[d]^{\beta} \\
    4 \ar[rr]^{\gamma} \ar[d]^{\nu} 
      && 4 \ar[rr]_{\gamma} \ar[d]^{\nu}
     &&  4  \ar[d]^{\nu} \\
    2 && 2 && 2 
  } 
\]
and the relation $\gamma^2 = 0$, which is a minimal non-polynomial 
growth algebra of type (1) in \cite[Theorem~3.2]{NoS}. 
Hence, by general theory (see \cite{DS0,DS1,Ga}), 
$A^{(2)}$ is an algebra of non-polynomial growth.

Similarly, consider the quotient algebra $A^{(1)}$ of $D(\lambda)^{(1)}$ 
given by its Gabriel quiver $Q^{(1)}$ 
and the relations:
\begin{align*}
 \alpha \beta \alpha &= 0 ,
 &
 \beta \alpha \beta &= 0 ,
 &
 \gamma \alpha \beta \gamma \alpha &= 0 ,
 &
 \alpha \beta \gamma \alpha \beta &= 0 ,
 &
 \beta \gamma \alpha \beta \gamma &= 0 .
\end{align*}
Then $A^{(1)}$ admits the Galois covering
\[
  F^{(1)} : R \to R / G^{(1)} = A^{(1)}
\]
where $R$ is the locally bounded category considered above 
and $G^{(1)}$ is the free abelian group of rank $2$ generated
by the obvious horizontal and vertical shifts such that $G^{(2)}$ 
is a subgroup of $G^{(1)}$ and $G^{(1)}/G^{(2)}$ is the Klein group
$(\bZ/2\bZ) \oplus (\bZ/2\bZ)$. 
Then we conclude as above that $A^{(1)}$ is of non-polynomial growth, 
and consequently $D(\lambda)^{(1)}$ is of
non-polynomial growth.

The algebras $D(\lambda)^{(1)}$ and $D(\lambda)^{(2)}$, 
for $\lambda \in K^*$, are tame algebras 
(see Proposition~\ref{prop:6.5}).
\end{proof}

%\smallskip

Let $(Q,f)$ be a triangulation quiver,
$m_{\bullet} : \cO(g) \to \bN^*$
a weight function,
and
$c_{\bullet} : \cO(g) \to K^*$
a parameter function.
We consider the quotient algebra
\[
   \Gamma(Q,f,m_{\bullet},c_{\bullet}) = KQ/L(Q,f,m_{\bullet},c_{\bullet}),
\]
where $L(Q,f,m_{\bullet},c_{\bullet})$
is the ideal in the path algebra $K Q$ of $Q$ over $K$
generated by the elements $\alpha f(\alpha)$ and $A_{\alpha}$,
for all arrows $\alpha \in Q_1$.
Then $\Gamma(Q,f,m_{\bullet},c_{\bullet})$ is a string algebra,
which we call the string algebra of the 
weighted triangulation algebra
$\Lambda(Q,f,m_{\bullet},c_{\bullet})$.
We note that it is the largest string
quotient algebra of
$\Lambda(Q,f,m_{\bullet},c_{\bullet})$,
with respect to dimension,
and the Gabriel quiver of
$\Gamma(Q,f,m_{\bullet},c_{\bullet})$ 
is obtained from $Q$ by removing all virtual arrows.
Observe also that
$\Gamma(Q,f,m_{\bullet},c_{\bullet})$
is a quotient algebra of the biserial
weighted triangulation algebra
$B(Q,f,m_{\bullet},c_{\bullet})$.

\begin{theorem}
\label{th:6.8}
Let
$\Lambda = \Lambda(Q,f,m_{\bullet},c_{\bullet})$
be a weighted triangulation algebra
which is not isomorphic to one of the algebras
$D(\lambda)$,
$D(\lambda)^{(1)}$,
$D(\lambda)^{(2)}$,
$T(\lambda)$,
$S(\lambda)$,
$\Lambda(\lambda)$,
for $\lambda \in K^*$.
Then 
$\Gamma = \Gamma(Q,f,m_{\bullet},c_{\bullet})$
is of non-polynomial growth.
In particular, $\Lambda$
is of non-polynomial growth.
\end{theorem}

\begin{proof}
We have the presentation
$\Gamma = K Q_{\Gamma} / I_{\Gamma}$,
where $Q_{\Gamma}$ is the Gabriel quiver
of $\Gamma$ and
$I_{\Gamma} = L(Q,f,m_{\bullet},c_{\bullet}) \cap K Q_{\Gamma}$.
Observe that $I_{\Gamma}$ is the ideal 
in the path algebra $K Q_{\Gamma}$
generated by paths $\alpha f(\alpha)$ and $A_{\alpha}$
for all non-virtual arrows $\alpha$ in $Q_1$.
By general theory, in order to prove that $\Gamma$
is of non-polynomial growth, it is sufficient to indicate 
two primitive walks $v$ and $w$ of the bound quiver
$(Q_{\Gamma},I_{\Gamma})$
such that $v w$ and $w v$ are also primitive
walks (see the proof of \cite[Lemma~1]{Sk0}.
We consider several cases.

\smallskip

(1)
Assume $m_{\alpha} n_{\alpha}  \geq 3$
for all $\alpha \in Q_1$.
If $|Q_0| \geq 3$,
then the required primitive walks are constructed
in the proof of \cite[Proposition~10.2]{ESk-WSA}.
If $|Q_0| = 2$, then
$(Q,f)$ is of the form
\[
   \xymatrix{ 
      1 \ar@(dl,ul)[]^{\alpha} \ar@/^1.5ex/[r]^{\beta} 
     & 2 \ar@/^1.5ex/[l]^{\gamma} \ar@(ur,dr)[]^{\sigma}
   }
\]
with $f$-orbits
$(\alpha\ \beta\ \gamma)$
and
$(\sigma)$.
Since $\Lambda$ is not isomorphic
to a disc algebra $D(\lambda)$, 
we have 
$m_{\alpha} \geq 4$
or
$m_{\beta} \geq 2$.
If $m_{\alpha} \geq 4$,
we may take
$v = \alpha \gamma^{-1} \sigma \beta^{-1}$
and
$w = \alpha^2 \gamma^{-1} \sigma \beta^{-1}$.
For $m_{\beta} \geq 2$,
we may take
$v = \alpha \gamma^{-1} \sigma \beta^{-1}$
and
$w = \alpha \gamma^{-1} \sigma \beta^{-1} \gamma^{-1} \sigma \beta^{-1}$.

We may then assume that $|Q_0| \geq 3$.

\smallskip

(2)
Assume now that there is a virtual loop $\alpha$ in $Q_1$.
Then $(Q,f)$ contains a subquiver of the form
\[
   \xymatrix{ 1 \ar@(dl,ul)[]^{\alpha} \ar@/^1.5ex/[r]^{\beta} & 2 \ar@/^1.5ex/[l]^{\gamma}}
\]
with $f$-orbit
$(\alpha\ \beta\ \gamma)$,
$\beta = g(\gamma)$,
and 
$g(\beta) \neq g^{-1}(\gamma)$.
In particular, we have $n_{\beta} \geq 4$.
In the special case
$m_{\beta} n_{\beta}  =4$,
$(Q,f)$ is the quiver
\[
  \xymatrix{
%  \xymatrix@C=1pc{
    1
%    \ar `ld_u[] `_rd[]^{\alpha} []
    \ar@(ld,ul)^{\alpha}[]
    \ar@<.5ex>[r]^{\beta}
    & 2
    \ar@<.5ex>[l]^{\gamma}
%    \ar `ru_d[] `_lu[]^{\eta} [] &
    \ar@<.5ex>[r]^{\sigma}
    & 3
    \ar@<.5ex>[l]^{\delta}
%    \ar `ru_d[] `_lu[]^{\xi} [] &
    \ar@(ru,dr)^{\eta}[]
  } 
%,
\]
with $f$-orbits
$(\alpha\ \beta\ \gamma)$
and
$(\eta\ \delta\ \sigma)$,
and  $g$-orbits
$(\alpha)$, 
$(\beta\ \sigma\ \delta\ \gamma)$,  
$(\eta)$.
Since $\Lambda$ is not isomorphic
to a triangle algebra $T(\lambda)$, 
we conclude that $\eta$ is not virtual,
and hence 
$m_{\eta} \geq 4$.
Then we may take
$v = \delta \beta^{-1} \gamma^{-1} \sigma$
and
$w = \delta \beta^{-1} \gamma^{-1} \sigma \eta^{-1}$.
We also note that, if 
$(Q,f)$ is the above quiver and
$m_{\beta} \geq 2$,
then we may take
$v = \delta \beta^{-1} \gamma^{-1} \sigma$
and
$w = \delta \beta^{-1} \gamma^{-1} \sigma \delta \sigma$.
Hence we may
assume that $|Q_0| \geq 4$.
Clearly, then $n_{\beta} \geq 5$,
and hence 
$m_{\beta} n_{\beta} \geq 5$.
Let $\sigma = g(\beta)$,
$\delta = g^{-1}(\gamma)$,
and $\xi = f(\sigma)$.
Assume $\xi$ is a virtual arrow.
Then $(Q,f)$ admits a subquiver of the form
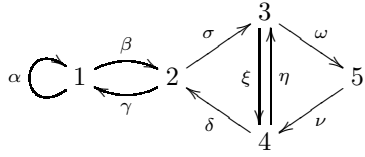
\[
   \xymatrix@R=1pc{ 
      && 3 \ar@<-.5ex>[dd]_{\xi} \ar[rd]^{\omega} \\
      1 \ar@(dl,ul)[]^{\alpha} \ar@/^1.5ex/[r]^{\beta} & 
      2 \ar@/^1.5ex/[l]^{\gamma} \ar[ur]^{\sigma} &&
      5 \ar[dl]^{\nu} \\     
      && 4 \ar@<-.5ex>[uu]_{\eta} \ar[lu]^{\delta}
   }
\]
with  $f$-orbits
$(\alpha\ \beta\ \gamma)$,
$(\sigma\ \xi\ \delta)$,
$(\omega\ \nu\ \eta)$.
Then we conclude that
$n_{\beta} \geq 7$.
Let $u$ be the path 
$g^3(\beta) g^4(\beta) \dots g^{n_{\beta}-4}(\beta)
 = g(\omega) \dots g^{-1}(\nu)$
of length $\geq 1$ from $5$ to $5$.
Observe that
$\gamma \beta$,
$\sigma \omega$,
$\nu \delta$
and $u$ are non-zero paths in $\Gamma$.
Then we may take the primitive walks
$v = \nu \delta \beta^{-1} \gamma^{-1} \sigma \omega$
and
$w =\nu \delta \beta^{-1} \gamma^{-1} \sigma \omega u^{-1}$.
Finally, assume that the arrow $\xi$ is not virtual,
and so $\xi$ is an arrow of $Q_{\Gamma}$.
Let $p$ be the path 
$g^2(\beta) g^3(\beta) \dots g^{n_{\beta}-3}(\beta)
 = g(\sigma) \dots g^{-1}(\delta)$
of length $\geq 1$.
We note that $\sigma p \delta$ is a non-zero path
of $\Gamma$ because it is of length $n_{\beta} - 2$.
We may take the required primitive walks as follows
$v = \delta \beta^{-1} \gamma^{-1} \sigma p$
and
$w =\delta \beta^{-1} \gamma^{-1} \sigma  p \xi^{-1} p$.

\smallskip

(3)
Assume now that 
there is a pair $\xi, \eta$ of virtual arrows  in $Q_1$.
Then $(Q,f)$ contains a subquiver of the form
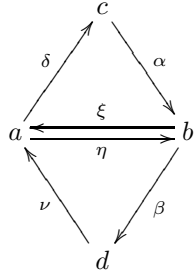
\[
%  \xymatrix@R=2pc@C=1.5pc{
%  \xymatrix@R=3.5pc@C=1.8pc{
  \xymatrix@R=3.pc@C=1.8pc{
%  \xymatrix@C=.8pc{
    & c
    \ar[rd]^{\alpha}
    \\   
    a
    \ar[ru]^{\delta}
    \ar@<-.5ex>[rr]_{\eta}
    && b
    \ar@<-.5ex>[ll]_{\xi}
    \ar[ld]^{\beta}
    \\
    & d
    \ar[lu]^{\nu}
  }
\]
with  $f$-orbits
$(\alpha\ \xi\ \delta)$
and
$(\beta\ \nu\ \eta)$.
Consider first the case
when $c = d$, so 
$(Q,f)$ is of the form
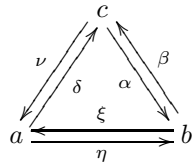
\[
  \xymatrix@R=3.pc@C=1.8pc{
    & c
    \ar@<-.5ex>[rd]_{\alpha}
    \ar@<-.5ex>[ld]_{\nu}
    \\   
    a
    \ar@<-.5ex>[ru]_{\delta}
    \ar@<-.5ex>[rr]_{\eta}
    && b
    \ar@<-.5ex>[ll]_{\xi}
    \ar@<-.5ex>[lu]_{\beta}
  }
\]
with  $g$-orbits
$(\xi\ \eta)$,
$(\alpha\ \beta)$,
$(\nu\ \delta)$.
It follows from our general assumption
that $\alpha$, $\beta$, $\gamma$, $\delta$
are not virtual arrows,
and hence 
$m_{\alpha} n_{\alpha} \geq 4$
and
$m_{\delta} n_{\delta} \geq 4$.
Moreover, because $\Lambda$ is not isomorphic
to a triangle algebra $T(\lambda)$,
we may assume that
$m_{\alpha} n_{\alpha}  \geq 6$.
Then
$v = \alpha \beta \delta^{-1} \nu^{-1}$
and
$w = \alpha \beta  \alpha \beta \delta^{-1} \nu^{-1}$
is a required pair of primitive walks.

Assume now that $c \neq d$.
Then
$n_{\alpha} \geq 3$
and
$n_{\delta} \geq 3$.
We note that if 
$\cO(\alpha) = \cO(\delta)$,
then 
$n_{\alpha} = n_{\delta} \geq 6$.
Consider first the case when one of 
$m_{\alpha} n_{\alpha}$
or
$m_{\delta} n_{\delta}$,
say $m_{\alpha} n_{\alpha}$,
is equal to $3$.
Then $(Q,f)$ contains a subquiver of the form
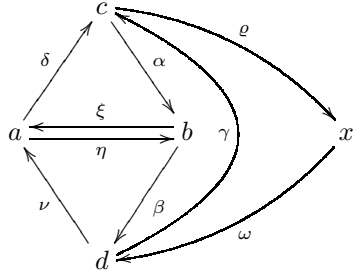
\[
%  \xymatrix@R=2pc@C=1.5pc{
%  \xymatrix@R=3.5pc@C=1.8pc{
  \xymatrix@R=3.pc@C=1.8pc{
%  \xymatrix@C=.8pc{
    & c
    \ar[rd]^{\alpha}
    \ar@/^2.5ex/[rrrd]^{\varrho}
    \\   
    a
    \ar[ru]^{\delta}
    \ar@<-.5ex>[rr]_{\eta}
    && b
    \ar@<-.5ex>[ll]_{\xi}
    \ar[ld]^{\beta}
    && x \ar@/^2.5ex/[llld]^{\omega}
    \\
    & d
    \ar[lu]^{\nu}
    \ar@/_12ex/[uu]^{\gamma}
  }
\]
with  $f$-orbits
$(\alpha\ \xi\ \delta)$,
$(\beta\ \nu\ \eta)$,
$(\gamma\ \varrho\ \omega)$,
and $n_{\delta} \geq 5$.
We denote by $u$ the path 
$g^2(\delta) \dots g^{n_{\delta}-3}(\delta)$
of length $\geq 1$ from $x$ to $x$.
Observe that 
$\nu \delta \varrho$,
$\omega \nu \delta$
and $u$
are
non-zero paths in $\Gamma$.
Then we may choose the pair of primitive walks
$v = \nu \delta \gamma^{-1}$
and
$w =\nu \delta \varrho p^{-1} \omega$.
We consider now the case when
$m_{\alpha} n_{\alpha} = 4 = m_{\delta} n_{\delta}$.
Then 
$n_{\alpha} = 4 = n_{\delta}$,
$m_{\alpha} = 1 = m_{\delta}$,
and $(Q,f)$ 
is the quiver of the form
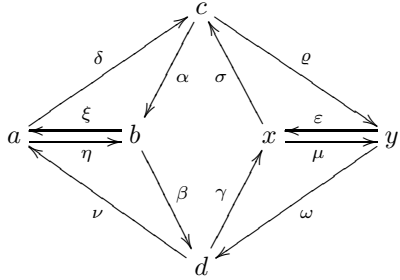
\[
%  \xymatrix@R=2pc@C=1.5pc{
%  \xymatrix@R=3.5pc@C=1.8pc{
  \xymatrix@R=3.pc@C=1.2pc{
%  \xymatrix@C=.8pc{
    &&& c
    \ar[ld]^{\alpha}
    \ar[rrrd]^{\varrho}
    \\   
    a
    \ar[rrru]^{\delta}
    \ar@<-.5ex>[rr]_(.6){\eta}
    && b
    \ar@<-.5ex>[ll]_(.4){\xi}
    \ar[rd]^{\beta}
    && x
    \ar[lu]^{\sigma}
    \ar@<-.5ex>[rr]_(.4){\mu}
    && y
    \ar@<-.5ex>[ll]_(.6){\varepsilon}
    \ar[llld]^{\omega}
%    &&&& \bullet
%    \ar[llld]^{\nu}
    \\
    &&& d
    \ar[lllu]^{\nu}
    \ar[ur]^{\gamma}
  }
\]
with  $f$-orbits
$(\alpha\ \xi\ \delta)$,
$(\beta\ \nu\ \eta)$,
$(\gamma\ \mu\ \omega)$,
$(\sigma\ \varrho\ \varepsilon)$,
and hence $g$-orbits
$(\xi\ \eta)$,
$(\alpha\ \beta\ \gamma\ \sigma)$,
$(\delta\ \varrho\ \omega\ \nu)$,
$(\varepsilon\ \mu)$.
Since $\Lambda$ is not isomorphic
to a spherical algebra 
$S(\lambda)$, 
we have $m_{\eta} n_{\eta} \geq 4$.
We note that 
$\nu \delta$
and
$\gamma \sigma$
are
non-zero paths in $\Gamma$.
Then we may take a required pair of primitive walks as follows:
$v = \gamma^{-1} \nu \delta \sigma^{-1}$
and
$w = \mu \varepsilon \gamma^{-1} \nu \delta \sigma^{-1}$ .

Assume now that 
$n_{\alpha} \geq 4$
and
$n_{\delta} \geq 5$.
We have two cases to consider .
Let
$\cO(\alpha) \neq \cO(\delta)$.
We denote
\begin{align*}
 \gamma &= g(\beta) , &
 \sigma &= g^{-1}(\alpha) , &
 \mu &= f(\gamma) , &
 \omega &= f(\mu) , &
 \varrho &= f(\sigma) , &
 \varepsilon &= f(\varrho) .
\end{align*}
Then $(Q,f)$ contains a subquiver 
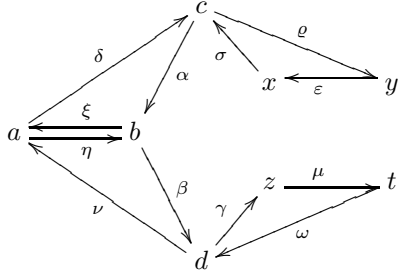
\[
%  \xymatrix@R=3.pc@C=1.2pc{
  \xymatrix@R=.5pc@C=1.2pc{
    &&& c
    \ar[lddd]^{\alpha}
    \ar[rrrdd]^{\varrho}
    \\ \\   
    &&&& x
    \ar[luu]^{\sigma}
    && y
    \ar@<-.5ex>[ll]^(.6){\varepsilon}
    \\   
    a
    \ar[rrruuu]^{\delta}
    \ar@<-.5ex>[rr]_(.6){\eta}
    && b
    \ar@<-.5ex>[ll]_(.4){\xi}
    \ar[rddd]^{\beta}
    \\
    &&&& z
    \ar@<-.5ex>[rr]^(.4){\mu}
    && t
    \ar[llldd]^{\omega}
    \\ \\   
    &&& d
    \ar[llluuu]^{\nu}
    \ar[uur]^{\gamma}
  }
\]
where $y \neq t$, and possibly $x = z$.
Moreover, we have the path
$q = g^2(\delta) \dots g^{n_{\delta}-3}(\delta)$
of length $\geq 1$ from $y$ to $t$,
and the subpath $p$ of
$\alpha g(\alpha) \dots g^{n_{\alpha}-1}(\alpha)$
of length $\geq 0$ from $z$ to $x$.
Then we may choose a required pair of primitive walks
as follows:
$v = \sigma \delta^{-1} \nu^{-1} \gamma p$
and
$w =  \sigma \delta^{-1} \nu^{-1} \gamma p \varepsilon^{-1} q \mu^{-1} p$.

Finally, assume that
$\cO(\alpha) = \cO(\delta)$.
Consider first the case
$n_{\alpha} = 6$.
Then $(Q,f)$ is of the form
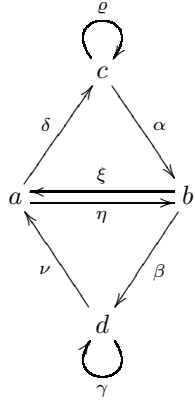
\[
%  \xymatrix@R=2pc@C=1.5pc{
%  \xymatrix@R=3.5pc@C=1.8pc{
  \xymatrix@R=3.pc@C=1.8pc{
%  \xymatrix@C=.8pc{
    & c 
    \ar@(lu,ur)^{\varrho}[]
    \ar[rd]^{\alpha}
    \\   
    a
    \ar[ru]^{\delta}
    \ar@<-.5ex>[rr]_{\eta}
    && b
    \ar@<-.5ex>[ll]_{\xi}
    \ar[ld]^{\beta}
    \\
    & d
    \ar[lu]^{\nu}
    \ar@(rd,ld)^{\gamma}[]
  }
\]
with
$f(\varrho) = \varrho$,
$f(\gamma) = \gamma$,
and
$\cO(\alpha) = (\alpha\ \beta\ \gamma\ \nu\ \delta\ \varrho)$.
Since $\Lambda$ is not isomorphic
to an algebra 
$D(\lambda)^{(2)}$, 
with $\lambda \in K^*$, we have also
$m_{\alpha} \geq 2$.
Then we may take a required pair of primitive walks as follows:
$v = \alpha \beta \gamma^{-1} \nu \delta \varrho^{-1}$
and
$w = \alpha \beta \gamma^{-1} \nu \delta \varrho^{-1} 
          \alpha \beta \gamma \nu \delta \varrho$.
Assume now that 
$n_{\alpha} \geq 7$.
Then  one of the arrows
$\gamma = g(\beta)$ or $\varrho = g(\delta)$,
say $\gamma$, 
is not fixed by $f$.
Hence we have a triangle
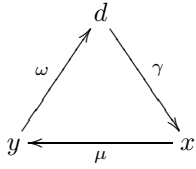
\[
  \xymatrix@R=3.pc@C=1.8pc{
    & d
    \ar[rd]^{\gamma}
    \\   
    y
    \ar[ru]^{\omega}
    && x
    \ar[ll]^{\mu}
  }
\]
with
$\mu = f(\gamma)$
and
$\omega = f(\mu)$.
Consider the subpaths of
$\alpha g(\alpha) \dots g^{n_{\alpha}-1}(\alpha)$: 
$p = g(\delta) \dots g^{n_{\alpha}-1}(\alpha)$
and
$q = g(\gamma) \dots g^{-1}(\omega)$.
Then we choose a required pair of primitive walks
as follows: 
$v = \alpha \beta \omega^{-1} q^{-1} \gamma^{-1} \nu \delta p^{-1}$
and
$w = \alpha \beta \gamma q \mu^{-1} q \omega \nu \delta p^{-1}$.
\end{proof}

We have also the following consequence of 
Propositions \ref{prop:3.8} and \ref{prop:3.9},
Lemma~\ref{lem:6.7},
and Theorem~\ref{th:6.8}.

\begin{theorem}
\label{th:6.9}
Let
$\Lambda = \Lambda(Q,f,m_{\bullet},c_{\bullet})$
be a weighted triangulation algebra.
Then the following statements are equivalent:
\begin{enumerate}[(i)]
 \item
  $\Lambda$ is of polynomial growth.
 \item
  $\Lambda$ is isomorphic to a non-singular disc,
  triangle, tertahedral, or spherical algebra. 
\end{enumerate}
\end{theorem}

\section*{Acknowledgements}

Both authors thank the program "Research in Pairs" of MFO Oberwolfach, and 
as well the Faculty of Mathematics 
and Computer Science of the 
Nicolaus Copernicus University in Toru\'{n} for support. 

%\section*{References}

\end{document}